\documentclass[reqno,12pt]{amsart}
\usepackage{a4wide}
\usepackage{amssymb}
\usepackage{amsmath}
\usepackage{amsthm}
\usepackage[usenames]{color}

\newcommand{\R}{\mathbb{R}}

\newtheorem{theorem}{Theorem}[section]

\newtheorem{lemma}[theorem]{Lemma}

\newtheorem{remark}[theorem]{Remark}
\newtheorem{definition}[theorem]{Definition}

\numberwithin{equation}{section}
\usepackage[colorlinks]{hyperref}
\usepackage{cite}

\usepackage{marginnote}

\newcommand{\Fg}{\mathbf{g}}

\usepackage{color}

\begin{document} 

\title[Viscous contact wave of the Landau equation]{Asymptotics toward viscous contact waves for solutions of the Landau equation}

\author[R.-J. Duan]{Renjun Duan}
\address[R.-J. Duan]{Department of Mathematics, The Chinese University of Hong Kong, Shatin, Hong Kong,
        People's Republic of China}
\email{rjduan@math.cuhk.edu.hk}

\author[D.-C. Yang]{Dongcheng Yang}
\address[D.-C. Yang]{Department of Mathematics, The Chinese University of Hong Kong, Shatin, Hong Kong,
        People's Republic of China}
\email{dcmath@sina.com}

\author[H.-J. Yu]{Hongjun Yu} \address[H.-J. Yu]{School of Mathematical Sciences, South China Normal University, Guangzhou 510631, People's Republic of China}
\email{yuhj2002@sina.com}

\begin{abstract}
In the paper, we are concerned with the large time asymptotics toward the viscous contact waves for solutions of the Landau equation with physically realistic Coulomb interactions. Precisely, for the corresponding Cauchy problem in the spatially one-dimensional setting, we construct the unique global-in-time solution near a local Maxwellian whose fluid quantities are the viscous contact waves in the sense of hydrodynamics and also prove that the solution tends toward such local Maxwellian in large time. The result is proved by elaborate energy estimates and seems the first one on the dynamical stability of contact waves for the Landau equation. One key point of the proof is to introduce a new  time-velocity weight function that includes an exponential factor of the form $\exp (q(t)\langle \xi\rangle^2)$ with
$$
q(t):=q_1-q_2\int_0^tq_3(s)\,ds,
$$
where $q_1$ and $q_2$ are given positive constants and $q_3(\cdot)$ is defined by the energy dissipation rate of the solution itself. The time derivative of such weight function is able to induce an extra quartic dissipation term for treating the large-velocity growth in the nonlinear estimates due to  degeneration of the linearized Landau operator in the Coulomb case. Note that in our problem the explicit time-decay of solutions around contact waves is unavailable but no longer needed under the crucial use of the above weight function, which is different from the situation in \cite{Duan,Duan-Yu1}.

\end{abstract}


\maketitle

\tableofcontents

\thispagestyle{empty}

\section{Introduction}

\subsection{Formulation of the problem}
In the paper, we are concerned with the Landau equation in the spatially one-dimensional setting
\begin{equation}
\label{1.1}
\partial_{t}F+\xi_{1}\partial_{x}F=Q(F,F).
\end{equation}
Here, the unknown $F=F(t,x,\xi)\geq0$ stands for the density
distribution function of particles with position
$x\in\mathbb{R}$ and velocity $\xi=(\xi_{1},\xi_{2},\xi_{3})\in\mathbb{R}^{3}$ at time $t>0$.  The Landau collision operator $Q(\cdot,\cdot)$ on the right hand side of \eqref{1.1} is a bilinear integro-differential operator acting only on velocity variables of the form
\begin{equation*}
Q(F_{1},F_{2})(\xi)=\nabla_{\xi}\cdot\int_{\mathbb{R}^{3}}\phi(\xi-\xi_{*})\left\{F_{1}(\xi_{*})\nabla_{\xi}F_{2}(\xi)-\nabla_{\xi_{*}}F_{1}(\xi_{*})F_{2}(\xi)
\right\}\,d\xi_{*},
\end{equation*}
where for the Landau collision kernel $\phi(z)$ with $z=\xi-\xi_\ast$ (cf. \cite{Guo-2002,H}), we consider only the case of physically realistic Coulomb interactions through the paper, namely,
\begin{equation}\label{def.lk}
\phi_{ij}(z)=\frac{1}{|z|}\big(\delta_{ij}-\frac{z_iz_j}{|z|^{2}}\big),\quad 1\leq i,j\leq 3.
\end{equation}
To solve \eqref{1.1} we supplement it with initial data
\begin{equation}
\label{ide}
F(0,x,\xi)=F_{0}(x,\xi)
\end{equation}
that connects two distinct global Maxwellians at the far fields $x=\pm\infty$ in the way that
\begin{align}
\label{1.2}
F_{0}(x,\xi)\to
\frac{\rho_\pm}{{(2\pi R\theta_\pm)^{\frac{3}{2}}}}\exp\Big(-\frac{|\xi-u_\pm|^2}{2R\theta_\pm}\Big)\ \ \mbox{as}\ \ x\to\pm\infty,
\end{align}
where $\rho_\pm>0$, $\theta_\pm>0$ and $u_{1\pm}$ with $u_{\pm}=(u_{1\pm},0,0)$ are constants and $R$ is the gas constant.

We are interested in studying the global existence and large-time behavior of solutions to the Cauchy problem \eqref{1.1}, \eqref{ide} and \eqref{1.2} in case of
\begin{equation}
\label{cond.fd}
\rho_+\neq \rho_-, \ u_+=u_-,\  \rho_+\theta_+=\rho_-\theta_-.
\end{equation}
Indeed, the qualitative behavior of solutions is closely related to that of the corresponding fluid dynamic system, for instance, the compressible Euler and Navier-Stokes equations at the zero-order and first-order levels, respectively, cf.~\cite{SR,Ukai-Yang}. Under such condition \eqref{cond.fd}, the asymptotic state of solutions to those fluid equations would be a contact wave introduced in the context of conservation laws, cf.~\cite{Liu-Xin,Smoller}. Our goal in this paper is to construct a global solution of \eqref{1.1}, \eqref{ide} and \eqref{1.2} with the condition  \eqref{cond.fd} that tends in large time toward the viscous contact wave in the sense of Definition \ref{def.vcw} to be specified later. We remark that the viscous contact wave was first introduced in \cite{HMS} as the asymptotic wave pattern for the compressible Navier-Stokes equations and later in \cite{Huang-Yang} for the Boltzmann equation with hard sphere collisions. It has been a challenging open problem to extend the stability results on contact waves in  \cite{Huang-Yang} and \cite{Huang-Xin-Yang} to the case of the Boltzmann equation with physically realistic long-range interactions or even the Landau equation with grazing collisions. The relevant literature will be further reviewed later on.

For the above purpose, motivated by \cite{Huang-Xin-Yang}, we have to turn to the Landau equation in the Lagrangian coordinates.  As in \cite{Liu-Yu} and \cite{Liu-Yang-Yu}, associated with a solution $F(t,x,\xi)$ to the Landau equation \eqref{1.1}, we define five macroscopic (fluid) quantities:
the mass density $\rho(t,x)$, momentum $\rho(t,x)u(t,x)$, and
energy density $e(t,x)+\frac 12|u(t,x)|^2$:
\begin{equation}
\label{1.3}
  \left\{
\begin{array}{rl}
&\rho(t,x)\equiv \int_{\mathbb{R}^3}\psi_{0}(\xi)F(t,x,\xi)\,d\xi,\\
&\rho(t,x)u_i(t,x)\equiv\int_{\mathbb{R}^3}\psi_{i}(\xi)F(t,x,\xi)\,d\xi,\ \
\mbox{for} \ \ i=1,2,3,\\
 &\rho(t,x)\left[e(t,x)+\frac
12|u(t,x)|^2)\right]\equiv\int_{\mathbb{R}^3}\psi_{4}(\xi)F(t,x,\xi)\,d\xi, \  \ \
\ \
\end{array} \right.
\end{equation}
and the corresponding local Maxwellian $M$:
\begin{equation}
\label{1.5}
M=M_{[\rho,u,\theta]}(t,x,\xi):=\frac{\rho(t,x)}{(2\pi R\theta(t,x))^{\frac{3}{2}}}\exp\Big(-\frac{|\xi-u(t,x)|^{2}}{2R\theta(t,x)}\Big).
\end{equation}
Here $\rho=\rho(t,x)>0$ is the mass density, $u=u(t,x)=(u_{1},u_{2},u_{3})$ is the bulk velocity, $e=e(t,x)>0$
is the internal energy depending on the temperature $\theta$ by $e=\frac{3}{2}R\theta=\theta$ with $R=2/3$
taken for convenience,
and $\psi_{i}(\xi)$ $(i=0,1,2,3,4)$ are
five collision invariants given by
\begin{equation*}
\psi_{0}(\xi)=1, \quad \psi_{i}(\xi)=\xi_{i}~(i=1,2,3),\quad \psi_{4}(\xi)=\frac{1}{2}|\xi|^{2},
\end{equation*}
satisfying
\begin{equation}
\label{1.4}
\int_{\mathbb{R}^{3}}\psi_{i}(\xi)Q(F,F)\,d\xi=0,\quad \mbox{for $i=0,1,2,3,4$.}
\end{equation}
Furthermore, in terms of $[\rho,\rho u_1](t,x)$, we introduce the coordinate transformation
\begin{equation}\label{def.lct}
(t,x)\rightarrow \Big(t,\int^{(t,x)}_{(0,0)}\rho(\tau,y)\,dy-(\rho u_{1})(\tau,y)\,d\tau\Big),
\end{equation}
where $\int^{B}_{A}g\,dy+h\,d\tau$ represents a line integration from point $A$ to point $B$ on
the half-plane $\mathbb{R}^{+}\times\mathbb{R}$. Still using the $(t,x)$ variables for simplicity of notations, we then rewrite
the Landau equation \eqref{1.1} as the one in the Lagrangian coordinates
\begin{equation}
\label{1.16}
\partial_{t}F+\frac{\xi_1-u_1}{v}\partial_x F
=Q(F,F),
\end{equation}
with initial data
\begin{align}
\label{1.17}
F(0,x,\xi)=F_{0}(x,\xi)\to M_{[1/v_{\pm},u_{\pm},\theta_{\pm}]}(\xi):=\frac{1/v_\pm}{{(2\pi R\theta_\pm)^{\frac{3}{2}}}}\exp\Big(-\frac{|\xi-u_\pm|^2}{2R\theta_\pm}\Big)\ \ \mbox{as}\ \ x\to\pm\infty,
\end{align}
where
\begin{equation}
\label{def.vrho}
v=v(t,x):=\frac{1}{\rho(t,x)}
\end{equation}
denotes the specific volume of the particles and 
$v_{\pm}=1/\rho_{\pm}>0$.

\begin{definition}\label{def.vcw}
Given the far-field data $[v_\pm,u_\pm,\theta_\pm]$ satisfying that $v_\pm>0$, $\theta_\pm>0$ and
\begin{equation*}
u_{-}=u_{+},\quad  p_{-}:=\frac{R\theta_{-}}{v_{-}}=\frac{R\theta_{+}}{v_{+}}=:p_{+},\quad v_{-}\neq v_{+},
\end{equation*}
a viscous contact wave corresponding to the Cauchy problem \eqref{1.16} and \eqref{1.17} on the Landau equation is defined to be a local Maxwellian
\begin{equation}
\label{1.36}
 \overline{M}= M_{[1/\overline{v},\overline{u},\overline{\theta}]}(t,x,\xi):=\frac{1/\overline{v}(t,x)}{\big(2\pi
R\overline{\theta}(t,x)\big)^{\frac{3}{2}}}\exp\Big(-\frac{|\xi-\overline{u}(t,x)|^2}{2R\overline{\theta}(t,x)}\Big),
\end{equation}
where the macroscopic variables $[\bar{v},\bar{u},\bar{\theta}](t,x)$ are the viscous contact wave constructed in \eqref{1.28}  in the sense of hydrodynamic equations of \eqref{1.16}, satisfying the same far-field condition as in \eqref{1.17}, i.e.,
$$
M_{[1/\overline{v},\overline{u},\overline{\theta}]}(t,x,\xi)\to M_{[1/v_{\pm},u_{\pm},\theta_{\pm}]}(\xi)\ \ \mbox{as}\ \ x\to\pm\infty,
$$
see Subsection \ref{subsec.vcw} later on for details.
\end{definition}

\begin{remark}
For convenience of the proof regarding the dynamical stability of the viscous contact wave \eqref{1.36}, throughout the paper we fix a normalized global Maxwellian with the fluid constant state $(1,0,3/2)$
\begin{equation}
\label{1.38}
\mu=M_{[1,0,\frac{3}{2}]}(\xi):=(2\pi)^{-\frac{3}{2}}\exp\big(-\frac{|\xi|^{2}}{2}\big)
\end{equation}
as a reference equilibrium state, and choose both the far-field data $[v_+,u_+,\theta_+]$ and $[v_-,u_-,\theta_-]$
in \eqref{1.17} to be close enough to the constant state $(1,0,3/2)$ such that the viscous contact wave further satisfies that
\begin{equation}
\label{1.37}
\left\{\begin{aligned}
 &\sup_{t\geq 0,\,x\in \R}\{|\bar{v}(t,x)-1|+|\bar{u}(t,x)|+|\bar{\theta}(t,x)-\frac{3}{2}|\}<\eta_{0},\\
& \frac{1}{2}\sup_{t\geq0,\,x\in\mathbb{R}}\bar{\theta}(t,x)<\frac{3}{2}<\inf_{t\geq0,\,x\in\mathbb{R}}\bar{\theta}(t,x),
\end{aligned}\right.
\end{equation}
for some constant $\eta_{0}>0$  enough small.
\end{remark}

\subsection{Weight, norm and main result}
To present the main result, we first introduce the appropriate perturbations, weight functions and norms.
To the end, for any solution $F(t,x,\xi)$ to  the Cauchy problem \eqref{1.16} and \eqref{1.17}, we define the macroscopic perturbation $[\widetilde{v},\widetilde{u},\widetilde{\theta}]$ and the microscopic perturbation $\mathbf{g} $  by
\begin{equation}
\left\{
\begin{array}{rl}
& \widetilde{v}(t,x)={v}(t,x)-\overline{v}(t,x),\\
& \widetilde{u}(t,x)={u}(t,x)-\overline{u}(t,x),\\
& \widetilde{\theta}(t,x)={\theta}(t,x)-\overline{\theta}(t,x),\\
 &  \mathbf{g}(t,x,\xi)=\frac{\widetilde{G}(t,x,\xi)}{\sqrt{\mu}}=\frac{G(t,x,\xi)-\overline{G}(t,x,\xi)}{\sqrt{\mu}},\ \ G=F-M_{[1/v,u,\theta]},
 \end{array} \right.
\label{2.1}
\end{equation}
where $[v,u,\theta](t,x)$ is given by \eqref{1.3} and \eqref{def.vrho},   $[\bar{v},\bar{u},\bar{\theta}](t,x)$ is the fluid viscous contact wave, and the term $\overline{G}(t,x,\xi)$ is defined as
\begin{equation}
\label{2.3}
\overline{G}(t,x,\xi)=\frac{1}{v}L_M^{-1}P_1\xi_1M\Big\{\frac{|\xi-u|^2\overline{\theta}_x}{2R\theta^2}
+\frac{(\xi-u)\cdot \overline{u}_{x}}{R\theta}\Big\}
\end{equation}
with the microscopic projection $P_1$ given in \eqref{1.20a}.
In order to prove the stability of the local Maxwellian $M_{[1/\overline{v},\overline{u},\overline{\theta}]}$ defined in \eqref{1.36}, the most
key step is to establish uniform energy estimates on $[\widetilde{v},\widetilde{u},\widetilde{\theta}]$ and $\mathbf{g} $. Moreover, inspired by \cite{Guo-2002} and \cite{Duan}, a crucial point in the proof is to introduce the weight function in the following


\begin{definition}[Time-velocity weight]
We define
\begin{equation}
\label{1.32}
w(\beta)(t,\xi):= \langle \xi\rangle^{(l-|\beta|)}e^{q(t)\langle \xi\rangle^{2}},\ \ \  l\geq |\beta|,
\quad \langle \xi\rangle=\sqrt{1+|\xi|^{2}},
\end{equation}
with 
\begin{equation}
\label{1.36a}
q(t):=q_{1}-q_{2}\int^{t}_{0}q_{3}(s)\,ds>0,\quad \forall\,t\geq 0,
\end{equation}
where the strictly positive constants $q_{1}>0$ and $q_{2}>0$ will be chosen in the proof later, see also Theorem \ref{thm1.1},  and the function $q_{3}(t)$, depending on both the viscous contact wave and the macroscopic perturbation,  is given by
\begin{align}
\label{1.33}
q_{3}(t)&:=\|\bar{v}_{x}(t)\|_{L_{x}^{\infty}}^{3}+\|\bar{v}_{t}(t)\|^{2}+\sum_{|\alpha|=2}\|\partial^{\alpha}\bar{v}(t)\|^{2}
+\sum_{1\leq|\alpha|\leq2}\|\partial^{\alpha}\bar{u}(t)\|^{2}\notag\\
&\qquad+\sum_{1\leq|\alpha|\leq2}(\|\partial^{\alpha}\widetilde{v}(t)\|^{2}
+\|\partial^{\alpha}\widetilde{u}(t)\|^{2}).
\end{align}
In addition, we require that $q_3(\cdot)$ is integrable in time satisfying
\begin{equation*}
q_\infty:=q_{1}-q_{2}\int^{\infty}_{0}q_{3}(s)\,ds>0,
\end{equation*}
so that  $q(t)$ is a strictly positive continuous function monotonically decreasing from $q_1>0$ at $t=0$ to $q_\infty>0$ as $t\to+\infty$.
\end{definition}

\begin{remark}\label{rem.q12}
It should be emphasized that the non-negative function $q_3(t)$ contains two parts
\begin{eqnarray*}
q_3^I(t)&=&\|\bar{v}_{x}(t)\|_{L_{x}^{\infty}}^{3}+\|\bar{v}_{t}(t)\|^{2}+\sum_{|\alpha|=2}\|\partial^{\alpha}\bar{v}(t)\|^{2}
+\sum_{1\leq|\alpha|\leq2}\|\partial^{\alpha}\bar{u}(t)\|^{2},\\
\quad q_3^{I\!I}(t)&=&\sum_{1\leq|\alpha|\leq2}(\|\partial^{\alpha}\widetilde{v}(t)\|^{2}
+\|\partial^{\alpha}\widetilde{u}(t)\|^{2}),
\end{eqnarray*}
related to the viscous contact wave and the macroscopic perturbation, respectively, and both parts can be verified to be integrable in large time. Indeed, due to the time-decay property \eqref{1.29},  $q_3^I(t)$ has a fast decay as $(1+t)^{-(1+\vartheta)}$ with some constant $\vartheta>0$. Also, $q_3^{I\!I}(t)$ is part of the energy dissipation functional \eqref{2.11} that will be proved to be integrable over $(0,+\infty)$ as given in \eqref{1.40} but does not enjoy any explicit time decay because of the techniques of the proof.
\end{remark}

Corresponding to the reference global Maxwellian $\mu$ in \eqref{1.38}, the Landau collision frequency is
\begin{equation}
\label{1.34}
\sigma^{ij}(\xi):=\phi^{ij}\ast \mu(\xi)=\int_{{\mathbb R}^3}\phi^{ij}(\xi-\xi')\mu(\xi')\,d\xi'.
\end{equation}
We remark that $\sigma^{ij}(\xi)$ is a positive definite symmetric matrix. We denote the
weighted $L^2$ norms as
$$
|w\Fg|^2_{2}:= \int_{{\mathbb R}^3}w^{2}\Fg^2\,d\xi,\ \ \ \|w\Fg\|^2:={ \int_{\mathbb R}\int_{{\mathbb
R}^3}}w^{2}\Fg^2\,d\xi dx.
$$
In terms of linearization of the nonlinear Landau operator around $\mu$ (cf.~\cite{Guo-2002}), with \eqref{1.34} we define the weighted dissipative norms:
$$|\Fg|^2_{\sigma,w}:=\sum_{i,j=1}^3\int_{{\mathbb
R}^3}w^{2}[\sigma^{ij}\partial_{\xi_i}\Fg\partial_{\xi_j}\Fg+\sigma^{ij}\frac{\xi_i}{2}\frac{\xi_j}{2}\Fg^2]\,d\xi,\ \mbox{and}\ \ \|\Fg\|_{\sigma,w}:= \||\Fg|_{\sigma,w} \|.$$
And let $|\Fg|_{\sigma}=|\Fg|_{\sigma,1}$ and $\|\Fg\|_{\sigma}=\|\Fg\|_{\sigma,1}$.
From \cite[Corollary 1, p.399]{Guo-2002} and \cite[Lemma 5, p.315]{Strain-Guo-2008}, one has
\begin{equation}
\label{1.35}
|\Fg|_\sigma\approx |\langle \xi\rangle^{-\frac{1}{2}}\Fg|_2+\Big|\langle \xi\rangle^{-\frac{3}{2}}\frac{\xi}{|\xi|}\cdot\nabla_\xi \Fg\Big|_2+\Big|\langle \xi\rangle^{-\frac{1}{2}}\frac{\xi}{|\xi|}\times \nabla_\xi \Fg\Big|_2.
\end{equation}
We also denote
$$\|\partial^\alpha_\beta
\Fg\|^2_{w(\beta)}:=\int_{\mathbb R}\int_{{\mathbb
R}^3}w^{2}(\beta)|\partial^\alpha_\beta \Fg(x,\xi)|^2\, d\xi dx,$$
and
$$
\|\partial^\alpha_\beta \Fg\|^2_{\sigma,w(\beta)}:=\sum_{i,j=1}^3\int_{\mathbb R}\int_{{\mathbb
R}^3}w^{2}(\beta)[\sigma^{ij}\partial_{\xi_i}\partial^\alpha_\beta \Fg(x,\xi)\partial_{\xi_j}\partial^\alpha_\beta \Fg(x,\xi)+\sigma^{ij}\frac{\xi_i}{2}\frac{\xi_j}{2}|\partial^\alpha_\beta \Fg(x,\xi)|^2]\, d\xi dx.
$$
Then, we introduce the instant energy
functional $\mathcal{E}_{2,l,q} (t)$ by
\begin{eqnarray}
\label{2.10}
\mathcal{E}_{2,l,q}(t)=\sum_{|\alpha|\leq
2}\|\partial^\alpha
[\widetilde{v},\widetilde{u},\widetilde{\theta}](t)\|^2+\sum_{|\alpha|+|\beta|\leq
2}\|\partial^\alpha_\beta \mathbf{g}(t)\|^2_{w(\beta)},
\end{eqnarray}
and the corresponding energy dissipation functional $\mathcal{D}_{2,l,q} (t)$ by
\begin{equation}
\label{2.11}
\mathcal{D}_{2,l,q}(t)=\sum_{1\leq|\alpha|\leq2}\|\partial^\alpha
[\widetilde{v} ,\widetilde{u} ,\widetilde{\theta} ](t)\|^2+\sum_{|\alpha|+|\beta|\leq2}\|\partial^\alpha_\beta \mathbf{g}(t)\|^2_{\sigma,w(\beta)}.
\end{equation}
We remark that as usual, the instant energy functional $\mathcal{E}_{2,l,q} (t)$ is assumed to be small enough a priori, which will be closed by the energy estimate in the end.



With the above preparations, the main result of this paper can
be stated as follows.

\begin{theorem}\label{thm1.1}
Let $M_{[1/\overline{v},\overline{u},\overline{\theta}](t,x)}(\xi)$ be the viscous contact wave given in Definition \ref{def.vcw} with the
small wave strength
$\delta=|\theta_{+}-\theta_{-}|>0$.
Then, there are a sufficiently small constant $\varepsilon_{0}>0$ and a generic constant $C_{0}>0$  such that if the initial data $F_{0}(x,\xi)\geq 0$ satisfies
\begin{align}
\label{1.39}
C_{0}[\mathcal{E}_{2,l,q}(0)+\delta]\leq \varepsilon^{2}_{0},
\end{align}
where in \eqref{1.32} and \eqref{1.36a} we have choosen $q_{1}=\varepsilon_0$ and  $q_{2}=\frac{1}{\tilde{C}_0\sqrt{\varepsilon_{0}}}$
with some positive constant  $\tilde{C}_0$ and $l\geq 2$ arbitrarily given,
 then the Cauchy problem  \eqref{1.16} and \eqref{1.17} on the Landau equation with Coulomb interaction \eqref{def.lk}  admits a unique  global solution $F(t,x,\xi)\geq 0$ satisfying
\begin{align}
\label{1.40}
\sup_{t\geq 0}\mathcal{E}_{2,l,q}(t)
+c_0\int^{+\infty}_{0}\mathcal{D}_{2,l,q}(t)\,dt\leq \varepsilon^{2}_{0},
\end{align}
for a generic constant $c_0>0$. Moreover, the solution tends in large time toward the viscous contact wave in the sense that
\begin{align}
\label{1.41}
\lim_{t\to +\infty}\|\frac{F(t,x,\xi)-M_{[1/\overline{v},\overline{u},\overline{\theta}](t,x)}(\xi)}{\sqrt{\mu}}\|_{L_{x}^{\infty}L_{\xi}^{2}}=0.
\end{align}
\end{theorem}

\subsection{Relevant literature}
Due to importance of the Landau equation with Coulomb interactions in plasma physics, a lot of fundamental mathematical investigations have been made. In particular, with focus on the spatially inhomogeneous case, we would mention Lions \cite{Lions} and Villani \cite{Vi98,Vi96} for the global existence of weak solutions up to a defect measure,  Desvillettes \cite{Des} and Alexandre-Villani \cite{AV04} for the grazing collision limit of the non-cutoff Boltzmann equation to the Landau equation, Degond-Lemou \cite{DL} for the spectrum analysis of the linearized Landau equation, and Bobylev-Pulvirenti-Saffirio \cite{BPS} for the derivation of the Landau equation from particle systems.

Moreover, closely related to the current work,  Guo \cite{Guo-2002} constructed global solutions to the Landau equation near global Maxwellians
in the torus, see also \cite{Hsiao-Yu}. The polynomial decay rate  and   the  exponential decay rate of the Landau equation near global Maxwellians in the torus were shown by Strain-Guo \cite{Strain-Guo-2006} and \cite{Strain-Guo-2008}, respectively. In the presence of self-consistent forces,
global solutions of the Vlasov-Poisson-Landau system near global Maxwellians have been obtained in Guo \cite{Guo-JAMS} for the torus and in Strain-Zhu \cite{Strain-Zhu} for the whole space, see also \cite{Duan, Wang,Wang-1}.
In addition to those works, great contributions also have been done in many other kinds of topics of the Landau equation, for instance \cite{CaMi, CTW, DV2000, GIMV, GHJO, HeSn, KGH,Luk} and references therein.	

As in the context of the Boltzmann equation, it is a fundamental problem to determine the global existence and large-time behavior of solutions to the Cauchy problem \eqref{1.1}, \eqref{ide} and \eqref{1.2} for the Landau equation  whenever initial data admit a small total variation in space variable over the whole line, in particular connecting two distinct global Maxwellians at infinities, cf.~\cite{SR,Ukai-Yang}. Based on the corresponding fluid dynamic approximation through the Euler or Navier-Stokes equations, one may expect to construct the solution in large time to be either one of the wave patterns, such as shock wave, rarefaction wave and contact wave, or their superposition, cf.~\cite{Smoller}. This can be motivated by the pointwise estimate of solutions via the method of Green functions systematically developed by Liu-Yu \cite{LY-G04}.

We recall some literatures for the existence and stability of wave patterns for the Boltzmann equation so as to make a comparison with the Landau case later.
Under the Grad's angular cutoff assumption, Caflisch-Nicolaenko \cite{Caflisch-Nicolaenko} constructed the shock profile solutions of the Boltzmann equation for hard potentials.
Liu-Yu \cite{Liu-Yu} and Yu \cite{Yu} established the positivity and large-time behavior of shock profile solutions of the Boltzmann equation for hard sphere model, respectively. Motivated by \cite{Huang-Yang}, Huang-Xin-Yang \cite{Huang-Xin-Yang} studied the stability of contact waves with general perturbations for hard potentials and Huang-Wang-Yang \cite{Huang-Wang-Yang} established the hydrodynamic limit with contact waves. The nonlinear stability of rarefaction waves to the Boltzmann equation  was studied in \cite{Liu-Yang-Yu-Zhao,Xin-Yang-Yu,Yang-Zhao}. In addition, the stability of nonlinear wave patterns to the Boltzmann equation with a self-consistent electric field for the hard sphere model has been  considered in \cite{Duan-Liu-2015,Li-Wang-Yang-Zhong} and references therein.   Here we would like to emphasize that in the context of viscous conservation laws, in particular for the compressible Navier-Stokes equations, the stability of contact waves has been extensively studied in the much earlier stage  by \cite{Huang-Li-Matsumura,HMS,Huang-Matsumura-Xin, Kawashima-M, Liu-Xin, Xin} and the references therein; see also a recent very nice survey by Matsumura \cite{Ma}.

Although the wave patterns of the Boltzmann equation with cutoff have been heavily studied as mentioned above, much less is known to the study of wave patterns on the non-cutoff Boltzmann or Landau equations for physically realistic long-range interactions. We would start to work on the project on the Landau equation first, in order to shed a little light on the non-cutoff Boltzmann case for the future. As such,  the first and third authors of this paper studied in \cite{Duan-Yu1} the nonlinear stability of rarefaction waves for the Landau equation with Coulomb potentials, and the current authors also obtained  in \cite{DYY} the small Knudsen rate of convergence to rarefaction waves.  However, the stability of viscous contact waves and  viscous shock profiles for the Landau equation still remains open.

In this work, we are devoted to showing the nonlinear stability of viscous contact waves  to the Landau equation with slab symmetry for the physical Coulomb interaction. More precisely, we construct the unique global  solution to the Landau equation around a local Maxwellian whose fluid quantities are
  viscous contact wave profiles, and we prove that such a local Maxwellian
is time-asymptotically stable. To the best of our knowledge, this seems the first result about the asymptotic stability of viscous contact waves under small perturbations for
the Landau equation. We remark that the explicit time rates of convergence to viscous contact waves as in \eqref{1.41} have to be left open, though they can be obtained via the technique of anti-derivatives in those fundamental works \cite{Huang-Matsumura-Xin,Huang-Xin-Yang,Huang-Yang} mentioned before.

\subsection{Key points of the proof}
As in two previous works \cite{Duan-Yu1,DYY}, the proof of the result is generally based on the analysis on
the compressible Euler and Navier-Stokes equations and decomposition of solutions with respect to the local Maxwellian that was initiated in \cite{Liu-Yu} and developed in \cite{Liu-Yang-Yu} for the Boltzmann theory. Main difficulties for treating the contact wave of the Landau equation are explained as follows.

Recall that by using the decomposition in \cite{Liu-Yu,Liu-Yang-Yu}, the strong velocity dissipation effect of the linearized  operator  and the anti-derivative techniques, the authors in \cite{Huang-Xin-Yang} can overcome the difficulties arising from the slow time-decay rate of the contact wave profiles and  the nonlinear terms $\frac{\xi_{1}}{v}G_{x}$ and $\frac{u_{1}}{v}G_{x}$ as in the microscopic equation \eqref{1.21}, so the dynamic stability of contact discontinuities for the  Boltzmann equation with cutoff hard potentials was proved. However, such a strong velocity dissipation effect in \cite{Huang-Xin-Yang}  is  not available for the Landau equation with Coulomb interactions because the linearized Landau operator $\mathcal{L}$ in \eqref{2.5}  lacks a spectral gap that results in the very weak velocity dissipation by \eqref{2.6} and \eqref{1.35}. Thus, the  approach in \cite{Huang-Xin-Yang} can not be applied to  the Landau equation with Coulomb interactions or even to the cutoff Boltzmann equation with soft potentials.

To  overcome the difficulties above,
we make a crucial use of the new time-velocity weight function $w(\beta)(t,\xi)$ given in \eqref{1.32}.  The factor $\exp\{q(t)\langle \xi\rangle^{2}\}$ in \eqref{1.32} is used to induce an extra {\it quartic} energy dissipation term
\begin{equation}
\label{def.qed}
q_2q_3(t)\sum_{|\alpha|+|\beta|\leq 2}
\|\langle \xi\rangle\partial^{\alpha}_{\beta} \mathbf{g}(t)\|^{2}_{w(\beta)}
\end{equation}
when treating the energy estimates on the nonlinear terms $\frac{\xi_{1}}{v}G_{x}$ and $\frac{u_{1}}{v}G_{x}$ as in the microscopic equation \eqref{1.21}.
One of the key observations is that
the function $q_{3}(t)$ constructed   by \eqref{1.33} is integrable in all time by using the high-order dissipation rate of the macroscopic component in the solution and the time-decay properties of the viscous contact wave profiles. Note that the extra dissipation \eqref{def.qed} is {\it quartic} due to the dependence of $q_3(t)$ on the normal energy dissipation as mentioned in Remark \ref{rem.q12}.  Different from \cite{Duan, Duan-Yu1}, the time-decay rate of solutions is unavailable in the current problem, so the new weighted energy method looks more robust with possible applications to many problems in the similar situation.  The  other factor $\langle \xi\rangle^{(l-|\beta|)}$ in \eqref{1.32} is used to take care of the derivative estimates of the free transport term and the time-asymptotic stability of the contact wave as in \cite{Guo-2002,Duan-Yu1}.

Since the terms $\|\overline{\theta}_x\|^2$ and $\|\overline{u}_x\|^2$ decay in time respectively at rates $(1+t)^{-1/2}$ and  $(1+t)^{-1}$ which are not integrable with respect to time, we need to consider the subtraction of $G(t,x,\xi)$ by $\overline{G}(t,x,\xi)$ as \eqref{2.3} to cancel these terms.
The inverse of linearized operator $L^{-1}_M$
defined as \eqref{1.22} is more complicated than the one in \cite{Huang-Xin-Yang} for the cutoff hard potential Boltzmann equation.
In order to handle the terms involving $L^{-1}_M$, we will make use of the Burnett functions $\hat{A}_i$ and $\hat{B}_{ij}$ as in \eqref{5.1},
see Section \ref{sec.6.1} for the basic properties of the Burnett functions.
Indeed, in terms of the Burnett functions, the terms involving $L^{-1}_M$, such as \eqref{2.14},
can be represented as the inner products of $A_i$ and $B_{ij}$ with $\Theta_{1}$ as in \eqref{1.22}, where $A_i$ and $B_{ij}$ defined in \eqref{5.2} are the inverse of $\hat{A}_i$ and $\hat{B}_{ij}$ under the linear operator $L_M$, respectively, see the identities \eqref{5.15}, \eqref{5.16} and \eqref{2.18} for details. Notice that  $A_i$ and $B_{ij}$ enjoy the fast velocity decay so as to bound any polynomial velocity growth in $\Theta_{1}$, see \eqref{5.18}.
This kind of technique will be used for the energy estimates on both  the macroscopic component $[\widetilde{v},\widetilde{u},\widetilde{\theta}]$
and the microscopic component  $\mathbf{g}$.

In addition, the term $\|[\overline{v}_x,\overline{\theta}_x]\|^2$ decays in time at a rate $(1+t)^{-1/2}$ that is much slower than the one in case we consider the rarefaction profiles in \cite{Duan-Yu1}. This results in the appearance of a difficult term
$$
\int_{\mathbb{R}}(\widetilde{v}^{2}+\widetilde{\theta}^{2})\omega^{2}\,dx
$$
in the energy estimates on the macroscopic component $[\widetilde{v},\widetilde{u},\widetilde{\theta}]$, where $\omega=\omega(t,x)$ is defined in \eqref{2.17}. In Lemma \ref{lem5.6}, we make use of some key observations from the compressible Navier-Stokes equations around contact waves in \cite{Huang-Li-Matsumura} as well as  the above Burnett function technique to get the estimates of such a difficult term.

In the end, to simplify the  energy estimates, we use the decompositions $F=M+\overline{G}+\sqrt{\mu} \mathbf{g}$ as in \cite{Duan-Yu1} to
improve the decompositions in \cite{Huang-Xin-Yang,Liu-Yang-Yu} such that some similar basic estimates
in \cite{Strain-Guo-2008,Wang} around global Maxwellians can be adopted in a convenient way for the current problem on the perturbation around local Maxwellians.

\subsection{Organization of the paper}
The rest of this paper is arranged as follows. In Section \ref{seca.2} we provide some preliminaries for the macro-micro decomposition and the basic properties of the viscous contact waves. In Section \ref{sec.2}, we will
reformulate the system,  make the a priori assumption and establish the non-weighted energy estimates.
In Section \ref{sec.3}, we will establish the  weighted energy
estimates. In Section \ref{sec.4}, we will establish the existence of global  solutions and the large-time asymptotic toward viscous contact waves of solutions to the Cauchy problem for the Landau equation \eqref{1.16} and \eqref{1.17}. In the appendix Section \ref{sec.5}, we will give some basic estimates frequently used in the previous sections.

\medskip

\noindent{\it Notations.} Throughout the paper we shall use $\langle \cdot , \cdot \rangle$  to denote the standard $L^{2}$ inner product in $\mathbb{R}_{\xi}^{3}$
with its corresponding $L^{2}$ norm $|\cdot|_2$. We also use $( \cdot , \cdot )$ to denote $L^{2}$ inner product in
$\mathbb{R}_{x}$ or $\mathbb{R}_{x}\times \mathbb{R}_{\xi}^{3}$  with its corresponding $L^{2}$ norm $\|\cdot\|$.
Let nonnegative integer $\alpha$ and $\beta$ be multi indices $\alpha=[\alpha_{0},\alpha_{1}]$ and $\beta=[\beta_{1},\beta_{2},\beta_{3}]$,
respectively. Denote
$\partial_{\beta}^{\alpha}=\partial_{t}^{\alpha_{0}}\partial_{x}^{\alpha_{1}}
\partial_{\xi_{1}}^{\beta_{1}}\partial_{\xi_{2}}^{\beta_{2}}\partial_{\xi_{3}}^{\beta_{3}}$.
If each component of $\beta$ is not greater than the corresponding one  of
$\overline{\beta}$, we use the standard notation
$\beta\leq\overline{\beta}$. And $\beta<\overline{\beta}$ means that
$\beta\leq\overline{\beta}$ and $|\beta|<|\overline{\beta}|$.
$C^{\bar\beta}_{\beta}$ is the usual  binomial coefficient.
Throughout the paper, generic positive constants are denoted  by either $c$ or $C$,
and $c_{1}$, $c_{2}$ or $C_{1}$, $C_{2}$ etc. are some given constants.
The notation $A\approx B$ is used to denote that there exists constant $c_{0}>1$
such that $c_{0}^{-1}B\leq A\leq c_{0}B$.

\section{Preliminaries}\label{seca.2}

\subsection{Macro-micro decomposition in Eulerian coordinates}

In the present and next subsections, we present the macro-micro decomposition for the Landau equation. To better understand the Lagrangian formulation in the next subsection, we first start with the formulation in the Eulerian coordinates for the convenience of readers.

Recall \eqref{1.3} and \eqref{1.5}. We denote an $L^{2}_{\xi}(\mathbb{R}^{3})$ inner product as
$\langle h,g\rangle =\int_{\mathbb{R}^{3}}h(\xi)g(\xi)\,d\xi$.
And the macroscopic kernel space
is spanned by the following
five pairwise-orthogonal base
\begin{equation}
\label{1.7}
\left\{
\begin{array}{rl}
&\chi_{0}(\xi)=\frac{1}{\sqrt{\rho}}M,
\\
&\chi_{i}(\xi)=\frac{\xi_{i}-u_{i}}{\sqrt{R\rho\theta}}M, \quad \mbox{for $i=1,2,3$,}
\\
&\chi_{4}(\xi)=\frac{1}{\sqrt{6\rho}}\big(\frac{|\xi-u|^{2}}{R\theta}-3\big)M,
\\
&\langle \chi_{i},\frac{\chi_{j}}{M}\rangle=\delta_{ij},
\quad i,j=0,1,2,3,4.
\end{array} \right.
\end{equation}
Using these five basic functions, we define
\begin{equation}
\label{1.8}
P_{0}h=\sum_{i=0}^{4}\langle h,\frac{\chi_{i}}{M}\rangle\chi_{i},\quad P_{1}h=h-P_{0}h,
\end{equation}
where $P_{0}$ and  $P_{1}$ are called the macroscopic projection and microscopic projection, respectively.
A function $h(\xi)$ is called microscopic or non-fluid if
\begin{equation}
\label{1.9}
\int_{\mathbb{R}^{3}}h(\xi)\psi_{i}(\xi)\,d\xi=0, \quad \mbox{for $i=0,1,2,3,4$}.
\end{equation}

For a non-trivial solution profile connecting two different global Maxwellians at $x=\pm\infty$,
we decompose the Landau equation \eqref{1.1} and its solution with respect to the local Maxwellian \eqref{1.5} as
\begin{equation*}
F=M+G, \quad P_{0}F=M, \quad P_{1}F=G,
\end{equation*}
where $M$ and $G$ represent the macroscopic and microscopic
component in the solution respectively. Due to the fact that $Q(M,M)=0$, the Landau equation \eqref{1.1} becomes
\begin{equation}
\label{1.11}
(M+G)_{t}+\xi_{1}(M+G)_{x}=L_{M}G+Q(G,G),
\end{equation}
  where the linearized Landau operator $L_{M}$ around the local Maxwellian $M$ is defined as
\begin{equation*}
L_{M}h:=Q(h,M)+Q(M,h).
\end{equation*}
 And its null space $\mathcal{N}$ is spanned by
$\{\chi_{i}, i=0,1,2,3,4\}$.

 Multiplying \eqref{1.11} by the collision invariants $\psi_{i}(\xi)$ 
and integrating the resulting equations with respect to $\xi$ over $\mathbb{R}^{3}$, one gets the following macroscopic system that
\begin{equation}
\label{1.12}
\begin{cases}
\rho_{t}+(\rho u_{1})_{x}=0,
\\
(\rho u_{1})_{t}+(\rho u_{1}^{2})_{x}+p_{x}=-\int_{\mathbb{R}^{3}} \xi^{2}_{1}G_{x}\, d\xi,
\\
(\rho u_{i})_{t}+(\rho u_{1}u_{i})_{x}=-\int_{\mathbb{R}^{3}} \xi_{1}\xi_{i}G_{x}\, d\xi, ~~i=2,3,
\\
\big(\rho (\theta+\frac{|u|^{2}}{2})\big)_{t}+\big(\rho u_{1}(\theta+\frac{|u|^{2}}{2})+pu_{1}\big)_{x}
=-\int_{\mathbb{R}^{3}} \frac{1}{2}\xi_{1}|\xi|^{2}G_{x}\, d\xi.
\end{cases}
\end{equation}
Here the pressure $p=R\rho\theta$, and we have used \eqref{1.3}, \eqref{1.4} and the fact that $G_{t}$ is microscopic by \eqref{1.9}.
\par
Applying the projection operator $P_{1}$ to \eqref{1.11}, we obtain the following microscopic equation that
\begin{align}
\label{1.13}
G_{t}+P_{1}(\xi_{1}G_{x})+P_{1}(\xi_{1}M_{x})
=L_{M}G+Q(G,G),
\end{align}
which implies that
\begin{equation}
\label{1.14}
G=L^{-1}_{M}[P_{1}(\xi_{1}M_{x})]+L^{-1}_{M}\Theta_{1},
\quad
\Theta_{1}:=G_{t}+P_{1}(\xi_{1}G_{x})-Q(G,G).
\end{equation}
Substituting the expression of $G$ in \eqref{1.14} into \eqref{1.12},
we further obtain the following fluid-type system
\begin{equation}
\label{1.15}
\begin{cases}
\rho_{t}+(\rho u_{1})_{x}=0,
\\
(\rho u_{1})_{t}+(\rho u_{1}^{2})_{x}+p_{x}=\frac{4}{3}(\mu(\theta)u_{1x})_{x}-(\int_{\mathbb{R}^{3}} \xi^{2}_{1}L^{-1}_{M}\Theta_{1}\, d\xi)_{x},
\\
(\rho u_{i})_{t}+(\rho u_{1}u_{i})_{x}=(\mu(\theta)u_{ix})_{x}-(\int_{\mathbb{R}^{3}} \xi_{1}\xi_{i}L^{-1}_{M}\Theta_{1}\, d\xi)_{x}, ~~i=2,3,
\\
\big(\rho (\theta+\frac{|u|^{2}}{2})\big)_{t}+\big(\rho u_{1}(\theta+\frac{|u|^{2}}{2})+pu_{1}\big)_{x}=(\kappa(\theta)\theta_{x})_{x}+\frac{4}{3}(\mu(\theta)u_{1}u_{1x})_{x}
\\
\ \ \hspace{3cm}+(\mu(\theta)u_{2}u_{2x})_{x}+(\mu(\theta)u_{3}u_{3x})_{x}
-\frac{1}{2}(\int_{\mathbb{R}^{3}}\xi_{1}|\xi|^{2}L^{-1}_{M}\Theta_{1}\, d\xi)_{x}.
\end{cases}
\end{equation}
Here the viscosity coefficient $\mu(\theta)>0$ and the heat conductivity coefficient $\kappa(\theta)>0$  are smooth functions depending only on $\theta$.
The explicit formula of  $\mu(\theta)$ and $\kappa(\theta)$ are defined by \eqref{5.3}.

\subsection{Macro-micro decomposition in Lagrangian coordinates}

As mentioned before, it is necessary for us to reformulate the problem in the Lagrangian coordinates. Recall the coordinate transform \eqref{def.lct} as well as the reformulated Cauchy problem \eqref{1.16} and \eqref{1.17}  in the Lagrangian coordinates. In terms of \eqref{def.lct}, it is then direct to obtain all the formulations similar to those in the previous subsection.  In fact,
with \eqref{def.vrho}, the five pairwise-orthogonal base in \eqref{1.7} becomes
\begin{equation*}
\left\{
\begin{array}{rl}
&\chi_{0}(\xi)=\sqrt{v}M,
\\
&\chi_{i}(\xi)=\sqrt{v}\frac{\xi_{i}-u_{i}}{\sqrt{R\theta}}M, \quad \mbox{for $i=1,2,3$,}
\\
&\chi_{4}(\xi)=\sqrt{v}\frac{1}{\sqrt{6}}\big(\frac{|\xi-u|^{2}}{R\theta}-3\big)M,
\\
&\langle \chi_{i},\frac{\chi_{j}}{M}\rangle=\delta_{ij},
\quad i,j=0,1,2,3,4.
\end{array} \right.
\end{equation*}
Using these five basic functions, $P_{0}$ and $P_{1}$ in \eqref{1.8} can rewrite as
\begin{equation}
\label{1.20a}
P_{0}h=\sum_{i=0}^{4}\langle h,\frac{\chi_{i}}{M}\rangle\chi_{i},\quad P_{1}h=h-P_{0}h.
\end{equation}
By using the fact that $F=M+G$, the macroscopic system \eqref{1.12} and \eqref{1.15}
in the Lagrangian coordinates become, respectively
\begin{equation}
\label{1.19}
\begin{cases}
v_{t}-u_{1x}=0,
\\
u_{1t}+p_{x}=-\int_{\mathbb{R}^{3}} \xi^{2}_{1}G_{x}\,d\xi,
\\
u_{it}=-\int_{\mathbb{R}^{3}} \xi_{1}\xi_{i}G_{x}\,d\xi, ~~i=2,3,
\\
(\theta+\frac{|u|^{2}}{2})_{t}+(pu_{1})_{x}
=-\int_{\mathbb{R}^{3}} \frac{1}{2}\xi_{1}|\xi|^{2}G_{x}\,d\xi,
\end{cases}
\end{equation}
and
\begin{equation}
\label{1.20}
\begin{cases}
v_{t}-u_{1x}=0,
\\
u_{1t}+p_{x}=\frac{4}{3}(\frac{\mu(\theta)}{v}u_{1x})_{x}-(\int_{\mathbb{R}^{3}} \xi^{2}_{1}L^{-1}_{M}\Theta_{1}\, d\xi)_{x},
\\
u_{it}=(\frac{\mu(\theta)}{v}u_{ix})_{x}-(\int_{\mathbb{R}^{3}} \xi_{1}\xi_{i}L^{-1}_{M}\Theta_{1}\, d\xi)_{x}, ~~i=2,3,
\\
(\theta+\frac{|u|^{2}}{2})_{t}+(pu_{1})_{x}
=(\frac{\kappa(\theta)}{v}\theta_{x})_{x}+\frac{4}{3}(\frac{\mu(\theta)}{v}u_{1}u_{1x})_{x}
\\
\ \ \hspace{3cm} +(\frac{\mu(\theta)}{v}u_{2}u_{2x})_{x}+(\frac{\mu(\theta)}{v}u_{3}u_{3x})_{x}
-\frac{1}{2}(\int_{\mathbb{R}^{3}}\xi_{1}|\xi|^{2}L^{-1}_{M}\Theta_{1}\, d\xi)_{x},
\end{cases}
\end{equation}
where the pressure $p=\frac{2\theta}{3v}$. Moreover, the microscopic equation \eqref{1.13} becomes
\begin{align}
\label{1.21}
G_{t}-\frac{u_{1}}{v}G_{x}+\frac{1}{v}P_{1}(\xi_{1}G_{x})+\frac{1}{v}P_{1}(\xi_{1}M_{x})
=L_{M}G+Q(G,G),
\end{align}
which implies that
\begin{equation}
\label{1.22}
G=L^{-1}_{M}[\frac{1}{v}P_{1}(\xi_{1}M_{x})]+L^{-1}_{M}\Theta_{1},
\quad
\Theta_{1}:=G_{t}-\frac{u_{1}}{v}G_{x}+\frac{1}{v}P_{1}(\xi_{1}G_{x})-Q(G,G).
\end{equation}

\subsection{Viscous contact waves}\label{subsec.vcw}

Now we turn to  define the contact wave profile for the Landau equation \eqref{1.16} and \eqref{1.17} as in \cite{Huang-Matsumura-Xin,
Huang-Xin-Yang}. If we take the microscopic component $G$ be to zero in \eqref{1.19}, we have the following compressible Euler system that
\begin{equation}
\label{1.23}
\begin{cases}
v_{t}-u_{1x}=0,
\\
u_{1t}+p_{x}=0,
\\
u_{it}=0, ~~i=2,3,
\\
(\theta+\frac{|u|^{2}}{2})_{t}+(pu_{1})_{x}=0,
\end{cases}
\end{equation}
with a Riemann initial data
\begin{equation}
\label{1.24}
[v_{0},u_{0},\theta_{0}](x)
=\begin{cases}
[v_{+},u_{+},\theta_{+}],\quad x>0,
\\
[v_{-},u_{-},\theta_{-}],\quad x<0.
\end{cases}
\end{equation}
Here $u_{\pm}=(u_{1\pm},0,0)^{t}$, $v_{\pm}>0$ and $\theta_{\pm}>0$
are given constants as in \eqref{1.17}.
It is well known that the Riemann problem \eqref{1.23} and \eqref{1.24} admits a contact discontinuity solution (cf. \cite{Smoller})
\begin{align}
\label{1.25}
[\overline{V},\overline{U},\overline{\Theta}](t,x)
=\begin{cases}
[v_{+},u_{+},\theta_{+}],\quad x>0,
\\
[v_{-},u_{-},\theta_{-}],\quad x<0,
\end{cases}
\end{align}
under the conditions that
\begin{equation*}
u_{-}=u_{+},\quad  p_{-}:=\frac{R\theta_{-}}{v_{-}}=\frac{R\theta_{+}}{v_{+}}=:p_{+},\quad v_{-}\neq v_{+}.
\end{equation*}
Note that \eqref{1.20} becomes the compressible Navier-Stokes equations by letting $\Theta_{1}$ be  zero. By using  the mass equation \eqref{1.20}$_{1}$ and
the energy equation \eqref{1.20}$_{4}$ with $p_{+}=\frac{2\theta}{3v}$, we obtain a nonlinear diffusion equation as follows
(cf. \cite{HMS,Huang-Xin-Yang,Huang-Matsumura-Xin})
\begin{equation}
\label{1.26}
\theta_{t}=(a(\theta)\theta_{x})_{x}, \quad
a(\theta)=\frac{9\kappa(\theta) p_{+}}{10\theta}>0,
\end{equation}
which admits a unique self-similar solution $\Theta(\zeta)$  with $\zeta=\frac{x}{\sqrt{1+t}}$ satisyfing the boundary conditions $\Theta(t,\pm\infty)=\theta_{\pm}$. Moreover, there exists a constant $c_{1}>0$ depending only on $\theta_{\pm}$ such that for any $t\geq0$ and $x\in \R$, it holds that
\begin{align}
\label{1.27}
(1+t)^{3/2}|\Theta_{xxx}|+(1+t)|\Theta_{xx}|+(1+t)^{1/2}|\Theta_{x}|+|\Theta-\theta_{\pm}|\leq C\delta e^{-\frac{c_{1}x^{2}}{1+t}},
\end{align}
where $\delta=|\theta_{+}-\theta_{-}|$ is the strength of the diffusion wave $\Theta$ and $C>0$ is a generic constant.
Then we can define the contact wave profile $(\bar{v},\bar{u},\bar{\theta})(t,x)$ as follows
\begin{align}
\label{1.28}
\bar{v}=\frac{2}{3p_{+}}\Theta,\quad \bar{u}_{1}=u_{1-}+\frac{2a(\Theta)}{3p_{+}}\Theta_{x},\quad \bar{u}_{2}=\bar{u}_{3}=0, \quad \bar{\theta}=\Theta.
\end{align}
By \eqref{1.26}, \eqref{1.27} and \eqref{1.28}, for any $q\geq1 $ and any integer
$k\geq 1$,  we can verify that $(\bar{v},\bar{u},\bar{\theta})(t,x)$ has the following properties
\begin{equation}
\label{1.29}
\left\{
\begin{array}{rl}
&\|\partial^{k}_{x}[\bar{v},\bar{\theta}]\|_{L^{q}}
+(1+t)^{\frac{1}{2}}\|\partial^{k}_{x}\bar{u}\|_{L^{q}}
\leq C\delta (1+t)^{-\frac{1}{2}(k-\frac{1}{q})},
\\
&\|\partial^{k}_{t}[\bar{v},\bar{\theta}]\|_{L^{q}}
+(1+t)^{\frac{1}{2}}\|\partial^{k}_{t}\bar{u}\|_{L^{q}}
\leq C\delta (1+t)^{-(k-\frac{1}{2q})}.
\end{array}\right.
\end{equation}
In view of \eqref{1.25}, \eqref{1.26}, \eqref{1.27} and \eqref{1.28}, we can obtain
\begin{align*}
\|[\bar{v}-\overline{V},\bar{u}-\overline{U},\overline{\theta}-\overline{\Theta}]\|_{L^{q}}
\leq C\left(\max_{\Theta\in[\min\{\theta_{-},\theta_{+}\},\,\max\{\theta_{-},\theta_{+}\}]}\kappa(\Theta)\right)^{\frac{1}{2q}}(1+t)^{\frac{1}{2q}},\quad q\geq 1,
\end{align*}
which means that the viscous contact wave $(\bar{v},\bar{u},\bar{\theta})$  defined in \eqref{1.28}
can be regarded as a local-in-time smooth approximation to the contact discontinuity solution
$[\overline{V},\overline{U},\overline{\Theta}]$ for the Euler system \eqref{1.23} in $L^{q}$-norm $(q\geq1)$
as the heat conductivity coefficient $\kappa(\cdot)$ tends to zero.
More importantly, $[\bar{v},\bar{u},\bar{\theta}]$ satisfies
\begin{align}
\label{1.30}
\begin{cases}
\bar{v}_{t}-\bar{u}_{1x}=0,
\\
\bar{u}_{1t}+\bar{p}_{x}
=\frac{4}{3}(\frac{\mu(\bar{\theta})}{\bar{v}}\bar{u}_{1x})_{x}+\mathcal{R}_{1},
\\
\bar{u}_{it}=(\frac{\mu(\bar{\theta})}{\bar{v}}\bar{u}_{ix})_{x}, ~~i=2,3,
\\
(\bar{\theta}+\frac{|\bar{u}|^{2}}{2})_{t}+(\bar{p}\bar{u}_{1})_{x}=(\frac{\kappa(\bar{\theta})}{\bar{v}}\bar{\theta}_{x})_{x}
+\frac{4}{3}(\frac{\mu(\bar{\theta})}{\bar{v}}\bar{u}_{1}\bar{u}_{1x})_{x}+\mathcal{R}_{2},
\end{cases}
\end{align}
where $\bar{p}=\frac{2\bar{\theta}}{3\bar{v}}=p_{+}$ and
\begin{align*}
\mathcal{R}_{1}=\bar{u}_{1t}-\frac{4}{3}(\frac{\mu(\bar{\theta})}{\bar{v}}\bar{u}_{1x})_{x},\quad
\mathcal{R}_{2}=\bar{u}\bar{u}_{t}-\frac{4}{3}(\frac{\mu(\bar{\theta})}{\bar{v}}\bar{u}_{1}\bar{u}_{1x})_{x}.
\end{align*}
These facts will be frequently used in the proofs later on.


\section{Non-weighted energy estimates}\label{sec.2}

In this section, we will deduce the energy estimates for the Cauchy problem \eqref{1.16} and \eqref{1.17}.
We first reformulate the system and make the a priori assumption in subsection \ref{sub2.1}.
Then we derive the lower order energy estimates  and establish the high order energy estimates in subsection \ref{sub2.2} and  subsection \ref{sub2.3}, respectively. Those energy estimates are carried out without any weight function for the time being. The weighted energy estimates will be made in the next section.

\subsection{Reformulated system}\label{sub2.1}
We will first derive the equation of the microscopic component $\mathbf{g}$ in \eqref{2.1}. Since the term
$\frac{1}{v}P_1(\xi_1M_x)$ in \eqref{1.21} contains $\| \overline{\theta}_{x}\|^2$ and the time decay of $\| \overline{\theta}_{x}\|^2$ is $(1+t)^{-\frac{1}{2}}$ by \eqref{1.29}, which is not integrable about the time $t$, we need to subtract $\overline{G}$ in \eqref{2.3} from $G$ to cancel this term as in \cite{Huang-Xin-Yang,Liu-Yang-Yu-Zhao}. Thus one has from \eqref{2.3} and \eqref{1.5} that
\begin{align*}
\frac{1}{v}P_{1}(\xi_{1}M_{x})=\frac{1}{v}P_{1}\xi_{1}M\big\{\frac{|\xi-u|^{2}
\widetilde{\theta}_{x}}{2R\theta^{2}}+\frac{(\xi-u)\cdot\widetilde{u}_{x}}{R\theta}\big\}
+L_{M}\overline{G}.
\end{align*}
Recalling that  $G=\overline{G}+\sqrt{\mu}\mathbf{g}$, by \eqref{1.21} we derive the equation of the microscopic component $\mathbf{g}$
as follows
\begin{align}
\label{2.4}
v\mathbf{g}_{t}+\xi_{1}\mathbf{g}_{x}-u_{1}\mathbf{g}_{x}
&=v\mathcal{L}\mathbf{g}+v\Gamma(\mathbf{g},\frac{M-\mu}{\sqrt{\mu}})+
v\Gamma(\frac{M-\mu}{\sqrt{\mu}},\mathbf{g})+v\Gamma(\frac{G}{\sqrt{\mu}},\frac{G}{\sqrt{\mu}})
\nonumber\\
&\quad+\frac{P_{0}(\xi_{1}\sqrt{\mu}\mathbf{g}_{x})}{\sqrt{\mu}}
-\frac{1}{\sqrt{\mu}}P_{1}\xi_{1}M\big\{\frac{|\xi-u|^{2}
\widetilde{\theta}_{x}}{2R\theta^{2}}+\frac{(\xi-u)\cdot\widetilde{u}_{x}}{R\theta}\big\}
\nonumber\\
&\quad-\frac{P_{1}(\xi_{1}\overline{G}_{x})}{\sqrt{\mu}}+u_{1}\frac{\overline{G}_{x}}{\sqrt{\mu}}
-v\frac{\overline{G}_{t}}{\sqrt{\mu}},
\end{align}
where $\Gamma$ and $\mathcal{L}$ are defined by
\begin{equation}
\label{2.5}
\Gamma(f,g):=\frac{1}{\sqrt{\mu}}Q(\sqrt{\mu}f,\sqrt{\mu}g),
\quad \mathcal{L}f:=\Gamma(\sqrt{\mu},f)+\Gamma(f,\sqrt{\mu}).
\end{equation}
Here  we have used the fact that
\begin{equation*}
\frac{1}{\sqrt{\mu}}L_{M}(\sqrt{\mu}f)=
\frac{1}{\sqrt{\mu}}\{Q(M,\sqrt{\mu}f)+Q(\sqrt{\mu}f,M)\}
=\mathcal{L}f+\Gamma(f,\frac{M-\mu}{\sqrt{\mu}})+
\Gamma(\frac{M-\mu}{\sqrt{\mu}},f).
\end{equation*}
Note that  the linearized Landau operator $\mathcal{L}$ is self-adjoint and non-positive definite, and its null space $\ker\mathcal{L}$ is spanned by the five functions $\{\sqrt{\mu},\xi\sqrt{\mu},|\xi|^{2}\sqrt{\mu}\}$, cf. \cite{Guo-2002}. Moreover,  there exists a constant $c_{2}>0$ such that
\begin{equation}
\label{2.6}
-\langle\mathcal{L}g, g \rangle\geq c_{2}|g|^{2}_{\sigma},
\end{equation}
for any $g\in (\ker\mathcal{L})^{\perp}$,

On the other hand, by using \eqref{1.30} and \eqref{1.20},
we obtain the system for the perturbation $[\widetilde{v},\widetilde{u},\widetilde{\theta}]$ in \eqref{2.1} as follows
\begin{equation}
\label{2.7}
\begin{cases}
\widetilde{v_{t}}-\widetilde{u}_{1x}=0,
\\
\widetilde{u}_{1t}+(p-p_{+})_{x}=\frac{4}{3}\big(\frac{\mu(\theta)}{v}u_{1x}\big)_{x}
-\bar{u}_{1t}-\int_{\mathbb{R}^{3}} \xi^{2}_{1}(L^{-1}_{M}\Theta_{1})_{x} \,d\xi,
\\
\widetilde{u}_{it}=\big(\frac{\mu(\theta)}{v}\widetilde{u}_{ix}\big)_{x}
-\int_{\mathbb{R}^{3}} \xi_{1}\xi_{i}(L^{-1}_{M}\Theta_{1})_{x} \,d\xi, ~~i=2,3,
\\
\widetilde{\theta}_{t}+pu_{1x}-p_{+}\bar{u}_{1x}
=\big(\frac{\kappa(\theta)}{v}\theta_{x}-\frac{\kappa(\bar{\theta})}{\bar{v}}\bar{\theta}_{x}\big)_{x}
+Q_{1}
\\
\hspace{1cm}+u\cdot\int_{\mathbb{R}^{3}} \xi\xi_{1}(L^{-1}_{M}\Theta_{1})_{x} \,d\xi
-\frac{1}{2}\int_{\mathbb{R}^{3}}\xi_{1}|\xi|^{2}(L^{-1}_{M}\Theta_{1})_{x} \,d\xi,
\end{cases}
\end{equation}
where we have used the facts that $p_{+}=\bar{p}=\frac{2\bar{\theta}}{3\bar{v}}$ and $\bar{u}_{2}=\bar{u}_{3}=0$ as well as
\begin{equation}
\label{2.8}
Q_{1}=\frac{4}{3}\frac{\mu(\theta)}{v}u^{2}_{1x}
+\frac{\mu(\theta)}{v}u^{2}_{2x}+\frac{\mu(\theta)}{v}u^{2}_{3x}.
\end{equation}

To prove the global existence of the solution in Theorem \ref{thm1.1},
the key point is to establish uniform energy estimates on the
macroscopic part $[\widetilde{v},\widetilde{u},\widetilde{\theta}]$
and the microscopic part $\mathbf{g}$. 
For an arbitrary time $T>0$, we shall make the following a priori assumption:
\begin{align}
\label{2.9}
\sup_{0\leq t\leq T}\mathcal{E}_{2,l,q}(t)+C_{1}\int_0^T\mathcal{D}_{2,l,q}(s)\,ds\leq  \varepsilon^{2}_{0},
\end{align}
where $\varepsilon_{0}>0$ is a small constant and $C_{1}>0$ is a constant to be determined in the end of the proof, see Section \ref{sec.4} later on. We remark that all the estimates below are independent of $T$.

With \eqref{2.9}, \eqref{1.32}, \eqref{1.33} and the assumptions in Theorem \ref{thm1.1}, for any $t\in(0,T]$,  we arrive at
\begin{align}
\label{2.9a}
 q_3(t)+ \int_0^tq_3(s)\,ds\leq  \tilde{C}_0\varepsilon^{2}_{0} ,\ \ \ q_2\big(q_3(t)+\int_0^tq_3(s)\,ds\big)\in(0,  \varepsilon_{0}^{\frac{3}{2}}),\ \ \ \mbox{and} \ \ \sup_{0\leq t\leq T}q(t)\in(0,  \varepsilon_{0}).
\end{align}
By \eqref{1.37}, the a priori assumption \eqref{2.9} and the imbedding inequality,  we have that for any small $\eta_0>0$ and $\varepsilon_0>0$, it holds that
\begin{equation}
\label{2.12}
|v(t,x)-1|+|u(t,x)|+|\theta(t,x)-\frac{3}{2}|<C(\eta_{0}+\varepsilon_{0}),\quad
1<\theta(t,x)<3,
\end{equation}
uniformly in all $(t,x)$.
\subsection{Lower order energy estimates}\label{sub2.2}
In this subsection, we derive the lower order energy estimates for the macroscopic
component $[\widetilde{v},\widetilde{u},\widetilde{\theta}]$
and the microscopic component  $\mathbf{g}$. First,
multiplying  \eqref{2.7}$_{2}$ by $\widetilde{u}_{1}$ and \eqref{2.7}$_{3}$ by $\widetilde{u}_{i}$ $(i=2,3)$, then adding the resulting equations
together and using the fact that $p-p_{+}=\frac{R\widetilde{\theta}-p_{+}\widetilde{v}}{v}$, we have
$$
(\frac{\widetilde{u}^{2}}{2})_{t}-\frac{R\widetilde{\theta}-p_{+}\widetilde{v}}{v}\widetilde{u}_{1x}=-\frac{4}{3}\frac{\mu(\theta)}{v}u_{1x}\widetilde{u}_{1x}
-\sum^{3}_{i=2}\frac{\mu(\theta)}{v}\widetilde{u}^{2}_{ix}
-\bar{u}_{1t}\widetilde{u}_{1}-\widetilde{u}\cdot\int_{\mathbb{R}^{3}}\xi \xi_{1}(L^{-1}_{M}\Theta_{1})_{x}\, d\xi+(\cdot\cdot\cdot)_{x}.
$$
Here and in the sequel the notation $(\cdot\cdot\cdot)_{x}$ represents the term in the conservative form so that
it vanishes after integration. Multiplying \eqref{2.7}$_{1}$ by $\frac{p_{+}}{v}\widetilde{v}$
gives that
$$
\frac{p_{+}}{v}\widetilde{v}\widetilde{u}_{1x}=\frac{p_{+}}{v}\widetilde{v}\widetilde{v_{t}}
=-\frac{2}{3}\bar{\theta}(\frac{1}{v}-\frac{1}{\bar{v}})\widetilde{v_{t}}
=\big(\frac{2}{3}\bar{\theta}\Phi(\frac{v}{\bar{v}})\big)_{t}+\bar{p}\Phi(\frac{\bar{v}}{v})\bar{v}_{t},
$$
where we have used the facts that $p_{+}=\bar{p}$ and $\Phi(s)=s-\ln s-1$. Multiplying \eqref{2.7}$_{4}$ by $\frac{\widetilde{\theta}}{\theta}$ gives
\begin{align*}
\frac{\widetilde{\theta}}{\theta}\widetilde{\theta}_{t}=&-\frac{R}{v}\widetilde{\theta}\widetilde{u}_{1x}
+\frac{\widetilde{\theta}}{\theta}(p_{+}-p)\bar{u}_{1x}
-(\frac{\widetilde{\theta}}{\theta})_{x}\big(\frac{\kappa(\theta)}{v}\theta_{x}-\frac{\kappa(\bar{\theta})}{\bar{v}}\bar{\theta}_{x}\big)
+\frac{\widetilde{\theta}}{\theta}Q_{1}
\nonumber\\
&+\frac{\widetilde{\theta}}{\theta}u\cdot\int_{\mathbb{R}^{3}} \xi\xi_{1}(L^{-1}_{M}\Theta_{1})_{x}\,d\xi
-\frac{1}{2}\frac{\widetilde{\theta}}{\theta}\int_{\mathbb{R}^{3}}\xi_{1}|\xi|^{2}(L^{-1}_{M}\Theta_{1})_{x}\, d\xi+(\cdot\cdot\cdot)_{x}.
\end{align*}
Note that
$$
\big(\bar{\theta}\Phi(\frac{\theta}{\bar{\theta}})\big)_{t}=\frac{\widetilde{\theta}}{\theta}\widetilde{\theta}_{t}
-\Phi(\frac{\bar{\theta}}{\theta})\bar{\theta}_{t}.
$$
Combining the above   equalities, we   arrive at
\begin{align}
\label{2.13}
&\Big(\frac{2}{3}\bar{\theta}\Phi(\frac{v}{\bar{v}})+\frac{1}{2}\widetilde{u}^{2}
+\bar{\theta}\Phi(\frac{\theta}{\bar{\theta}})\Big)_{t}+\frac{4}{3}\frac{\mu(\theta)}{v}\widetilde{u}^{2}_{1x}
+\sum_{i=2}^{3}\frac{\mu(\theta)}{v}\widetilde{u}^{2}_{ix}+\frac{\kappa(\theta)}{v\theta}\widetilde{\theta}^{2}_{x}+(\cdot\cdot\cdot)_{x}
\nonumber\\
&=-\bar{p}\Phi(\frac{\bar{v}}{v})\bar{v}_{t}
-\Phi(\frac{\bar{\theta}}{\theta})\bar{\theta}_{t}+\frac{\widetilde{\theta}}{\theta}(p_{+}-p)\bar{u}_{1x}
-\frac{\widetilde{\theta}_{x}}{\theta}(\frac{\kappa(\theta)}{v}-\frac{\kappa(\bar{\theta})}{\bar{v}})\bar{\theta}_{x}
\nonumber\\
&\qquad+\frac{\widetilde{\theta}\theta_{x}}{\theta^{2}}(\frac{\kappa(\theta)\theta_{x}}{v}-\frac{\kappa(\bar{\theta})\bar{\theta}_{x}}{\bar{v}})
-\frac{4}{3}\frac{\mu(\theta)}{v}\bar{u}_{1x}\widetilde{u}_{1x}-\bar{u}_{1t}\widetilde{u}_{1}
+\frac{\widetilde{\theta}}{\theta}Q_{1}+\mathbb{H}.
\end{align}
Here the term $\mathbb{H}$ is given by
\begin{equation}
\label{2.14}
\mathbb{H}=-\widetilde{u}\cdot\int_{\mathbb{R}^{3}}\xi \xi_{1}(L^{-1}_{M}\Theta_{1})_{x} \,d\xi
+\frac{\widetilde{\theta}}{\theta}\big\{ u\cdot\int_{\mathbb{R}^{3}} \xi\xi_{1}(L^{-1}_{M}\Theta_{1})_{x}\,d\xi
-\frac{1}{2}\int_{\mathbb{R}^{3}}\xi_{1}|\xi|^{2}(L^{-1}_{M}\Theta_{1})_{x} \,d\xi\big\}.
\end{equation}
First note that $\Phi'(1)=0$ and $\Phi(s)$ is strictly convex around $s=1$, we can obtain
\begin{equation}
\label{2.15}
\Phi(\frac{v}{\bar{v}})\approx\widetilde{v}^{2},
\quad \Phi(\frac{\theta}{\bar{\theta}})\approx\widetilde{\theta}^{2}.
\end{equation}

Since both $\mu(\theta)$ and $\kappa(\theta)$ are positive smooth functions about $\theta$, there exists $c_{3}>1$ such that
$c^{-1}_{3}\leq\mu(\theta)\leq c_{3}$ and $c^{-1}_{3}\leq\kappa(\theta)\leq c_{3}$.  For any $\lambda\in(0,c_{1}/4]$ with $c_{1}$ as in \eqref{1.27}, we denote
\begin{align}
\label{2.17}
\omega(t,x)=(1+t)^{-\frac{1}{2}}\exp\big(-\frac{\lambda x^{2}}{1+t}\big).
\end{align}
We thus have from this, \eqref{2.15}, \eqref{1.27}, \eqref{1.28}, \eqref{1.29} and the H\"{o}lder inequality that
\begin{align*}
\int_{\mathbb{R}}&\Big\{-\bar{p}\Phi(\frac{\bar{v}}{v})\bar{v}_{t}
-\Phi(\frac{\bar{\theta}}{\theta})\bar{\theta}_{t}+\frac{\widetilde{\theta}}{\theta}(p_{+}-p)\bar{u}_{1x}
-\frac{\widetilde{\theta}_{x}}{\theta}(\frac{\kappa(\theta)}{v}-\frac{\kappa(\bar{\theta})}{\bar{v}})\bar{\theta}_{x}\Big\}\, dx
\nonumber\\
&\leq C\delta(1+t)^{-1}\int_{\mathbb{R}}e^{-\frac{c_{1}x^{2}}{1+t}}(\widetilde{v}^{2}+\widetilde{\theta}^{2})\,dx
+C\delta\|\widetilde{\theta}_{x}\|^{2}
\nonumber\\
&\leq C\delta\int_{\mathbb{R}}(\widetilde{v}^{2}+\widetilde{\theta}^{2})\omega^{2}\,dx+C\delta\|\widetilde{\theta}_{x}\|^{2},
\end{align*}
where we have used  the facts that $p-p_{+}=\frac{R\widetilde{\theta}-p_{+}\widetilde{v}}{v}$.
Similarly, it holds that
\begin{align*}
\int_{\mathbb{R}} \frac{\widetilde{\theta}\theta_{x}}{\theta^{2}}(\frac{\kappa(\theta)\theta_{x}}{v}-\frac{\kappa(\bar{\theta})\bar{\theta}_{x}}{\bar{v}})
dx&\leq C\int_{\mathbb{R}} |\widetilde{\theta}\theta_{x}|
(|\widetilde{\theta}_{x}|+|\widetilde{v}||\bar{\theta}_{x}|+|\widetilde{\theta}| |\bar{\theta}_{x}|)\,dx
\nonumber\\
&\leq C\delta(1+t)^{-1}\int_{\mathbb{R}}e^{-\frac{c_{1}x^{2}}{1+t}}(\widetilde{v}^{2}+\widetilde{\theta}^{2})\,dx
+C\|\widetilde{\theta}\|_{L_{x}^{\infty}}\|\widetilde{\theta}_{x}\|^{2}
+C\delta\|\widetilde{\theta}_{x}\|^{2}
\nonumber\\
&\leq C\delta\int_{\mathbb{R}}(\widetilde{v}^{2}+\widetilde{\theta}^{2})\omega^{2}\,dx+C(\delta+\varepsilon_{0})\|\widetilde{\theta}_{x}\|^{2},
\end{align*}
where we have used \eqref{2.9} and the following   imbedding inequality
\begin{align*}
\|g\|_{L^{\infty}(\R)}\leq \sqrt{2}\|g\|^{\frac{1}{2}}_{L^2(\R)}\|g'\|^{\frac{1}{2}}_{L^2(\R)},
\quad \mbox{for}\quad g=g(x)\in H^{1}(\mathbb{R})\subset L^{\infty}(\mathbb{R}).
\end{align*}
By  this, \eqref{1.29} and the H\"{o}lder inequality, one has
\begin{align*}
\int_{\mathbb{R}}\big\{ -\frac{4}{3}\frac{\mu(\theta)}{v}\bar{u}_{1x}\widetilde{u}_{1x}-\bar{u}_{1t}\widetilde{u}_{1} \big\}\,dx
&\leq C\|\bar{u}_{1x}\|\|\widetilde{u}_{1x}\|+
 C \|\widetilde{u}_{1}\|_{L^{\infty}_x} \|\bar{u}_{1t}\|_{L^{1}_x}
\\
&\leq C\delta\|\widetilde{u}_{1x}\|^{2}+C\|\widetilde{u}_{1}\|^{2}\|\widetilde{u}_{1x}\|^{2}+C\delta (1+t)^{-\frac{4}{3}}
\\
&\leq C(\delta+\varepsilon_{0})\|\widetilde{u}_{1x}\|^{2}+C\delta(1+t)^{-\frac{4}{3}}.
\end{align*}
By the expression of $Q_{1}$ in \eqref{2.8}, we can obtain
\begin{align*}
\int_{\mathbb{R}}|\frac{\widetilde{\theta}}{\theta}Q_{1}|\,dx&=\int_{\mathbb{R}}|\frac{\widetilde{\theta}}{\theta}\big\{\frac{4}{3}\frac{\mu(\theta)}{v}u^{2}_{1x}
+\frac{\mu(\theta)}{v}u^{2}_{2x}+\frac{\mu(\theta)}{v}u^{2}_{3x}\big\}|\,dx
\\
&\leq C\|\widetilde{\theta}\|_{L^{\infty}}\int_{\mathbb{R}}(\widetilde{u}_{x}^{2}+\bar{u}^{2}_{x})\,dx
\leq C\varepsilon_{0}\|\widetilde{u}_{x}\|^{2}+C\delta(1+t)^{-\frac{3}{2}}.
\end{align*}
Finally we  estimate $\int_{\mathbb{R}}\mathbb{H}\,dx$.  By using \eqref{2.14} and the integration by parts, one has
\begin{align*}
\int_{\mathbb{R}}\mathbb{H}\,dx &= \int_{\mathbb{R}}(\widetilde{u}_{x}-\frac{\widetilde{\theta}}{\theta}u_{x})
\cdot\int_{\mathbb{R}^{3}}\xi \xi_{1}L^{-1}_{M}\Theta_{1} \,d\xi dx\\
&\qquad+ \int_{\mathbb{R}}(\frac{\widetilde{\theta}}{\theta})_{x}\int_{\mathbb{R}^{3}}(\frac{1}{2}\xi_{1}|\xi|^{2}-\xi_{1}\xi\cdot u)L^{-1}_{M}\Theta_{1}\, d\xi dx :=J_1+J_2.
\end{align*}
We first estimate the term $J_2$. By using \eqref{5.1}, \eqref{5.2} and the self-adjoint property of $L^{-1}_{M}$, we have
\begin{align}
\label{5.15}
\int_{\mathbb{R}^{3}} (\frac{1}{2}\xi_{1}|\xi|^{2}&-\xi_{1}\xi\cdot u)L^{-1}_{M}\Theta_{1} \,d\xi=
\int_{\mathbb{R}^{3}} L^{-1}_{M}\{P_{1}(\frac{1}{2}\xi_{1}|\xi|^{2}-\xi_{1}\xi\cdot u)M\}\frac{\Theta_{1}}{M}\, d\xi
\nonumber\\
=&\int_{\mathbb{R}^{3}} L^{-1}_{M}\{(R\theta)^{\frac{3}{2}}\hat{A}_{1}(\frac{\xi-u}{\sqrt{R\theta}})M\}\frac{\Theta_{1}}{M}\, d\xi
=(R\theta)^{\frac{3}{2}}\int_{\mathbb{R}^{3}}A_{1}(\frac{\xi-u}{\sqrt{R\theta}})\frac{\Theta_{1}}{M}\, d\xi,
\end{align}
and
\begin{align}
\label{5.16}
\int_{\mathbb{R}^{3}}\xi_{1}\xi_{i}L^{-1}_{M}\Theta_{1} \,d\xi=&
\int_{\mathbb{R}^{3}} L^{-1}_{M}\{P_{1}( \xi_{1}\xi_{i}M)\}\frac{\Theta_{1}}{M} \,d\xi
\nonumber\\
=&\int_{\mathbb{R}^{3}} L^{-1}_{M}\{R\theta\hat{B}_{1i}(\frac{\xi-u}{\sqrt{R\theta}})M\}\frac{\Theta_{1}}{M}\, d\xi
=R\theta\int_{\mathbb{R}^{3}}B_{1i}(\frac{\xi-u}{\sqrt{R\theta}})\frac{\Theta_{1}}{M} \,d\xi.
\end{align}
By using \eqref{5.15} and the expression of $J_{2}$, one has
\begin{align}
\label{2.18}
J_{2}=\int_{\mathbb{R}}\Big\{(\frac{\widetilde{\theta}}{\theta})_{x}
(R\theta)^{\frac{3}{2}}\int_{\mathbb{R}^{3}}A_{1}(\frac{\xi-u}{\sqrt{R\theta}})\frac{\Theta_{1}}{M}\, d\xi\Big\}\,dx.
\end{align}
   It follows from \eqref{1.22} that
\begin{equation*}
\Theta_{1}:=G_{t}-\frac{u_{1}}{v}G_{x}+\frac{1}{v}P_{1}(\xi_{1}G_{x})-Q(G,G).
\end{equation*}
For any multi-index $\beta$ and $b\geq 0$, by \eqref{5.4}, \eqref{1.38} and \eqref{2.12}, we have
\begin{equation}
\label{5.18}
\int_{\mathbb{R}^{3}}\frac{|\langle \xi\rangle^{b}\sqrt{\mu}\partial_{\beta}A_{1}(\frac{\xi-u}{\sqrt{R\theta}})|^{2}}{M^{2}}\,d\xi
+\int_{\mathbb{R}^{3}}\frac{|\langle \xi\rangle^{b}\sqrt{\mu}\partial_{\beta}B_{1i}(\frac{\xi-u}{\sqrt{R\theta}})|^{2}}{M^{2}}\,d\xi\leq C.
\end{equation}
By this and the similar expansion as \eqref{5.20}, we have from \eqref{5.38},
\eqref{1.29} and \eqref{2.9}  that
\begin{align}
\label{2.19}
 \int_{\mathbb{R}}\Big\{(\frac{\widetilde{\theta}}{\theta})_{x}&
(R\theta)^{\frac{3}{2}}\int_{\mathbb{R}^{3}}A_{1}(\frac{\xi-u}{\sqrt{R\theta}})\frac{\overline{G}_{t}}{M} \,d\xi\Big\}\,dx
\leq C(\|\widetilde{\theta}_{x}\|+\|\widetilde{\theta}\theta_{x}\|)
\times\|\frac{\overline{G}_{t}}{\sqrt{\mu}}\|
\nonumber\\
&\leq C(\|\widetilde{\theta}_{x}\|+\|\widetilde{\theta}\|_{L_{x}^{\infty}}\|\theta_{x}\|)
\times (\|[\bar{u}_{1xt},\bar{\theta}_{xt}]\|+\|[\bar{u}_{1x},\bar{\theta}_{x}]\cdot[v_{t},u_{t},\theta_{t}]\|)
\nonumber\\
&\leq C(\delta+\varepsilon_{0})\|[\widetilde{\theta}_{x},\widetilde{v}_{t},\widetilde{u}_{t},\widetilde{\theta}_{t}]\|^{2}
+C\delta(1+t)^{-\frac 43}.
\end{align}
For any $\epsilon>0$, by using \eqref{5.18}, \eqref{2.9}  and \eqref{1.35}, one has
\begin{align}
\label{2.20}
 \int_{\mathbb{R}}\Big\{(\frac{\widetilde{\theta}}{\theta})_{x}
(R\theta)^{\frac{3}{2}}\int_{\mathbb{R}^{3}}A_{1}(\frac{\xi-u}{\sqrt{R\theta}})\frac{\sqrt{\mu}\mathbf{g}_{t}}{M}\, d\xi\Big\}\,dx
&\leq C\{\|\widetilde{\theta}_{x}\|+\|\widetilde{\theta}\widetilde{\theta}_{x}\|+\|\widetilde{\theta}\bar{\theta}_{x}\|\}
\times\|\langle \xi\rangle^{-\frac{1}{2}}\mathbf{g}_{t}\|
\nonumber\\
&\leq C\epsilon\|\widetilde{\theta}_{x}\|^{2}+C_\epsilon\|\mathbf{g}_{t}\|^{2}_{\sigma}
+ C\delta\int_{\mathbb{R}}\widetilde{\theta}^{2}\omega^{2}\,dx,
\end{align}
where in the last inequality we have used the fact that
$$
\|\widetilde{\theta}\bar{\theta}_{x}\|^{2}\leq C\delta(1+t)^{-1}\int_{\mathbb{R}}e^{-\frac{c_{1}x^{2}}{1+t}}\widetilde{\theta}^{2}\,dx
\leq C\delta\int_{\mathbb{R}}\widetilde{\theta}^{2}\omega^{2}\,dx.
$$
Recalling that $G=\overline{G}+\sqrt{\mu}\mathbf{g}$, we deduce from \eqref{2.19} and \eqref{2.20} that
\begin{align}
\label{2.21}
&\int_{\mathbb{R}}\Big\{(\frac{\widetilde{\theta}}{\theta})_{x}
(R\theta)^{\frac{3}{2}}\int_{\mathbb{R}^{3}}A_{1}(\frac{\xi-u}{\sqrt{R\theta}})\frac{G_{t}}{M}\, d\xi\Big\}\,dx
\nonumber\\
&\leq C\epsilon\|\widetilde{\theta}_{x}\|^{2}+C_\epsilon\|\mathbf{g}_{t}\|^{2}_{\sigma}
+C(\delta+\varepsilon_{0})\|[\widetilde{\theta}_{x},\widetilde{v}_{t},\widetilde{u}_{t},\widetilde{\theta}_{t}]\|^{2}
+C\delta(1+t)^{-\frac 43}+C\delta\int_{\mathbb{R}}\widetilde{\theta}^{2}\omega^{2}\, dx.
\end{align}
By using \eqref{5.18} and the similar arguments as \eqref{2.21}, one has
\begin{align}
\label{2.21d}
&\int_{\mathbb{R}}\Big\{(\frac{\widetilde{\theta}}{\theta})_{x}
(R\theta)^{\frac{3}{2}}\int_{\mathbb{R}^{3}}A_{1}(\frac{\xi-u}{\sqrt{R\theta}})
\big\{\frac{1}{v}P_{1}(\xi_{1}G_{x})-\frac{u_{1}}{v}G_{x}\big\}
\frac{1}{M} \,d\xi\Big\}\,dx
\nonumber\\
&\leq C\epsilon\|\widetilde{\theta}_{x}\|^{2}
+C_{\epsilon}\|\mathbf{g}_{x}\|^{2}_{\sigma}
+C(\delta+\varepsilon_{0})\|[\widetilde{v}_{x},\widetilde{u}_{x},\widetilde{\theta}_{x}]\|^{2}
+C\delta(1+t)^{-\frac 43}+C\delta\int_{\mathbb{R}}\widetilde{\theta}^{2}\omega^{2}\,dx.
\end{align}
Recalling that $G=\overline{G}+\sqrt{\mu}\mathbf{g}$, by \eqref{2.5}, \eqref{5.7}, \eqref{5.18}, \eqref{1.35} and \eqref{5.19a}, we have
\begin{align}
\label{2.22}
&\int_{\mathbb{R}}\Big\{(\frac{\widetilde{\theta}}{\theta})_{x}
(R\theta)^{\frac{3}{2}}\int_{\mathbb{R}^{3}}A_{1}(\frac{\xi-u}{\sqrt{R\theta}})\frac{Q(G,G)}{M}\, d\xi\Big\}\,dx\notag\\
 &=\Big( \Gamma(\frac{G}{\sqrt{\mu}},\frac{G}{\sqrt{\mu}}), (\frac{\widetilde{\theta}}{\theta})_{x}
(R\theta)^{\frac{3}{2}} A_{1}(\frac{\xi-u}{\sqrt{R\theta}})\frac{\sqrt{\mu}}{M}\Big)
\nonumber\\
&\leq C\int_{\mathbb{R}}|(\frac{\widetilde{\theta}}{\theta})_{x}|
|\mu^{\varepsilon}\frac{G}{\sqrt{\mu}}|_{2}|\frac{G}{\sqrt{\mu}}|_{\sigma}\,dx
\nonumber\\
&\leq C\int_{\mathbb{R}}(|\widetilde{\theta}_{x}|+|\widetilde{\theta}\theta_{x}|)
(|\overline{u}_x|+|\overline{\theta}_x|+|\mathbf{g}|_{\sigma})^{2}\,dx
\nonumber\\
&\leq C(\delta+\varepsilon_{0})\|\widetilde{\theta}_{x}\|^{2}
+C\varepsilon_{0}\|\mathbf{g}\|^{2}_{\sigma}
+C\delta(1+t)^{-\frac 43}.
\end{align}
By using \eqref{2.18}, \eqref{2.21}, \eqref{2.21d} and \eqref{2.22}, we arrive at
\begin{align*}
J_{2}\leq C\epsilon\|\widetilde{\theta}_{x}\|^{2}+C_{\epsilon}\sum_{|\alpha|=1}\|\partial^{\alpha}\mathbf{g}\|^{2}_{\sigma}
&+C(\delta+\varepsilon_{0})\sum_{|\alpha|=1}\|\partial^{\alpha}[\widetilde{v},\widetilde{u},\widetilde{\theta}]\|^{2}
\\
& +C\varepsilon_{0}\|\mathbf{g}\|^{2}_{\sigma}
+C\delta(1+t)^{-\frac{4}{3}}+C\delta\int_{\mathbb{R}}\widetilde{\theta}^{2}\omega^{2}\,dx.
\end{align*}
The term  $J_{1}$ can be treated  similarly as $J_{2}$, we thereby have
\begin{align}\label{2.23a}
\int_{\mathbb{R}}\mathbb{H}\,dx &\leq C\epsilon(\|\widetilde{\theta}_{x}\|^{2}+\|\widetilde{u}_{x}\|^{2})+C_\epsilon\sum_{|\alpha|=1}\|\partial^{\alpha}\mathbf{g}\|^{2}_{\sigma}
+C(\delta+\varepsilon_{0})\sum_{|\alpha|=1}\|\partial^{\alpha}[\widetilde{v},\widetilde{u},\widetilde{\theta}]\|^{2}
\nonumber\\
&\qquad+C\varepsilon_{0}\|\mathbf{g}\|^{2}_{\sigma}
+C\delta(1+t)^{-\frac{4}{3}}+C\delta\int_{\mathbb{R}}\widetilde{\theta}^{2}\omega^{2}\,dx.
\end{align}
 Hence, integrating \eqref{2.13} about $x$ over $\mathbb{R}$ and taking a small $\epsilon>0$, we have from the above estimates that
\begin{align}
\label{2.23}
\frac{d}{dt}\int_{\mathbb{R}}\Big(\frac{2}{3}\bar{\theta}\Phi(\frac{v}{\bar{v}})&+\frac{1}{2}\widetilde{u}^{2}
+\bar{\theta}\Phi(\frac{\theta}{\bar{\theta}})\Big)dx+c\|[\widetilde{u}_{x},\widetilde{\theta}_{x}]\|^{2}\notag\\
&\leq C(\delta+\varepsilon_{0})\sum_{|\alpha|=1}\|\partial^{\alpha}[\widetilde{v},\widetilde{u},\widetilde{\theta}]\|^{2}
+C\varepsilon_{0}\|\mathbf{g}\|^{2}_{\sigma}
+C\sum_{|\alpha|=1}\|\partial^{\alpha}\mathbf{g}\|^{2}_{\sigma}\notag\\
&\qquad+C\delta(1+t)^{-\frac{4}{3}}
+C\delta\int_{\mathbb{R}}(\widetilde{v}^{2}+\widetilde{\theta}^{2})\omega^{2}\,dx.
\end{align}

Notice that there are no dissipation terms for $\widetilde{v}_{x}$ and $[\widetilde{v}_{t},\widetilde{u}_{t},\widetilde{\theta}_{t}]$ in \eqref{2.23}.
For these dissipation terms,  we   have from \eqref{1.19} and \eqref{1.28} that
\begin{equation}
\label{2.24}
\begin{cases}
\widetilde{v_{t}}-\widetilde{u}_{1x}=0,
\\
\widetilde{u}_{1t}+(\frac{R\widetilde{\theta}-p_{+}\widetilde{v}}{v})_{x}=-\bar{u}_{1t}-\int_{\mathbb{R}^{3}} \xi^{2}_{1}G_{x}\,d\xi,
\\
\widetilde{u}_{it}=-\int_{\mathbb{R}^{3}} \xi_{1}\xi_{i}G_{x}\,d\xi, ~~i=2,3,
\\
\widetilde{\theta}_{t}+\bar{\theta}_{t}+pu_{1x}
=-\int_{\mathbb{R}^{3}} \frac{1}{2}\xi_{1}|\xi|^{2}G_{x}d\xi+u\cdot\int_{\mathbb{R}^{3}}\xi \xi_{1}G_{x}\,d\xi.
\end{cases}
\end{equation}
We take the inner product of \eqref{2.24}$_{2}$  with $-\widetilde{v}_{x}$ with respect to $x$ over $\mathbb{R}$ to get
\begin{equation}
\label{2.25}
\big(\widetilde{u}_{1t},-\widetilde{v}_{x}\big)+\big(\frac{R}{v}\widetilde{\theta}_{x}-\frac{R\widetilde{\theta}}{v^2}v_{x}+\frac{p_+\widetilde{v}}{v^2}v_{x},-\widetilde{v}_{x}\big)
+\big(\frac{p_+}{v}\widetilde{v}_{x},\widetilde{v}_{x}\big)
=
 \big(\bar{u}_{1t},\widetilde{v}_{x}\big)+\big(\int_{\mathbb{R}^{3}} \xi^{2}_{1}G_{x}\,d\xi,\widetilde{v}_{x}\big).
\end{equation}
By using \eqref{2.24}$_{1}$ and the integration by parts, one has
$$
-(\widetilde{u}_{1t},\widetilde{v}_{x})=-(\widetilde{u}_{1},\widetilde{v}_{x})_{t}+(\widetilde{u}_{1},\widetilde{v}_{tx})
=-(\widetilde{u}_{1},\widetilde{v}_{x})_{t}-\|\widetilde{u}_{1x}\|^{2}.
$$
By using \eqref{1.27}, \eqref{1.28} \eqref{1.29} and \eqref{2.17}, we have
\begin{align*}
&|\big(\frac{R}{v}\widetilde{\theta}_{x}-\frac{R\widetilde{\theta}}{v^2}v_{x}+\frac{p_+\widetilde{v}}{v^2}v_{x},-\widetilde{v}_{x}\big)|
+|(\bar{u}_{1t},\widetilde{v}_{x})|
\\
&\leq (\epsilon+\delta+\varepsilon_0)\|\widetilde{v}_{x}\|^{2}+C_{\epsilon}\|\widetilde{\theta}_{x}\|^{2}
+C_{\epsilon}\delta(1+t)^{-\frac{4}{3}}+C\delta\int_{\mathbb{R}}(\widetilde{v}^{2}+\widetilde{\theta}^{2})\omega^{2}dx.
\end{align*}
Recalling that $G=\overline{G}+\sqrt{\mu}\mathbf{g}$, we have from \eqref{5.38}  and \eqref{1.29} that
\begin{align*}
|(\int_{\mathbb{R}^{3}} \xi^{2}_{1}G_{x}\,d\xi,\widetilde{v}_{x})|
&\leq C\int_{\mathbb{R}}|\widetilde{v}_{x}|\big\{|[\overline{u}_{xx},\overline{\theta}_{xx}]|
+|[\overline{u}_x,\overline{\theta}_x]\cdot[v_{x},u_{x},\theta_{x}]|+|\mathbf{g}_{x}|_{\sigma} \big\}\,dx
\\
&\leq (\epsilon+\delta)\|\widetilde{v}_{x}\|^{2}+C_{\epsilon}\|\mathbf{g}_{x}\|_{\sigma}^{2}
+C_{\epsilon}\delta\|[\widetilde{v}_{x},\widetilde{u}_{x},\widetilde{\theta}_{x}]\|^{2}
+C_{\epsilon}\delta(1+t)^{-\frac{4}{3}}.
\end{align*}
By taking $\epsilon$, $\delta$ and $\varepsilon_0$ small enough, by using \eqref{2.25} and the above estimates, we arrive at
\begin{equation}
\label{2.26}
-(\widetilde{u}_{1},\widetilde{v}_{x})_{t}+c\|\widetilde{v}_{x}\|^{2}
\leq C(\|\widetilde{u}_{x}\|^{2}+\|\widetilde{\theta}_{x}\|^{2}+\|\mathbf{g}_{x}\|_{\sigma}^{2})
+C\delta(1+t)^{-\frac{4}{3}}+C\delta\int_{\mathbb{R}}(\widetilde{v}^{2}+\widetilde{\theta}^{2})\omega^{2}dx.
\end{equation}
Taking the inner product of \eqref{2.24}$_{1}$, \eqref{2.24}$_{2}$, \eqref{2.24}$_{3}$, \eqref{2.24}$_{4}$ with $\widetilde{v}_{t}$,
$\widetilde{u}_{1t}$, $\widetilde{u}_{it}$ $(i=2,3)$, $\widetilde{\theta}_{t}$, respectively, we   arrive at
\begin{equation}
\label{2.27}
\|[\widetilde{v}_{t},\widetilde{u}_{t},\widetilde{\theta}_{t}]\|^{2}
\leq C\Big\{\|[\widetilde{v}_{x},\widetilde{u}_{x},\widetilde{\theta}_{x}]\|^{2}
+\|\mathbf{g}_{x}\|_{\sigma}^{2}+\delta(1+t)^{-\frac{4}{3}}+\delta\int_{\mathbb{R}}(\widetilde{v}^{2}+\widetilde{\theta}^{2})\omega^{2}\,dx\Big\}.
\end{equation}
For some small $\kappa_{1}>0$, we have from a suitable linear combination of \eqref{2.26} and \eqref{2.27} that
\begin{align*}
&-\kappa_{1}(\widetilde{u}_{1},\widetilde{v}_{x})_{t}+c\kappa_{1}(\|\widetilde{v}_{x}\|^{2}+\|[\widetilde{v}_{t},\widetilde{u}_{t},\widetilde{\theta}_{t}]\|^{2})\\
&\leq C\kappa_{1}\Big\{\|[\widetilde{u}_{x},\widetilde{\theta}_{x}]\|^{2}
+\|\mathbf{g}_{x}\|_{\sigma}^{2}+\delta(1+t)^{-\frac{4}{3}}+\delta\int_{\mathbb{R}}(\widetilde{v}^{2}+\widetilde{\theta}^{2})\omega^{2}\,dx\Big\}.
\end{align*}
If we choosing $\varepsilon_{0}$  and  $\delta$ small enough, we get from
this and \eqref{2.23} that
\begin{align}
\label{2.28}
&\frac{d}{dt}\int_{\mathbb{R}}\Big(\big(\frac{2}{3}\bar{\theta}\Phi(\frac{v}{\bar{v}})+\frac{1}{2}\widetilde{u}^{2}
+\bar{\theta}\Phi(\frac{\theta}{\bar{\theta}})\big)-\kappa_{1}\widetilde{u}_{1}\widetilde{v}_{x}\Big)\,dx
+c\sum_{|\alpha|=1}\|\partial^{\alpha}[\widetilde{v},\widetilde{u},\widetilde{\theta}]\|^{2}
\nonumber\\
&\leq C\sum_{|\alpha|=1}\|\partial^{\alpha}\mathbf{g}\|^{2}_{\sigma}+C\varepsilon_{0}\|\mathbf{g}\|^{2}_{\sigma}
+C\delta(1+t)^{-\frac{4}{3}}+C\delta\int_{\mathbb{R}}(\widetilde{v}^{2}+\widetilde{\theta}^{2})\omega^{2}\,dx.
\end{align}
This completes the lower order energy estimates for the macroscopic component $[ \widetilde{v}, \widetilde{u}, \widetilde{\theta}]$.

Then we turn to prove the lower order energy estimates for the microscopic component $\mathbf{g}$. We take the inner product of \eqref{2.4} with $\mathbf{g}$  over $\mathbb{R}\times\mathbb{R}^{3}$ to get
\begin{align}
\label{2.29}
&\frac{1}{2}\frac{d}{dt}(v\mathbf{g},\mathbf{g})+(\xi_{1}\mathbf{g}_{x},\mathbf{g})-\frac{1}{2}(v_{t}\mathbf{g},\mathbf{g})+\frac{1}{2}(u_{1x}\mathbf{g},\mathbf{g})-(v\mathcal{L}\mathbf{g},\mathbf{g})
\nonumber\\
&=(v\Gamma(\mathbf{g},\frac{M-\mu}{\sqrt{\mu}})+
v\Gamma(\frac{M-\mu}{\sqrt{\mu}},\mathbf{g}),\mathbf{g})+(v\Gamma(\frac{G}{\sqrt{\mu}},\frac{G}{\sqrt{\mu}}),\mathbf{g})
\nonumber\\
&\qquad+(\frac{P_{0}(\xi_{1}\sqrt{\mu}\mathbf{g}_{x})}{\sqrt{\mu}},\mathbf{g})-(\frac{1}{\sqrt{\mu}}P_{1}\xi_{1}M\big\{\frac{|\xi-u|^{2}
\widetilde{\theta}_{x}}{2R\theta^{2}}+\frac{(\xi-u)\cdot\widetilde{u}_{x}}{R\theta}\big\},\mathbf{g})
\nonumber\\
&\qquad-(\frac{P_{1}(\xi_{1}\overline{G}_{x})}{\sqrt{\mu}},\mathbf{g})+(u_{1}\frac{\overline{G}_{x}}{\sqrt{\mu}},\mathbf{g})
-(v\frac{\overline{G}_{t}}{\sqrt{\mu}},\mathbf{g}).
\end{align}
We will estimate each term in \eqref{2.29}. First note that the second term on the left hand side of \eqref{2.29} vanishes by integration by parts. The third and fourth terms can be cancelled by the fact that $v_{t}=u_{1x}$, namely
\begin{align*}
\frac{1}{2}(v_{t}\mathbf{g},\mathbf{g})-\frac{1}{2}(u_{1x}\mathbf{g},\mathbf{g})=0.
\end{align*}
From \eqref{2.6} and \eqref{2.12}, we get
$$
-(v\mathcal{L}\mathbf{g},\mathbf{g}) \geq  c_{2} \|v^{\frac{1}{2}}\mathbf{g}\|^{2}_{\sigma} \geq  c_{3} \|\mathbf{g}\|^{2}_{\sigma}.
$$
In view of \eqref{5.30} and \eqref{5.34},  it is  seen that
\begin{align*}
&|(v\Gamma(\mathbf{g},\frac{M-\mu}{\sqrt{\mu}})+
v\Gamma(\frac{M-\mu}{\sqrt{\mu}},\mathbf{g}) +(v\Gamma(\frac{G}{\sqrt{\mu}},\frac{G}{\sqrt{\mu}}),\mathbf{g})|
\\
&\leq C(\eta_{0}+\delta+\varepsilon_{0})\big(\|\mathbf{g}\|_{\sigma}^2+\mathcal{D}_{2,l,q}(t)\big)+C\delta(1+t)^{-\frac{4}{3}}.
\end{align*}
By \eqref{1.20a} and \eqref{1.35}, we obtain
\begin{align*}
&|(\frac{P_{0}(\xi_{1}\sqrt{\mu}\mathbf{g}_{x})}{\sqrt{\mu}},\mathbf{g})|=|\sum_{i=0}^{4}(\frac{1}{\sqrt{\mu}}\langle \xi_{1}\sqrt{\mu}\mathbf{g}_{x},\frac{\chi_{i}}{M}\rangle\chi_{i},\mathbf{g})|
\\
&\leq \epsilon\|\langle \xi\rangle^{-\frac{1}{2}}\mathbf{g}\|^{2}+C_{\epsilon}\|\langle \xi\rangle^{-\frac{1}{2}}\mathbf{g}_{x}\|^{2}
\leq C\epsilon\|\mathbf{g}\|_{\sigma}^{2}+C_{\epsilon}\|\mathbf{g}_{x}\|_{\sigma}^{2}.
\end{align*}
Here we used the fact that $|\langle \xi\rangle^b\mu^{-\frac{1}{2}}M|_2\leq C$  for any $b\ge0$ by \eqref{2.12}.

By \eqref{1.20a}, \eqref{5.1} and   a direct calculation, one has
\begin{align}
\label{2.35a}
P_{1}\xi_{1}M\big\{\frac{|\xi-u|^{2}
\widetilde{\theta}_{x}}{2R\theta^{2}}+\frac{(\xi-u)\cdot\widetilde{u}_{x}}{R\theta}\big\}
=\frac{\sqrt{R}\widetilde{\theta}_{x}}{\sqrt{\theta}}\hat{A}_{1}(\frac{\xi-u}{\sqrt{R\theta}})M
+\sum_{j=1}^{3}\widetilde{u}_{jx}\hat{B}_{1j}(\frac{\xi-u}{\sqrt{R\theta}})M.
\end{align}
It follows that
\begin{align*}
|(\frac{1}{\sqrt{\mu}}P_{1}\xi_{1}M\big\{\frac{|\xi-u|^{2}
\widetilde{\theta}_{x}}{2R\theta^{2}}+\frac{(\xi-u)\cdot\widetilde{u}_{x}}{R\theta}\big\},\mathbf{g})|
\leq C\epsilon\|\mathbf{g}\|_{\sigma}^{2}+C_{\epsilon}\|[\widetilde{u}_{x},\widetilde{\theta}_{x}]\|^{2}.
\end{align*}
By \eqref{1.20a}, \eqref{5.38},   \eqref{1.35} and \eqref{1.29}, we have
\begin{align*}
&|(\frac{P_{1}(\xi_{1}\overline{G}_{x})}{\sqrt{\mu}},\mathbf{g})|=|(\frac{\xi_{1}\overline{G}_{x}}{\sqrt{\mu}},\mathbf{g})
-(\frac{P_{0}(\xi_{1}\overline{G}_{x})}{\sqrt{\mu}},\mathbf{g})|
\\
&\leq C\int_{\mathbb{R}}\big\{|[\overline{u}_{xx},\overline{\theta}_{xx}]|
+|[\overline{u}_x,\overline{\theta}_x]\cdot[v_{x},u_{x},\theta_{x}]|\big\}|\langle \xi\rangle^{-\frac{1}{2}}\mathbf{g}|_{2}\,dx
\\
&\leq C\epsilon\|\mathbf{g}\|_{\sigma}^{2}+C_{\epsilon}\delta\|[\widetilde{v}_{x},\widetilde{u}_{x},\widetilde{\theta}_{x}]\|^{2}
+C_{\epsilon}\delta(1+t)^{-\frac{4}{3}}
\\
&\leq C\epsilon\|\mathbf{g}\|_{\sigma}^{2}+C_{\epsilon}\delta\mathcal{D}_{2,l,q}(t)
+C_{\epsilon}\delta(1+t)^{-\frac{4}{3}}.
\end{align*}
Similarly, it holds that
\begin{align*}
|(u_{1}\frac{\overline{G}_{x}}{\sqrt{\mu}},\mathbf{g})|
+|(v\frac{\overline{G}_{t}}{\sqrt{\mu}},\mathbf{g})|
\leq C\epsilon\|\mathbf{g}\|_{\sigma}^{2}+C_{\epsilon}\delta\mathcal{D}_{2,l,q}(t)
+C_{\epsilon}\delta(1+t)^{-\frac{4}{3}}.
\end{align*}
Substituting the above   estimates into \eqref{2.29} and taking a small $\epsilon>0$, one has
\begin{align}
\label{2.30}
\frac{d}{dt}\|v^{\frac{1}{2}}\mathbf{g}\|^{2}+c\|\mathbf{g}\|^{2}_{\sigma}
\leq C\|[\widetilde{u}_{x},\widetilde{\theta}_{x}]\|^{2}+C\|\mathbf{g}_{x}\|_{\sigma}^{2}
+C(\eta_{0}+\delta+\varepsilon_{0})\mathcal{D}_{2,l,q}(t)
+C\delta(1+t)^{-\frac{4}{3}}.
\end{align}
For some large $\widetilde{C}_1>0$,  we have from $\eqref{2.28}$ and $\eqref{2.30}$ that
\begin{multline}
\label{2.31}
\frac{d}{dt}\Big\{\int_{\mathbb{R}}\Big(\widetilde{C}_1\big(\frac{2}{3}\bar{\theta}\Phi(\frac{v}{\bar{v}})+\frac{1}{2}\widetilde{u}^{2}
+\bar{\theta}\Phi(\frac{\theta}{\bar{\theta}})\big)-\widetilde{C}_1\kappa_{1}\widetilde{u}_{1}\widetilde{v}_{x}\Big)\,dx
+\|v^{\frac{1}{2}}\mathbf{g}\|^{2}\Big\}\\
+c\sum_{|\alpha|=1}\|\partial^{\alpha}[\widetilde{v},\widetilde{u},\widetilde{\theta}]\|^{2}
+c\|\mathbf{g}\|^{2}_{\sigma}
\leq C\sum_{|\alpha|=1}\|\partial^{\alpha}\mathbf{g}\|^{2}_{\sigma}
+C\delta(1+t)^{-\frac{4}{3}}\\
+C(\eta_{0}+\delta+\varepsilon_{0})\mathcal{D}_{2,l,q}(t)+C\delta\int_{\mathbb{R}}(\widetilde{v}^{2}+\widetilde{\theta}^{2})\omega^{2}\,dx.
\end{multline}
This completes the proof of the lower order energy estimates.\qed

\subsection{High order energy estimates}\label{sub2.3}
In what follows we will deduce the derivative estimates for the solutions.
Applying $\partial_{x}$ to \eqref{2.7} yields
\begin{equation}
\label{2.33}
\begin{cases}
\widetilde{v}_{tx}-\widetilde{u}_{1xx}=0,
\\
\widetilde{u}_{1tx}+\frac{R\widetilde{\theta}_{xx}-p_{+}\widetilde{v}_{xx}}{v}=\frac{4}{3}\big(\frac{\mu(\theta)}{v}\widetilde{u}_{1xx}\big)_{x}+Q_{2}
-\int_{\mathbb{R}^{3}} \xi^{2}_{1}(L^{-1}_{M}\Theta_{1})_{xx} \,d\xi,
\\
\widetilde{u}_{itx}=\big(\frac{\mu(\theta)}{v}\widetilde{u}_{ixx}\big)_{x}+\big[(\frac{\mu(\theta)}{v})_{x}\widetilde{u}_{ix}\big]_{x}
-\int_{\mathbb{R}^{3}} \xi_{1}\xi_{i}(L^{-1}_{M}\Theta_{1})_{xx} \,d\xi, ~~i=2,3,
\\
\widetilde{\theta}_{tx}+p_{+}\widetilde{u}_{1xx}
=\big(\frac{\kappa(\theta)}{v}\widetilde{\theta}_{xx}\big)_{x}
+Q_{3}
\\
\hspace{1cm}+(u\cdot\int_{\mathbb{R}^{3}} \xi\xi_{1}(L^{-1}_{M}\Theta_{1})_{x}\,d\xi)_{x}
-\frac{1}{2}\int_{\mathbb{R}^{3}}\xi_{1}|\xi|^{2}(L^{-1}_{M}\Theta_{1})_{xx} \,d\xi,
\end{cases}
\end{equation}
where we have used the fact that
\begin{equation}
\label{2.34}
(p-p_{+})_{xx}=(\frac{R\widetilde{\theta}-p_{+}\widetilde{v}}{v})_{xx}
=\frac{R\widetilde{\theta}_{xx}-p_{+}\widetilde{v}_{xx}}{v}-\frac{2v_{x}}{v}(p-\bar{p})_{x}-\frac{v_{xx}}{v}(p-\bar{p}),
\end{equation}
and we have denoted that
\begin{equation}
\label{2.35}
Q_{2}=\frac{2v_{x}}{v}(p-\bar{p})_{x}+\frac{v_{xx}}{v}(p-\bar{p})
+\frac{4}{3}\big(\frac{\mu(\theta)}{v}\bar{u}_{1xx}\big)_{x}
+\frac{4}{3}\big[(\frac{\mu(\theta)}{v})_{x}u_{1x}\big]_{x}-\bar{u}_{1tx},
\end{equation}
and
\begin{equation}
\label{2.36}
Q_{3}=-p_{x}u_{1x}-(p-p_{+})u_{1xx}+
\big[(\frac{\kappa(\theta)}{v}-\frac{\kappa(\bar{\theta})}{\bar{v}})\bar{\theta}_{x}\big]_{xx}
+\big[(\frac{\kappa(\theta)}{v})_{x}\widetilde{\theta}_{x}\big]_{x}+Q_{1x}.
\end{equation}

Multiplying \eqref{2.33}$_{1}$ by $p_{+}\widetilde{v}_{x}$, \eqref{2.33}$_{2}$ by $v\widetilde{u}_{1x}$,  \eqref{2.33}$_{3}$ by $\widetilde{u}_{ix}$ $(i=2,3)$ and \eqref{2.33}$_{4}$ by $\frac{R}{p_{+}}\widetilde{\theta}_{x}$, then adding the resulting equations together, we   arrive at
\begin{align}
\label{2.37}
&\big(\frac{p_{+}}{2}\widetilde{v}^{2}_{x}+\frac{v}{2}\widetilde{u}^{2}_{1x}+\frac{1}{2}\sum^{3}_{i=2}\widetilde{u}^{2}_{ix}
+\frac{R}{2p_{+}}\widetilde{\theta}^{2}_{x}\big)_{t}
+\frac{4}{3}\mu(\theta)\widetilde{u}^{2}_{1xx}
+\sum_{i=2}^{3}\frac{\mu(\theta)}{v}\widetilde{u}_{ixx}^{2}+\frac{R}{p_{+}}\frac{\kappa(\theta)}{v}\widetilde{\theta}^{2}_{xx}+(\cdot\cdot\cdot)_{x}
\nonumber\\
&=\frac{v_{t}}{2}\widetilde{u}^{2}_{1x}-\frac{4}{3}\frac{\mu(\theta)}{v}\widetilde{u}_{1xx}v_{x}\widetilde{u}_{1x}
+\sum^{3}_{i=2}\big[(\frac{\mu(\theta)}{v})_{x}\widetilde{u}_{ix}\big]_{x}\widetilde{u}_{ix}+v\widetilde{u}_{1x}Q_{2}+\frac{R}{p_{+}}\widetilde{\theta}_{x}Q_{3}
+\mathbb{H}_{1},
\end{align}
where
\begin{align*}
\mathbb{H}_{1}&=-v\widetilde{u}_{1x}\int_{\mathbb{R}^{3}} \xi^{2}_{1}(L^{-1}_{M}\Theta_{1})_{xx} \,d\xi
-\sum^{3}_{i=2}\widetilde{u}_{ix}\int_{\mathbb{R}^{3}} \xi_{1}\xi_{i}(L^{-1}_{M}\Theta_{1})_{xx} \,d\xi
\nonumber\\
&\qquad+\frac{R}{p_{+}}\widetilde{\theta}_{x}(u\cdot\int_{\mathbb{R}^{3}} \xi\xi_{1}(L^{-1}_{M}\Theta_{1})_{x}\,d\xi)_{x}
-\frac{1}{2}\frac{R}{p_{+}}\widetilde{\theta}_{x}\int_{\mathbb{R}^{3}}\xi_{1}|\xi|^{2}(L^{-1}_{M}\Theta_{1})_{xx} \,d\xi
\nonumber\\
&=(\cdot\cdot\cdot)_{x}+(v\widetilde{u}_{1x})_{x}\big(\int_{\mathbb{R}^{3}} \xi^{2}_{1}L^{-1}_{M}\Theta_{1}\, d\xi\big)_{x}
+\sum^{3}_{i=2}\widetilde{u}_{ixx}\big(\int_{\mathbb{R}^{3}} \xi_{1}\xi_{i}(L^{-1}_{M}\Theta_{1}\, d\xi\big)_{x}
\nonumber\\
&\qquad+\frac{R}{p_{+}}\widetilde{\theta}_{xx}u_{x}\cdot\int_{\mathbb{R}^{3}} \xi\xi_{1}L^{-1}_{M}\Theta_{1}\,d\xi
+\frac{R}{p_{+}}\widetilde{\theta}_{xx}\big(\int_{\mathbb{R}^{3}}(\frac{1}{2}\xi_{1}|\xi|^{2}-u\cdot\xi\xi_{1})L^{-1}_{M}\Theta_{1}\,d\xi\big)_{x}.
\end{align*}
We shall estimate each term in \eqref{2.37}. By \eqref{1.29} and \eqref{2.9}, the imbedding inequality and the integration by parts, we obtain
\begin{align*}
&\int_{\mathbb{R}}\big(\frac{v_{t}}{2}\widetilde{u}^{2}_{1x}-\frac{4}{3}\frac{\mu(\theta)}{v}\widetilde{u}_{1xx}v_{x}\widetilde{u}_{1x}
+\sum^{3}_{i=2}\big[(\frac{\mu(\theta)}{v})_{x}\widetilde{u}_{ix}\big]_{x}\widetilde{u}_{ix}\big)\,dx
\\
&\leq C\|v_{t}\|_{L^{\infty}_{x}}\|\widetilde{u}_{1x}\|^{2}
+C(\|v_{x}\|_{L^{\infty}_{x}}+\|\theta_{x}\|_{L^{\infty}_{x}})\|\widetilde{u}_{x}\|\|\widetilde{u}_{xx}\|
\\
&\leq C(\delta+\varepsilon_{0})(\|\widetilde{u}_{x}\|^{2}+\|\widetilde{u}_{xx}\|^{2})
\leq C(\delta+\varepsilon_{0})\mathcal{D}_{2,l,q}(t).
\end{align*}
By the expressions of $Q_{2}$ and $Q_{3}$ in \eqref{2.35} and \eqref{2.36},   we can obtain
\begin{align*}
\int_{\mathbb{R}}\big(v\widetilde{u}_{1x}Q_{2}+\frac{R}{p_{+}}\widetilde{\theta}_{x}Q_{3}\big)\, dx
\leq C(\delta+\varepsilon_{0})\mathcal{D}_{2,l,q}(t)+C\delta(1+t)^{-\frac{4}{3}}.
\end{align*}
To estimate  $\mathbb{H}_{1}$, we only consider the second term in $\mathbb{H}_{1}$ while the last three terms in $\mathbb{H}_{1}$ can be treated similarly.
For the second term in $\mathbb{H}_{1}$, by \eqref{5.16}, one has
\begin{align}
\label{2.39a}
\int_{\mathbb{R}}(v\widetilde{u}_{1x})_{x}&\big(\int_{\mathbb{R}^{3}} \xi^{2}_{1}L^{-1}_{M}\Theta_{1} d\xi\big)_{x} dx
=\int_{\mathbb{R}}(v\widetilde{u}_{1x})_{x}\big[\int_{\mathbb{R}^{3}}R\theta B_{11}(\frac{\xi-u}{\sqrt{R\theta}})\frac{\Theta_{1}}{M} d\xi\big]_{x}\, dx
\nonumber\\
&=\int_{\mathbb{R}}(v\widetilde{u}_{1x})_{x}\int_{\mathbb{R}^{3}}\big[ R\theta B_{11}(\frac{\xi-u}{\sqrt{R\theta}})\frac{1}{M}\big]_{x}\Theta_{1} d\xi dx\notag\\
&\qquad+\int_{\mathbb{R}}(v\widetilde{u}_{1x})_{x}\int_{\mathbb{R}^{3}}R\theta B_{11}(\frac{\xi-u}{\sqrt{R\theta}})\frac{\Theta_{1x}}{M} \,d\xi dx.
\end{align}
By using \eqref{5.18} and the similar arguments as \eqref{2.23a}, we have from \eqref{2.39a} that
\begin{equation*}
\int_{\mathbb{R}}(v\widetilde{u}_{1x})_{x}\big(\int_{\mathbb{R}^{3}} \xi^{2}_{1} L^{-1}_{M}\Theta_{1} d\xi\big)_{x}\, dx\leq C\epsilon\|\widetilde{u}_{xx}\|^{2}+C_\epsilon\sum_{|\alpha|=2}\|\partial^{\alpha}\mathbf{g}\|^{2}_{\sigma}
+C(\delta+\varepsilon_{0})\mathcal{D}_{2,l,q}(t)
+C\delta(1+t)^{-\frac{4}{3}}.
\end{equation*}
The other terms in $\mathbb{H}_{1}$ can be handled in the same manner and it shares the similar bound. It holds that
\begin{align*}
\int_{\mathbb{R}}\mathbb{H}_{1} \, dx\leq C\epsilon(\|\widetilde{u}_{xx}\|^{2}+\|\widetilde{\theta}_{xx}\|^{2})+C_\epsilon\sum_{|\alpha|=2}\|\partial^{\alpha}\mathbf{g}\|^{2}_{\sigma}
+C(\delta+\varepsilon_{0})\mathcal{D}_{2,l,q}(t)
+C\delta(1+t)^{-\frac{4}{3}}.
\end{align*}
For $\epsilon>0$ small enough, integrating \eqref{2.37} with respect to $x$ over $\mathbb{R}$ and using the above   estimates, we arrive at
\begin{align}
\label{2.39}
&\frac{d}{dt}\int_{\mathbb{R}}\big(\frac{p_{+}}{2}\widetilde{v}^{2}_{x}+\frac{v}{2}\widetilde{u}^{2}_{1x}
+\frac{1}{2}\sum^{3}_{i=2}\widetilde{u}^{2}_{ix}
+\frac{R}{2p_{+}}\widetilde{\theta}^{2}_{x}\big)\,dx
+c(\|\widetilde{u}_{xx}\|^{2}+\|\widetilde{\theta}_{xx}\|^{2})
\nonumber\\
&\leq C\sum_{|\alpha|=2}\|\partial^{\alpha}\mathbf{g}\|^{2}_{\sigma}
+C(\delta+\varepsilon_{0})\mathcal{D}_{2,l,q}(t)+C\delta(1+t)^{-\frac{4}{3}}.
\end{align}
By the similar arguments as \eqref{2.39}, we can obtain
\begin{align}
\label{2.40}
&\frac{d}{dt}\int_{\mathbb{R}}\big(\frac{p_{+}}{2}\widetilde{v}^{2}_{t}+\frac{v}{2}\widetilde{u}^{2}_{1t}+\frac{1}{2}\sum^{3}_{i=2}\widetilde{u}^{2}_{it}
+\frac{R}{2p_{+}}\widetilde{\theta}^{2}_{t}\big)dx
+c(\|\widetilde{u}_{xt}\|^{2}+\|\widetilde{\theta}_{xt}\|^{2})
\nonumber\\
&\leq C\sum_{|\alpha|=2}\|\partial^{\alpha}\mathbf{g}\|^{2}_{\sigma}
+C(\delta+\varepsilon_{0})\mathcal{D}_{2,l,q}(t)+C\delta(1+t)^{-\frac{4}{3}}.
\end{align}

Note that  there are no dissipation terms for $[\widetilde{v}_{xx},\widetilde{v}_{tx}]$ and $[\widetilde{\rho}_{tt},\widetilde{u}_{tt},\widetilde{\theta}_{tt}]$
in \eqref{2.39} and \eqref{2.40}. To get the dissipation term $\widetilde{v}_{xx}$, applying $\partial_{x}$ to \eqref{2.24}$_{2}$ and using \eqref{2.34}, we have
\begin{align*}
\widetilde{u}_{1xt}+\frac{R\widetilde{\theta}_{xx}-p_{+}\widetilde{v}_{xx}}{v}-\frac{2v_{x}}{v}(p-\bar{p})_{x}-\frac{v_{xx}}{v}(p-\bar{p})=-\bar{u}_{1xt}-\int_{\mathbb{R}^{3}} \xi^{2}_{1}G_{xx}\,d\xi.
\end{align*}
By taking the inner product of the above equation with $-\widetilde{v}_{xx}$, we   arrive at
\begin{equation}
\label{2.41}
-(\widetilde{u}_{1x},\widetilde{v}_{xx})_{t}+c\|\widetilde{v}_{xx}\|^{2}
\leq C(\|\widetilde{u}_{xx}\|^{2}+\|\widetilde{\theta}_{xx}\|^{2}+\|\mathbf{g}_{xx}\|_{\sigma}^{2})
+C(\delta+\varepsilon_{0})\mathcal{D}_{2,l,q}(t)+C\delta(1+t)^{-\frac{4}{3}}.
\end{equation}
Applying $\partial_{x}$ to \eqref{2.24}$_{1}$ and taking the inner product of the resulting equation with $\widetilde{v}_{tx}$ yields
\begin{align}
\label{2.42}
\|\widetilde{v}_{tx}\|^{2}\leq C\|\widetilde{u}_{1xx}\|^{2}.
\end{align}
By using the system \eqref{2.24} again, one can arrive at
\begin{align}
\label{2.43}
\|[\widetilde{v}_{tt},\widetilde{u}_{tt},\widetilde{\theta}_{tt}]\|^{2}
\leq C\|[\widetilde{v}_{tx},\widetilde{u}_{tx},
\widetilde{\theta}_{tx}]\|^{2}+C\|\mathbf{g}_{tx}\|_{\sigma}^{2}
+C(\delta+\varepsilon_{0})\mathcal{D}_{2,l,q}(t)+C\delta(1+t)^{-\frac{4}{3}}.
\end{align}

For some   large constant $\widetilde{C}_{2}>0$, we have from
a suitable linear combination of \eqref{2.39}, \eqref{2.40}, \eqref{2.41}, \eqref{2.42} and  \eqref{2.43} that
\begin{multline}
\label{2.44}
\frac{d}{dt}\Big\{\widetilde{C}_{2}\sum_{|\alpha|=1}
\int_{\mathbb{R}}\big(\frac{p_{+}}{2}|\partial^{\alpha}\widetilde{v}|^{2}+\frac{v}{2}|\partial^{\alpha}\widetilde{u}_{1}|^{2}
+\frac{1}{2}\sum^{3}_{i=2}|\partial^{\alpha}\widetilde{u}_{i}|^{2}+\frac{R}{2p_{+}}|\partial^{\alpha}\widetilde{\theta}|^{2}\big)\,dx
-(\widetilde{u}_{1x},\widetilde{v}_{xx})\Big\}\\
+c\sum_{|\alpha|=2}\|\partial^{\alpha}[\widetilde{v},\widetilde{u},\widetilde{\theta}]\|^{2}
\leq C\sum_{|\alpha|=2}\|\partial^{\alpha}\mathbf{g}\|^{2}_{\sigma}
+C(\delta+\varepsilon_{0})\mathcal{D}_{2,l,q}(t)+C\delta(1+t)^{-\frac{4}{3}}.
\end{multline}

Next  we   deduce the derivative estimates for the microscopic component $\mathbf{g}$.
Applying $\partial^{\alpha}$ to \eqref{2.4} with $|\alpha|=1$ and taking the inner product with $\partial^{\alpha}\mathbf{g}$,
we can obtain
\begin{align}
\label{2.46}
&\frac{1}{2}(v\partial^{\alpha}\mathbf{g},\partial^{\alpha}\mathbf{g})_{t}-\frac{1}{2}(v_{t}\partial^{\alpha}\mathbf{g},\partial^{\alpha}\mathbf{g})
+\frac{1}{2}(u_{1x}\partial^{\alpha}\mathbf{g},\partial^{\alpha}\mathbf{g})
+(\partial^{\alpha}v\mathbf{g}_{t},\partial^{\alpha}\mathbf{g})
-(\partial^{\alpha}u_{1}\mathbf{g}_{x},\partial^{\alpha}\mathbf{g})
\nonumber\\
&=\big(\partial^{\alpha}(v\mathcal{L}\mathbf{g}),\partial^{\alpha}\mathbf{g}\big)+
\big(\partial^{\alpha}[v\Gamma(\mathbf{g},\frac{M-\mu}{\sqrt{\mu}})]+
\partial^{\alpha}[v\Gamma(\frac{M-\mu}{\sqrt{\mu}},\mathbf{g})]+\partial^{\alpha}[v\Gamma(\frac{G}{\sqrt{\mu}},\frac{G}{\sqrt{\mu}})],\partial^{\alpha}\mathbf{g}\big)
\nonumber\\
&\qquad+\big(\partial^{\alpha}\big[\frac{P_{0}(\xi_{1}\sqrt{\mu}\mathbf{g}_{x})}{\sqrt{\mu}}\big],\partial^{\alpha}\mathbf{g}\big)
-\big(\partial^{\alpha}\big[\frac{1}{\sqrt{\mu}}P_{1}\xi_{1}M\big\{\frac{|\xi-u|^{2}\widetilde{\theta}_{x}}{2R\theta^{2}}
+\frac{(\xi-u)\cdot\widetilde{u}_{x}}{R\theta}\big\}\big],\partial^{\alpha}\mathbf{g}\big)
\nonumber\\
&\qquad-\big(\partial^{\alpha}\big[\frac{P_{1}(\xi_{1}\overline{G}_{x})}{\sqrt{\mu}}\big],\partial^{\alpha}\mathbf{g}\big)
+\big(\partial^{\alpha}[u_{1}\frac{\overline{G}_{x}}{\sqrt{\mu}}],\partial^{\alpha}\mathbf{g}\big)
-\big(\partial^{\alpha}[v\frac{\overline{G}_{t}}{\sqrt{\mu}}],\partial^{\alpha}\mathbf{g}\big).
\end{align}
We shall estimate each term in \eqref{2.46}. First note that
$$
-\frac{1}{2}(v_{t}\partial^{\alpha}\mathbf{g},\partial^{\alpha}\mathbf{g})
+\frac{1}{2}(u_{1x}\partial^{\alpha}\mathbf{g},\partial^{\alpha}\mathbf{g})=0.
$$
In order to control the fourth and fifth terms,
we first estimate the following two terms.
For  any $|\alpha|=1$, we have from \eqref{1.33}, \eqref{1.35} and the imbedding inequality that
\begin{equation}
\label{2.48a}
 |(\partial^\alpha u_1\mathbf{f}_1,\mathbf{f}_2)|\leq \epsilon\|\langle\xi\rangle^{-\frac{1}{2}}\mathbf{f}_2\|^2+C_\epsilon\|\partial^\alpha u_1\|_{L^\infty_x}^2\|\langle\xi\rangle^{\frac{1}{2}}\mathbf{f}_1\|^2\leq C\epsilon\|\mathbf{f}_2\|^2_\sigma+C_\epsilon q_3(t)\|\langle\xi\rangle^{\frac{1}{2}}\mathbf{f}_1\|^2.
\end{equation}
By using \eqref{1.29}, \eqref{2.9} and the imbedding inequality, one has
\begin{align*}
 \sum_{|\alpha|=1}\|\partial^{\alpha}v\|^{3}_{L^{\infty}_x}
&\leq C(\|\bar{v}_{x}\|^{3}_{L^{\infty}_x}+\|\bar{v}_{t}\|^{3}_{L^{\infty}_x}+\sum_{|\alpha|=1}\|\partial^{\alpha} \widetilde{v}\|^{3}_{L^{\infty}_x}) \nonumber\\
 &\leq C(\|\bar{v}_{x}\|^{3}_{L^{\infty}_x}+\|\bar{v}_{t}\|^{2}+\|\bar{v}_{tx}\|^{2}+\sum_{1\leq |\alpha|\leq 2}\|\partial^{\alpha} \widetilde{v}\|^2)\leq Cq_3(t).
\end{align*}
For  any $|\alpha|=1$, by using this, \eqref{1.35}, \eqref{1.33}, the H\"{o}lder inequality and the imbedding inequality, one has
\begin{align}
\label{2.48aa}
 |(\partial^\alpha v\mathbf{f}_1,\mathbf{f}_2)|&\leq \epsilon\|\langle\xi\rangle^{-\frac{1}{2}}\mathbf{f}_2\|^2+ \frac{1}{4\epsilon}\|\partial^\alpha v\|_{L^\infty_x}^2\|\langle\xi\rangle^{\frac{1}{2}}\mathbf{f}_1\|^2 \nonumber\\
 &\leq \epsilon\|\langle\xi\rangle^{-\frac{1}{2}}\mathbf{f}_2\|^2+C_{\epsilon}\|\partial^\alpha v\|_{L^\infty_x}^2\|\langle\xi\rangle^{-\frac{1}{6}}|\mathbf{f}_1|^{\frac{1}{3}}\|^2_{L^6}\|\langle\xi\rangle^{\frac{2}{3}}|\mathbf{f}_1|^{\frac{2}{3}}\|^2_{L^3} \nonumber\\
&\leq \epsilon\|\langle\xi\rangle^{-\frac{1}{2}}\mathbf{f}_2\|^2+\epsilon\|\langle\xi\rangle^{-\frac{1}{2}}\mathbf{f}_1\|^2+C_\epsilon\|\partial^\alpha v\|_{L^\infty_x}^3\|\langle\xi\rangle^{\frac{2}{3}}|\mathbf{f}_1|^{\frac{2}{3}}\|^3_{L^3} \nonumber\\
&\leq C\epsilon(\|\mathbf{f}_2\|^2_\sigma+\|\mathbf{f}_1\|^2_\sigma)+C_\epsilon\|\partial^\alpha v\|_{L^\infty_x}^3\|\langle\xi\rangle  \mathbf{f}_1 \|^2 \nonumber\\
&\leq C\epsilon(\|\mathbf{f}_2\|^2_\sigma+\|\mathbf{f}_1\|^2_\sigma)+C_\epsilon q_3(t)\|\langle\xi\rangle  \mathbf{f}_1 \|^2.
\end{align}
By  \eqref{2.48a} and \eqref{2.48aa}, we arrive at
\begin{align}
\label{2.48aaa}
 \sum_{|\alpha|=1} |(\partial^\alpha u_1\mathbf{f}_1,\mathbf{f}_2)|+\sum_{|\alpha|=1}|(\partial^\alpha v\mathbf{f}_1,\mathbf{f}_2)|\leq C\epsilon(\|\mathbf{f}_2\|^2_\sigma+\|\mathbf{f}_1\|^2_\sigma)+C_\epsilon q_3(t)\|\langle\xi\rangle  \mathbf{f}_1 \|^2.
\end{align}
We define the functional as
\begin{align}
\label{2.48}
\mathcal{F}_{2,l,q}(t):=\sum_{|\alpha|+|\beta|\leq 2}
\|\langle \xi\rangle\partial^{\alpha}_{\beta} \mathbf{g}(t)\|^{2}_{w(\beta)}.
\end{align}
By \eqref{2.48aaa}  and \eqref{2.48}, we have
\begin{align*}
\big|&(\partial^{\alpha}v\mathbf{g}_{t},\partial^{\alpha}\mathbf{g})
-(\partial^{\alpha}u_{1}\mathbf{g}_{x},\partial^{\alpha}\mathbf{g})\big|
\nonumber\\
&\leq C\epsilon\sum_{|\alpha|=1}\|\partial^{\alpha}\mathbf{g}\|_\sigma^2+C_\epsilon q_3(t)\sum_{|\alpha|=1}\|\langle \xi\rangle\partial^{\alpha}\mathbf{g}\|^2
\nonumber\\
&\leq C\epsilon\sum_{|\alpha|=1}\|\partial^{\alpha}\mathbf{g}\|_\sigma^2+C_\epsilon q_3(t)\mathcal{F}_{2,l,q}(t).
\end{align*}
For any $|\alpha|=1$, we easily see
$$
\big(\partial^{\alpha}(v\mathcal{L}\mathbf{g}),\partial^{\alpha}\mathbf{g}\big)=\big(v\mathcal{L}\partial^{\alpha}\mathbf{g},\partial^{\alpha}\mathbf{g}\big)
+\big(\partial^{\alpha}v\mathcal{L}\mathbf{g},\partial^{\alpha}\mathbf{g}\big).
$$
From \eqref{2.6} and \eqref{2.12}, we get
$$
(v\mathcal{L}\partial^{\alpha}\mathbf{g},\partial^{\alpha}\mathbf{g})\leq - c_{2} \|v^{\frac{1}{2}}\partial^{\alpha}\mathbf{g}\|^{2}_{\sigma}\leq - \frac{c_{2}}{2}\|\partial^{\alpha}\mathbf{g}\|^{2}_{\sigma}.
$$
Recalling that $\mathcal{L}\mathbf{g}=\Gamma(\mathbf{g},\sqrt{\mu})+\Gamma(\sqrt{\mu}, \mathbf{g})$, we have from \eqref{5.7}, \eqref{2.9} and \eqref{1.29} that
\begin{align*}
(\partial^{\alpha}v\mathcal{L}\mathbf{g},\partial^{\alpha}\mathbf{g})&\leq C(\|\partial^{\alpha}\widetilde{v}\|_{L_{x}^{\infty}}
+\|\partial^{\alpha}\bar{v}\|_{L_{x}^{\infty}})\|\mathbf{g}\|_{\sigma}
\|\partial^{\alpha}\mathbf{g}\|_{\sigma}
\\
&\leq C(\delta+\varepsilon_{0})(\|\mathbf{g}\|^{2}_{\sigma}+\|\partial^{\alpha}\mathbf{g}\|^{2}_{\sigma}).
\end{align*}
It follows that
\begin{align*}
\big(\partial^{\alpha}(v\mathcal{L}\mathbf{g}),\partial^{\alpha}\mathbf{g}\big)\leq -\frac{c_{2}}{2}\|\partial^{\alpha}\mathbf{g}\|^{2}_{\sigma}+C(\delta+\varepsilon_{0})(\|\mathbf{g}\|^{2}_{\sigma}+\|\partial^{\alpha}\mathbf{g}\|^{2}_{\sigma}).
\end{align*}
By using \eqref{5.30} and \eqref{5.34}, one has
\begin{align*}
&\big|\big(\partial^{\alpha}[v\Gamma(\mathbf{g},\frac{M-\mu}{\sqrt{\mu}})]+
\partial^{\alpha}[v\Gamma(\frac{M-\mu}{\sqrt{\mu}},\mathbf{g})]
+\partial^{\alpha}[v\Gamma(\frac{G}{\sqrt{\mu}},\frac{G}{\sqrt{\mu}})],\partial^{\alpha}\mathbf{g}\big)\big|
\\
&\leq C(\eta_0+\delta+\varepsilon_{0})\big(\|\partial^{\alpha}\mathbf{g}\|_{\sigma}^2+\mathcal{D}_{2,l,q}(t)\big)
+C\delta(1+t)^{-\frac{4}{3}}.
\end{align*}
By using \eqref{1.35}, \eqref{1.20a}, \eqref{2.9}, \eqref{1.29} and  the imbedding inequality, we have
\begin{align*}
&\big|\big(\partial^{\alpha}\big[\frac{P_{0}(\xi_{1}\sqrt{\mu}\mathbf{g}_{x})}{\sqrt{\mu}}\big],\partial^{\alpha}\mathbf{g}\big)\big|
+\big|\big(\partial^{\alpha}\big[\frac{1}{\sqrt{\mu}}P_{1}\xi_{1}M\big\{\frac{|\xi-u|^{2}\widetilde{\theta}_{x}}{2R\theta^{2}}
+\frac{(\xi-u)\cdot\widetilde{u}_{x}}{R\theta}\big\}\big],\partial^{\alpha}\mathbf{g}\big)\big|
\\
&\leq \epsilon\|\langle \xi\rangle^{-\frac{1}{2}}\partial^{\alpha}\mathbf{g}\|^{2}+C_{\epsilon}\|\langle \xi\rangle^{\frac{1}{2}}\partial^{\alpha}\big[\frac{P_{0}(\xi_{1}\sqrt{\mu}\mathbf{g}_{x})}{\sqrt{\mu}}\big]\|^{2}
\\
&\qquad+C_{\epsilon}\|\langle \xi\rangle^{\frac{1}{2}}\partial^{\alpha}\big[\frac{1}{\sqrt{\mu}}P_{1}\xi_{1}M\big\{\frac{|\xi-u|^{2}\widetilde{\theta}_{x}}{2R\theta^{2}}
+\frac{(\xi-u)\cdot\widetilde{u}_{x}}{R\theta}\big\}\big]\|^{2}
\\
&\leq C\epsilon\|\partial^{\alpha}\mathbf{g}\|_{\sigma}^2+C_{\epsilon}\|\partial^{\alpha}\mathbf{g}_{x}\|_{\sigma}^2+C_\epsilon
\|[\partial^{\alpha}\widetilde{u}_{x},\partial^{\alpha}\widetilde{\theta}_{x}]\|^2
+C_\epsilon(\delta+\varepsilon_{0})\mathcal{D}_{2,l,q}(t).
\end{align*}
Here we used the fact that $|\langle \xi\rangle^b\mu^{-\frac{1}{2}}M|_2\leq C$  for any $b\ge0$ by \eqref{2.12}.

By   \eqref{5.38}, \eqref{1.35} and \eqref{1.20a}, one has
\begin{align*}
&\big|\big(\partial^{\alpha}\big[\frac{P_{1}(\xi_{1}\overline{G}_{x})}{\sqrt{\mu}}\big],\partial^{\alpha}\mathbf{g}\big)\big|
+\big|\big(\partial^{\alpha}[u_{1}\frac{\overline{G}_{x}}{\sqrt{\mu}}],\partial^{\alpha}\mathbf{g}\big)\big|
+\big|\big(\partial^{\alpha}[v\frac{\overline{G}_{t}}{\sqrt{\mu}}],\partial^{\alpha}\mathbf{g}\big)\big|
\\
&\leq \epsilon\|\langle \xi\rangle^{-\frac{1}{2}}\partial^{\alpha}\mathbf{g}\|^{2}+C_\epsilon\|\langle \xi\rangle^{\frac{1}{2}}
\partial^{\alpha}\big[\frac{P_{1}(\xi_{1}\overline{G}_{x})}{\sqrt{\mu}}\big]\|^{2}
+C_\epsilon\|\langle \xi\rangle^{\frac{1}{2}}\partial^{\alpha}[u_{1}\frac{\overline{G}_{x}}{\sqrt{\mu}}]\|^{2}
+C_\epsilon\|\langle \xi\rangle^{\frac{1}{2}}\partial^{\alpha}[v\frac{\overline{G}_{t}}{\sqrt{\mu}}]\|^{2}
\\
&\leq C\epsilon\|\partial^{\alpha}\mathbf{g}\|_{\sigma}^2+C_\epsilon(\delta+\varepsilon_{0})\mathcal{D}_{2,l,q}(t)+C_\epsilon\delta(1+t)^{-\frac{4}{3}}.
\end{align*}
Substituting the above   estimates into \eqref{2.46} and taking a small $\epsilon>0$, one has
\begin{multline}
\label{2.49}
\frac{1}{2}\frac{d}{dt}\sum_{|\alpha|=1}\|v^{\frac{1}{2}}\partial^{\alpha}\mathbf{g}\|^{2}+c\sum_{|\alpha|=1}\|\partial^{\alpha}\mathbf{g}\|^{2}_{\sigma}
\leq C\sum_{|\alpha|=1}\big(\|\partial^{\alpha}\mathbf{g}_{x}\|^{2}_{\sigma}+\|[\partial^{\alpha}\widetilde{u}_{x},\partial^{\alpha}\widetilde{\theta}_{x}]\|^2\big)\\
+C(\eta_{0}+\delta+\varepsilon_{0})\mathcal{D}_{2,l,q}(t)+C\delta(1+t)^{-\frac{4}{3}}+Cq_{3}(t)\mathcal{F}_{2,l,q}(t).
\end{multline}
Here $\mathcal{D}_{2,l,q}(t)$, $q_{3}(t)$ and $\mathcal{F}_{2,l,q}(t)$ given by \eqref{2.11}, \eqref{1.33} and \eqref{2.48}, respectively.

Finally, we will deduce the second order derivative  estimates for the microscopic component $\mathbf{g}$.
By \eqref{1.16} and \eqref{2.5}, one has
\begin{align}
\label{2.50}
v(\frac{F}{\sqrt{\mu}})_{t}-u_{1}(\frac{F}{\sqrt{\mu}})_{x}
+\xi_{1}(\frac{F}{\sqrt{\mu}})_{x}
&=v\mathcal{L}\mathbf{g}+v\Gamma(\mathbf{g},\frac{M-\mu}{\sqrt{\mu}})+
v\Gamma(\frac{M-\mu}{\sqrt{\mu}},\mathbf{g})
+v\Gamma(\frac{G}{\sqrt{\mu}},\frac{G}{\sqrt{\mu}})
\nonumber\\
&\qquad+\frac{1}{\sqrt{\mu}}P_1\xi_1M\big\{\frac{|\xi-u|^2\overline{\theta}_x}{2R\theta^2}
+\frac{(\xi-u)\cdot \overline{u}_{x}}{R\theta}\big\}.
\end{align}
Applying $\partial^{\alpha}$ to \eqref{2.50} with $|\alpha|=2$ and taking the inner product with $\frac{\partial^{\alpha}F}{\sqrt{\mu}}$,
one has
\begin{align}
\label{2.51}
&\frac{1}{2}\big(v\frac{\partial^{\alpha}F}{\sqrt{\mu}},\frac{\partial^{\alpha}F}{\sqrt{\mu}}\big)_{t}
+\big(\xi_{1}(\frac{\partial^{\alpha}F}{\sqrt{\mu}})_{x},\frac{\partial^{\alpha}F}{\sqrt{\mu}}\big)
-\frac{1}{2}\big(v_{t}\frac{\partial^{\alpha}F}{\sqrt{\mu}},\frac{\partial^{\alpha}F}{\sqrt{\mu}}\big)
-\big(u_{1}(\frac{\partial^{\alpha}F}{\sqrt{\mu}})_{x},\frac{\partial^{\alpha}F}{\sqrt{\mu}}\big)
\nonumber\\
&\qquad+\sum_{1\leq|\alpha_{1}|\leq|\alpha|}\big(C^{\alpha_{1}}_{\alpha}\big[\partial^{\alpha_{1}}v\frac{\partial^{\alpha-\alpha_{1}}F_{t}}{\sqrt{\mu}}
-\partial^{\alpha_{1}}u_{1}\frac{\partial^{\alpha-\alpha_{1}}F_{x}}{\sqrt{\mu}}\big],\frac{\partial^{\alpha}F}{\sqrt{\mu}}\big)
\nonumber\\
&=\big(\partial^{\alpha}(v\mathcal{L}\mathbf{g}),\frac{\partial^{\alpha}F}{\sqrt{\mu}}\big)
+\big(\frac{1}{\sqrt{\mu}}\partial^{\alpha}\big[P_1\xi_1M\big\{\frac{|\xi-u|^2\overline{\theta}_x}{2R\theta^2}
+\frac{(\xi-u)\cdot \overline{u}_{x}}{R\theta}\big\}\big],\frac{\partial^{\alpha}F}{\sqrt{\mu}}\big)
\nonumber\\
&\qquad+\big(\partial^{\alpha}[v\Gamma(\mathbf{g},\frac{M-\mu}{\sqrt{\mu}})]+
\partial^{\alpha}[v\Gamma(\frac{M-\mu}{\sqrt{\mu}},\mathbf{g})]
+\partial^{\alpha}[v\Gamma(\frac{G}{\sqrt{\mu}},\frac{G}{\sqrt{\mu}})],\frac{\partial^{\alpha}F}{\sqrt{\mu}}\big).
\end{align}
We will estimate \eqref{2.51} term by term. The second term of \eqref{2.51} vanishes by integration by parts. By $v_{t}=u_{1x}$ and the integration by parts, we can find that
$$
-\frac{1}{2}\big(v_{t}\frac{\partial^{\alpha}F}{\sqrt{\mu}},\frac{\partial^{\alpha}F}{\sqrt{\mu}}\big)
-\big(u_{1}(\frac{\partial^{\alpha}F}{\sqrt{\mu}})_{x},\frac{\partial^{\alpha}F}{\sqrt{\mu}}\big)=0.
$$
To control the fifth term of \eqref{2.51}. If $|\alpha_{1}|=1$, then $|\alpha-\alpha_{1}|=1$ since we consider $|\alpha|=2$.
Recalling that $F=M+\overline{G}+\sqrt{\mu}\mathbf{g}$, we have from \eqref{1.35}, \eqref{5.38}, \eqref{2.9},
\eqref{1.29}, \eqref{2.48aaa}  and \eqref{2.48} that
\begin{multline}
\label{2.59a}
|\big(\partial^{\alpha_{1}}v\frac{\partial^{\alpha-\alpha_{1}}F_{t}}{\sqrt{\mu}},\frac{\partial^{\alpha}F}{\sqrt{\mu}}\big)|
 \leq |(\partial^{\alpha_{1}}v\partial^{\alpha-\alpha_{1}}\mathbf{g}_{t},\partial^{\alpha}\mathbf{g})|\\
+C\|\partial^{\alpha_{1}}v\|_{L_{x}^{\infty}}
\Big(\|\langle \xi\rangle^{-\frac{1}{2}}\partial^{\alpha-\alpha_{1}}\mathbf{g}_{t}\|^{2}+\|\langle \xi\rangle^{-\frac{1}{2}}\partial^{\alpha}\mathbf{g}\|^{2}\\
\qquad\qquad\qquad\qquad\qquad\qquad\qquad\quad+
\|\langle \xi\rangle^{\frac{1}{2}}\frac{\partial^{\alpha-\alpha_{1}}(\overline{G}+M)_{t}}{\sqrt{\mu}}\|^{2}
+\|\langle \xi\rangle^{\frac{1}{2}}\frac{\partial^{\alpha}(\overline{G}+M)}{\sqrt{\mu}}\|^{2}
\Big)\\
\leq C\epsilon(\|\partial^{\alpha}\mathbf{g}\|_{\sigma}^{2}+\|\partial^{\alpha-\alpha_{1}}\mathbf{g}_{t}\|_{\sigma}^{2})
+C_{\epsilon}q_{3}(t)\mathcal{F}_{2,l,q}(t)+C(\delta+\varepsilon_{0})\mathcal{D}_{2,l,q}(t)+C\delta(1+t)^{-\frac{4}{3}}.
\end{multline}
If $|\alpha_{1}|=2$, then $|\alpha-\alpha_{1}|=0$ and we can obtain
\begin{align*}
|(\partial^{\alpha_{1}}v\mathbf{g}_{t},\partial^{\alpha}\mathbf{g})|
&\leq  \epsilon\|\langle \xi\rangle^{-\frac{1}{2}}\partial^{\alpha}\mathbf{g}\|^{2}
+C_{\epsilon}\|\partial^{\alpha_{1}}v\|^{2}\big\||\langle \xi\rangle^{\frac{1}{2}}\mathbf{g}_{t}|_2\big\|_{L_{x}^{\infty}}^{2}
\\
&\leq C\epsilon\|\partial^{\alpha}\mathbf{g}\|_{\sigma}^{2}+C_{\epsilon}(\|\partial^{\alpha_{1}}\widetilde{v}\|^{2}
+\|\partial^{\alpha_{1}}\bar{v}\|^{2})
\|\langle \xi\rangle^{\frac{1}{2}}\mathbf{g}_{t}\|\|\langle \xi\rangle^{\frac{1}{2}}\mathbf{g}_{tx}\|
\\
&\leq C\epsilon\|\partial^{\alpha}\mathbf{g}\|_{\sigma}^{2}+C_{\epsilon}q_{3}(t)\mathcal{F}_{2,l,q}(t),
\end{align*}
and
\begin{align}
\label{2.53}
\|\langle \xi\rangle^{-\frac{1}{2}}\frac{\partial^{\alpha}F}{\sqrt{\mu}}\|^{2}
&\leq C\|\langle \xi\rangle^{-\frac{1}{2}}\partial^{\alpha}\mathbf{g}\|^{2}
+C\|\langle \xi\rangle^{-\frac{1}{2}}\frac{\partial^{\alpha}\overline{G}}{\sqrt{\mu}}\|^{2}
+C\|\langle \xi\rangle^{-\frac{1}{2}}\frac{\partial^{\alpha}M}{\sqrt{\mu}}\|^{2}
\nonumber\\
&\leq C(\|\partial^{\alpha}\mathbf{g}\|_{\sigma}^{2}
+\|\partial^{\alpha}[\widetilde{v},\widetilde{u},\widetilde{\theta}]\|^{2})
+C(\delta+\varepsilon_{0})\mathcal{D}_{2,l,q}(t)+C\delta(1+t)^{-\frac{4}{3}}.
\end{align}
By these facts, if $|\alpha_{1}|=2$, we have
\begin{align}
\label{2.60a}
|\big(\partial^{\alpha_{1}}v\frac{\partial^{\alpha-\alpha_{1}}F_{t}}{\sqrt{\mu}},\frac{\partial^{\alpha}F}{\sqrt{\mu}}\big)|
&\leq |(\partial^{\alpha_{1}}v\mathbf{g}_{t},\partial^{\alpha}\mathbf{g})|
+C\|\partial^{\alpha_{1}}v\| \big\||\langle \xi\rangle^{-\frac{1}{2}}\mathbf{g}_{t}|_{2} \big\|_{L_{x}^{\infty}}
\|\langle \xi\rangle^{\frac{1}{2}}\frac{\partial^{\alpha}(\overline{G}+M)}{\sqrt{\mu}}\|
\nonumber\\
&\qquad+C\|\partial^{\alpha_{1}}v\|
 \big\||\langle \xi\rangle^{\frac{1}{2}}\frac{(\overline{G}+M)_{t}}{\sqrt{\mu}}|_{2} \big\|_{L_{x}^{\infty}}
\|\langle \xi\rangle^{-\frac{1}{2}}\frac{\partial^{\alpha}F}{\sqrt{\mu}}\|
\nonumber\\
&\leq C\epsilon\|\partial^{\alpha}\mathbf{g}\|_{\sigma}^{2}+C_{\epsilon}q_{3}(t)\mathcal{F}_{2,l,q}(t)
+C(\delta+\varepsilon_{0})\mathcal{D}_{2,l,q}(t)+C\delta(1+t)^{-\frac{4}{3}}.
\end{align}
It follows from \eqref{2.59a} and \eqref{2.60a} that
\begin{align}
\label{2.54a}
&\sum_{1\leq|\alpha_{1}|\leq|\alpha|}C^{\alpha_{1}}_{\alpha}|\big(\partial^{\alpha_{1}}v\frac{\partial^{\alpha-\alpha_{1}}F_{t}}{\sqrt{\mu}},
\frac{\partial^{\alpha}F}{\sqrt{\mu}}\big)|
\nonumber\\
&\leq C\epsilon\sum_{|\alpha|=2}\|\partial^{\alpha}\mathbf{g}\|_{\sigma}^{2}+C_{\epsilon}q_{3}(t)\mathcal{F}_{2,l,q}(t)
+C(\delta+\varepsilon_{0})\mathcal{D}_{2,l,q}(t)
+C\delta(1+t)^{-\frac{4}{3}}.
\end{align}
Similar arguments as \eqref{2.54a} imply
\begin{align}
\label{2.54}
&\sum_{1\leq|\alpha_{1}|\leq|\alpha|}C^{\alpha_{1}}_{\alpha}|\big(\partial^{\alpha_{1}}u_{1}(\frac{\partial^{\alpha-\alpha_{1}}F_{x}}{\sqrt{\mu}})
,\frac{\partial^{\alpha}F}{\sqrt{\mu}}\big)|
\nonumber\\
&\leq C\epsilon\sum_{|\alpha|=2}\|\partial^{\alpha}\mathbf{g}\|_{\sigma}^{2}+C_{\epsilon}q_{3}(t)\mathcal{F}_{2,l,q}(t)
+C(\delta+\varepsilon_{0})\mathcal{D}_{2,l,q}(t)
+C\delta(1+t)^{-\frac{4}{3}}.
\end{align}
For the first term on the right hand side of \eqref{2.51}, we have
$$
\big(\partial^{\alpha}(v\mathcal{L}\mathbf{g}),\frac{\partial^{\alpha}F}{\sqrt{\mu}}\big)
=\big(v\mathcal{L}\partial^{\alpha}\mathbf{g},\frac{\partial^{\alpha}F}{\sqrt{\mu}}\big)
+\sum_{1\leq|\alpha_{1}|\leq|\alpha|}C^{\alpha_{1}}_{\alpha}
\big(\partial^{\alpha_{1}}v\mathcal{L}\partial^{\alpha-\alpha_{1}}\mathbf{g},\frac{\partial^{\alpha}F}{\sqrt{\mu}}\big).
$$
Recalling that $F=M+\overline{G}+\sqrt{\mu}\mathbf{g}$, we have from \eqref{2.6} and \eqref{2.12} that
$$
-(v\mathcal{L}\partial^{\alpha}\mathbf{g},\partial^{\alpha}\mathbf{g})\geq  c_{2} \|v^{\frac{1}{2}}\partial^{\alpha}\mathbf{g}\|_{\sigma}^{2} \geq
\frac{c_{2}}{2}\|\partial^{\alpha}\mathbf{g}\|_{\sigma}^{2}.
$$
Recalling that $\mathcal{L}\mathbf{g}=\Gamma(\sqrt{\mu},\mathbf{g})+\Gamma(\mathbf{g},\sqrt{\mu})$, we have from
\eqref{5.7}, \eqref{5.38} and \eqref{1.29} that
$$
|(v\mathcal{L}\partial^{\alpha}\mathbf{g},\frac{\partial^{\alpha}\overline{G}}{\sqrt{\mu}})|
\leq C\|v\partial^{\alpha}\mathbf{g}\|_{\sigma}\|\frac{\partial^{\alpha}\overline{G}}{\sqrt{\mu}}\|_{\sigma}
\leq \epsilon\|\partial^{\alpha}\mathbf{g}\|_{\sigma}^{2}+C_{\epsilon}\delta\mathcal{D}_{2,l,q}(t)+C_{\epsilon}\delta(1+t)^{-\frac{4}{3}}.
$$
For $|\bar{\alpha}|= 2$, it is  seen by \eqref{1.5} with $\rho=1/v$ that
\begin{align*}
\partial^{\bar{\alpha}}M&=M\big(v\partial^{\bar{\alpha}}(\frac{1}{v})-\frac{3\partial^{\bar{\alpha}}\theta}{2\theta} +\frac{(\xi-u)^{2}\partial^{\bar{\alpha}}\theta}{2R\theta^{2}}+\sum^{3}_{i=1}
\frac{\partial^{\bar{\alpha}}u_{i}(\xi_{i}-u_{i})}{R\theta}\big)+\cdot\cdot\cdot
\nonumber\\
&=\big(\mu+(M-\mu)\big)\big(v\partial^{\bar{\alpha}}(\frac{1}{v})-\frac{3\partial^{\bar{\alpha}}\theta}{2\theta} +\frac{(\xi-u)^{2}\partial^{\bar{\alpha}}\theta}{2R\theta^{2}}+\sum^{3}_{i=1}
\frac{\partial^{\bar{\alpha}}u_{i}(\xi_{i}-u_{i})}{R\theta}\big)+\cdot\cdot\cdot\notag\\
&:=J^{\bar{\alpha}}_{1}
+J^{\bar{\alpha}}_{2}+J^{\bar{\alpha}}_{3}.
\end{align*}
Here the terms $J^{\bar{\alpha}}_{1}$ and $J^{\bar{\alpha}}_{2}$ are the higher order derivatives of $[v,u,\theta]$
with $\mu$ and $M-\mu$ and $J^{\bar{\alpha}}_{3}$ is the low order derivatives with $M$.
Since $\frac{J^{\alpha}_{1}}{\sqrt{\mu}}\in\ker{\mathcal{L}}$, it follows that
$(v\mathcal{L}\partial^{\alpha}\mathbf{g},\frac{J^{\alpha}_{1}}{\sqrt{\mu}})=0$. For the term
$\frac{J^{\alpha}_{2}}{\sqrt{\mu}}$, we have from  \eqref{5.7}, \eqref{5.32} and \eqref{1.29} that
$$
(v\mathcal{L}\partial^{\alpha}\mathbf{g},\frac{J^{\alpha}_{2}}{\sqrt{\mu}})
\leq C\|v\partial^{\alpha}\mathbf{g}\|_{\sigma}\|\frac{J^{\alpha}_{2}}{\sqrt{\mu}}\|_\sigma
\leq C(\eta_{0}+\varepsilon_{0})(\|\partial^{\alpha}\mathbf{g}\|_{\sigma}^{2}+\|\partial^{\alpha}[\widetilde{v},\widetilde{u},\widetilde{\theta}]\|^{2})
+C\delta(1+t)^{-\frac{4}{3}}.
$$
Similarly, it holds that
$$
(v\mathcal{L}\partial^{\alpha}\mathbf{g},\frac{J^{\alpha}_{3}}{\sqrt{\mu}})
\leq C\|v\partial^{\alpha}\mathbf{g}\|_{\sigma}\|\frac{J^{\alpha}_{3}}{\sqrt{\mu}}\|_\sigma
\leq \epsilon\|\partial^{\alpha}\mathbf{g}\|_{\sigma}^{2}+C_{\epsilon}(\delta+\varepsilon_{0})\mathcal{D}_{2,l,q}(t)+C_{\epsilon}\delta(1+t)^{-\frac{4}{3}}.
$$
By using \eqref{1.35} and the similar arguments as \eqref{2.53}, we obtain
\begin{align}\label{2.55a}
\|  \frac{\partial^{\alpha}F}{\sqrt{\mu}}\|^{2}_\sigma
&\leq C\| \partial^{\alpha}\mathbf{g}\|^{2}_\sigma
+C\| \frac{\partial^{\alpha}\overline{G}}{\sqrt{\mu}}\|^{2}_\sigma
+C\| \frac{\partial^{\alpha}M}{\sqrt{\mu}}\|^{2}_\sigma
\nonumber\\
&\leq C(\|\partial^{\alpha}\mathbf{g}\|_{\sigma}^{2}
+\|\partial^{\alpha}[\widetilde{v},\widetilde{u},\widetilde{\theta}]\|^{2})
+C(\delta+\varepsilon_{0})\mathcal{D}_{2,l,q}(t)+C\delta(1+t)^{-\frac{4}{3}}.
\end{align}
Recalling that $\mathcal{L}\mathbf{g}=\Gamma(\sqrt{\mu},\mathbf{g})+\Gamma(\mathbf{g},\sqrt{\mu})$,  we have from \eqref{5.7}, \eqref{2.53} and the imbedding inequality that
\begin{align*}
\sum_{1\leq|\alpha_{1}|\leq|\alpha|}C^{\alpha_{1}}_{\alpha}
|\big(\partial^{\alpha_{1}}v\mathcal{L}\partial^{\alpha-\alpha_{1}}\mathbf{g},\frac{\partial^{\alpha}F}{\sqrt{\mu}}\big)|
&\leq C\sum_{1\leq|\alpha_{1}|\leq|\alpha|}\int_{\mathbb{R}}|\partial^{\alpha_{1}}v|
|\partial^{\alpha-\alpha_{1}}\mathbf{g}|_{\sigma}|\frac{\partial^{\alpha}F}{\sqrt{\mu}}|_{\sigma}\,dx
\\
&\leq C(\delta+\varepsilon_{0})\mathcal{D}_{2,l,q}(t)+C\delta(1+t)^{-\frac{4}{3}}.
\end{align*}
For any small $\epsilon>0$, by using   the above  estimates, we arrive at
\begin{align*}
\big(\partial^{\alpha}(v\mathcal{L}\mathbf{g}),\frac{\partial^{\alpha}F}{\sqrt{\mu}}\big)&\leq-\frac{c_{2}}{4}\|\partial^{\alpha}\mathbf{g}\|_{\sigma}^{2}
+C(\eta_{0}+\varepsilon_{0})(\|\partial^{\alpha}\mathbf{g}\|_{\sigma}^{2}+\|\partial^{\alpha}[\widetilde{v},\widetilde{u},\widetilde{\theta}]\|^{2})\\
&\qquad+C(\delta+\varepsilon_{0})\mathcal{D}_{2,l,q}(t)+C\delta(1+t)^{-\frac{4}{3}}.
\end{align*}
By \eqref{2.35a} and \eqref{2.53}, we get
\begin{align*}
&|\big(\frac{1}{\sqrt{\mu}}\partial^{\alpha}P_1\xi_1M\big\{\frac{|\xi-u|^2\overline{\theta}_x}{2R\theta^2}
+\frac{(\xi-u)\cdot \overline{u}_{x}}{R\theta}\big\},\frac{\partial^{\alpha}F}{\sqrt{\mu}}\big)|
\\
&\leq C\delta\|\langle \xi\rangle^{-\frac{1}{2}}\frac{\partial^{\alpha}F}{\sqrt{\mu}}\|^{2}
+C\frac{1}{\delta}\|\langle \xi\rangle^{\frac{1}{2}}\frac{1}{\sqrt{\mu}}\partial^{\alpha}P_1\xi_1M\big\{\frac{|\xi-u|^2\overline{\theta}_x}{2R\theta^2}
+\frac{(\xi-u)\cdot \overline{u}_{x}}{R\theta}\big\}\|^{2}
\\
&\leq C(\delta+\varepsilon_{0})\mathcal{D}_{2,l,q}(t)+C\delta(1+t)^{-\frac{4}{3}}.
\end{align*}
By \eqref{2.55a}, \eqref{5.30} and \eqref{5.34}, one has
\begin{align*}
&|\big(\partial^{\alpha}[v\Gamma(\mathbf{g},\frac{M-\mu}{\sqrt{\mu}})]+
\partial^{\alpha}[v\Gamma(\frac{M-\mu}{\sqrt{\mu}},\mathbf{g})]
+\partial^{\alpha}[v\Gamma(\frac{G}{\sqrt{\mu}},\frac{G}{\sqrt{\mu}})],\frac{\partial^{\alpha}F}{\sqrt{\mu}}\big)|
\\
&\leq C(\eta_0+\delta+\varepsilon_{0})\|\partial^{\alpha}\mathbf{g}\|_{\sigma}^{2}+C(\eta_0+\delta+\varepsilon_{0})\mathcal{D}_{2,l,q}(t)+C\delta(1+t)^{-\frac{4}{3}}.
\end{align*}
Hence, we substitute the above   estimates into \eqref{2.51} and take a small  $\epsilon>0$   to get
\begin{align}
\label{2.56}
 \frac{d}{dt}\sum_{|\alpha|=2}\|v^{\frac{1}{2}}\frac{\partial^{\alpha}F}{\sqrt{\mu}}\|^{2}+c\sum_{|\alpha|=2}\|\partial^{\alpha}\mathbf{g}\|_{\sigma}^{2}&\leq C(\eta_{0}+\delta+\varepsilon_{0})\mathcal{D}_{2,l,q}(t)\notag\\
&\qquad+C\delta(1+t)^{-\frac{4}{3}}+Cq_{3}(t)\mathcal{F}_{2,l,q}(t).
\end{align}

For some large constants $\widetilde{C}_{4}>0$ and $ \widetilde{C}_{3}>0 $ with $\widetilde{C}_{4}\gg\widetilde{C}_{3}$, by a suitable linear combination of \eqref{2.44}, \eqref{2.49} and \eqref{2.56}, we have
\begin{align}
\label{2.57}
&\widetilde{C}_{3}\frac{d}{dt} \Big(\widetilde{C}_{2}\sum_{|\alpha|=1}
\int_{\mathbb{R}}\big(\frac{p_{+}}{2}|\partial^{\alpha}\widetilde{v}|^{2}+\frac{v}{2}|\partial^{\alpha}\widetilde{u}_{1}|^{2}
+\frac{1}{2}\sum^{3}_{i=2}|\partial^{\alpha}\widetilde{u}_{i}|^{2}+\frac{R}{2p_{+}}|\partial^{\alpha}\widetilde{\theta}|^{2}\big)dx
-(\widetilde{u}_{1x},\widetilde{v}_{xx})\Big)
\nonumber\\
&\qquad+\frac{d}{dt}\Big\{\sum_{|\alpha|=1}\|v^{\frac{1}{2}}\partial^{\alpha}\mathbf{g}\|^{2}
+\widetilde{C}_{4}\sum_{|\alpha|=2}\|v^{\frac{1}{2}}\frac{\partial^{\alpha}F}{\sqrt{\mu}}\|^{2}\Big\}
+c\sum_{|\alpha|=2}\|\partial^{\alpha}[\widetilde{v},\widetilde{u},\widetilde{\theta}]\|^{2}
+c\sum_{1\leq|\alpha|\leq2}\|\partial^{\alpha}\mathbf{g}\|^{2}_{\sigma}
\nonumber\\
&\leq C(\eta_{0}+\delta+\varepsilon_{0})\mathcal{D}_{2,l,q}(t)
+C\delta(1+t)^{-\frac{4}{3}}+Cq_{3}(t)\mathcal{F}_{2,l,q}(t).
\end{align}
 This completes the proof of the derivative estimates of $[\widetilde{v},\widetilde{u},\widetilde{\theta}]$
and $\mathbf{g}$.

For some large constant $\widetilde{C}_{5}>0$ with $\widetilde{C}_{5}\gg \widetilde{C}_{1}$ in \eqref{2.31}, we denote $E_{1}(t)$ as
\begin{align}
\label{2.59}
E_{1}(t)&=\Big\{\int_{\mathbb{R}}\Big(\widetilde{C}_1\big(\frac{2}{3}\bar{\theta}\Phi(\frac{v}{\bar{v}})+\frac{1}{2}\widetilde{u}^{2}
+\bar{\theta}\Phi(\frac{\theta}{\bar{\theta}})\big)-\widetilde{C}_1\kappa_{1}\widetilde{u}_{1}\widetilde{v}_{x}\Big)\,dx
+\|v^{\frac{1}{2}}\mathbf{g}\|^{2}\Big\}
\nonumber\\
&\qquad+\widetilde{C}_{5}\widetilde{C}_{3} \Big\{\widetilde{C}_{2}\sum_{|\alpha|=1}
\int_{\mathbb{R}}\big(\frac{p_{+}}{2}|\partial^{\alpha}\widetilde{v}|^{2}+\frac{v}{2}|\partial^{\alpha}\widetilde{u}_{1}|^{2}
+\frac{1}{2}\sum^{3}_{i=2}|\partial^{\alpha}\widetilde{u}_{i}|^{2}+\frac{R}{2p_{+}}|\partial^{\alpha}\widetilde{\theta}|^{2}\big)\,dx\notag\\
&\qquad\qquad\qquad\qquad\qquad-(\widetilde{u}_{1x},\widetilde{v}_{xx})\Big\}
\nonumber\\
&\qquad+\widetilde{C}_{5}\Big\{\sum_{|\alpha|=1}\|v^{\frac{1}{2}}\partial^{\alpha}\mathbf{g}\|^{2}
+\widetilde{C}_{4}\sum_{|\alpha|=2}\|v^{\frac{1}{2}}\frac{\partial^{\alpha}F}{\sqrt{\mu}}\|^{2}\Big\}.
\end{align}
By using this and taking the summation of \eqref{2.31} and \eqref{2.57}$\times\widetilde{C}_{5}$, we have
\begin{align}
\label{2.58}
&\frac{d}{dt}E_{1}(t)+c\sum_{1\leq|\alpha|\leq2}\|\partial^{\alpha}[\widetilde{v},\widetilde{u},\widetilde{\theta}]\|^{2}
+c\sum_{|\alpha|\leq2}\|\partial^{\alpha}\mathbf{g}\|^{2}_{\sigma}
\nonumber\\
&\leq C(\eta_{0}+\delta+\varepsilon_{0})\mathcal{D}_{2,l,q}(t)
+C\delta(1+t)^{-\frac{4}{3}}+C\delta\int_{\mathbb{R}}(\widetilde{v}^{2}+\widetilde{\theta}^{2})\omega^{2}\,dx
+Cq_{3}(t)\mathcal{F}_{2,l,q}(t).
\end{align}
Here $\mathcal{D}_{2,l,q}(t)$, $q_{3}(t)$ and $\mathcal{F}_{2,l,q}(t)$ are defined by \eqref{2.11}, \eqref{1.33} and \eqref{2.48}, respectively.
This estimate is the main energy estimate in this section and this completes the proof of the non-weighted energy estimates of solution.\qed

\section{ Weighted energy estimates}\label{sec.3}

In this section we will consider energy estimates with the weight function $w(\beta)$ in \eqref{1.32} in order to close the a priori assumption. And the weight function will be acted on the microscopic component $\mathbf{g}$ for the equation  \eqref{2.4}.

\subsection{Time-spatial derivative  estimates}
We first  consider the   estimates of the microscopic component $\mathbf{g}$ with the weight $w=w(0)$
in \eqref{1.32}. Applying $\partial^\alpha$ to \eqref{2.4} with $|\alpha|\leq 1$ and taking
the inner product of the resulting equation with $w^2(0)\partial^\alpha \mathbf{g}$ over $\mathbb{R}\times{\mathbb R}^3$, one has
\begin{align}
\label{3.1}
&(\partial^{\alpha}(v\mathbf{g}_{t}),w^2(0)\partial^\alpha \mathbf{g})
-(\partial^{\alpha}(u_{1}\mathbf{g}_{x}),w^2(0)\partial^\alpha \mathbf{g})
-(\partial^{\alpha}(v\mathcal{L}\mathbf{g}),w^2(0)\partial^\alpha \mathbf{g})
\nonumber\\
&=\big(\partial^{\alpha}[v\Gamma(\mathbf{g},\frac{M-\mu}{\sqrt{\mu}})]+\partial^{\alpha}[v\Gamma(\frac{M-\mu}{\sqrt{\mu}},\mathbf{g})]
+\partial^{\alpha}[v\Gamma(\frac{G}{\sqrt{\mu}},\frac{G}{\sqrt{\mu}})],w^2(0)\partial^\alpha \mathbf{g}\big)
\nonumber\\
&\quad-\big(\partial^{\alpha}[\frac{1}{\sqrt{\mu}}P_{1}\xi_{1}M\big\{\frac{|\xi-u|^{2}
\widetilde{\theta}_{x}}{2R\theta^{2}}+\frac{(\xi-u)\cdot\widetilde{u}_{x}}{R\theta}\big\}],w^2(0)\partial^\alpha \mathbf{g}\big)
+\big(\partial^{\alpha}[\frac{P_{0}(\xi_{1}\sqrt{\mu}\mathbf{g}_{x})}{\sqrt{\mu}}]
,w^2(0)\partial^\alpha \mathbf{g}\big)
\nonumber\\
&\quad-\big(\partial^{\alpha}[\frac{P_{1}(\xi_{1}\overline{G}_{x})}{\sqrt{\mu}}],w^2(0)\partial^\alpha \mathbf{g}\big)+\big(\partial^{\alpha}[u_{1}\frac{\overline{G}_{x}}{\sqrt{\mu}}],w^2(0)\partial^\alpha \mathbf{g}\big)
-\big(\partial^{\alpha}[v\frac{\overline{G}_{t}}{\sqrt{\mu}}],w^2(0)\partial^\alpha \mathbf{g}\big),
\end{align}
where we have used the fact that
$$
(\xi_1\partial^\alpha\mathbf{g}_{x},w^2(0)\partial^\alpha \mathbf{g})=0.
$$
We shall estimate \eqref{3.1} term by term.  Recalling  the weight function $w=w(0)$ in \eqref{1.32} and the fact that $|\alpha|\leq 1$, one has
\begin{align}\label{3.2aaa}
(\partial^{\alpha}(v\mathbf{g}_{t}),w^2(0)\partial^\alpha \mathbf{g})=(v\partial^{\alpha}\mathbf{g}_{t},w^2(0)\partial^\alpha \mathbf{g})
+(\partial^{\alpha}v\mathbf{g}_{t},w^2(0)\partial^\alpha \mathbf{g}).
\end{align}
If $|\alpha|=0$, the last term in the above equality vanishes. If $|\alpha|=1$, we have from   \eqref{2.48aa} and  \eqref{2.48} that
\begin{align}
\label{3.2a}
|(\partial^{\alpha}v\mathbf{g}_{t},w^2(0)\partial^\alpha \mathbf{g})|
&\leq C\epsilon\sum_{|\alpha|=1}\|w(0)\partial^\alpha \mathbf{g}\|_\sigma^{2}
+C_\epsilon q_{3}(t)\|\langle \xi\rangle w(0)\mathbf{g}_{t} \|^{2}
\nonumber\\
&\leq C\epsilon\sum_{|\alpha|=1}\|\partial^{\alpha} \mathbf{g}\|_{\sigma,w}^{2}
+C_{\epsilon}q_{3}(t)\mathcal{F}_{2,l,q}(t).
\end{align}
By the similar arguments as \eqref{3.2a}, one has
\begin{align}
\label{3.2aa}
\sum_{|\alpha|\leq 1}|(v_{t} \partial^\alpha \mathbf{g},w^{2}(0)\partial^\alpha \mathbf{g})|
\leq C\epsilon\sum_{|\alpha|\leq 1}\|\partial^\alpha \mathbf{g}\|_{\sigma,w}^{2}
+C_{\epsilon}q_{3}(t)\mathcal{F}_{2,l,q}(t).
\end{align}
By \eqref{1.32}, for any multi-indices $|\beta|\geq 0$,  we see
\begin{align}
\label{3.3}
\partial_tw^2(\beta)=-2q_{2}q_{3}(t)\langle \xi\rangle^{2} w^2(\beta).
\end{align}
It follows from this and \eqref{3.2aa} that
\begin{align}
\label{3.2}
(v\partial^{\alpha}\mathbf{g}_{t},w^2(0)\partial^\alpha \mathbf{g})
&=\frac{1}{2}\frac{d}{dt}\|v^{\frac{1}{2}}\partial^\alpha \mathbf{g}\|_w^2
-\frac{1}{2}(v\partial^\alpha \mathbf{g} ,[w^{2}(0)]_{t}\partial^\alpha \mathbf{g})
-\frac{1}{2}(v_{t} \partial^\alpha \mathbf{g},w^{2}(0)\partial^\alpha \mathbf{g})
\nonumber\\
&\geq \frac{1}{2}\frac{d}{dt}\|v^{\frac{1}{2}}\partial^\alpha \mathbf{g}\|_w^2
+ q_{2}q_{3}(t)\|v^{\frac{1}{2}}\langle \xi\rangle\partial^\alpha \mathbf{g}\|^2_w\notag\\
&\qquad-C\epsilon\sum_{|\alpha|\leq 1}\|\partial^{\alpha}\mathbf{g}\|_{\sigma,w}^{2}-C_\epsilon q_{3}(t)\mathcal{F}_{2,l,q}(t).
\end{align}
For the first term on the left hand side of \eqref{3.1}, by using  \eqref{3.2aaa}, \eqref{3.2a} and \eqref{3.2}, we arrive at
\begin{align*}
(\partial^{\alpha}(v\mathbf{g}_{t}),w^2(0)\partial^\alpha \mathbf{g})
&\geq \frac{1}{2}\frac{d}{dt}\|v^{\frac{1}{2}}\partial^\alpha \mathbf{g}\|_w^2
+ q_{2}q_{3}(t)\|v^{\frac{1}{2}}\langle \xi\rangle\partial^\alpha \mathbf{g}\|^2_w
\nonumber\\
 &\qquad-C\epsilon\sum_{|\alpha|\leq 1}\|\partial^{\alpha} \mathbf{g}\|^2_{\sigma,w}-C_\epsilon q_{3}(t)\mathcal{F}_{2,l,q}(t).
\end{align*}
 If $|\alpha|=1$, by using the integration by parts, we  see
$$(\partial^{\alpha}(u_{1}\mathbf{g}_{x}),w^2(0)\partial^\alpha \mathbf{g})=(\partial^{\alpha}u_{1}\mathbf{g}_{x},w^2(0)\partial^\alpha \mathbf{g})-\frac{1}{2}(u_{1x}\partial^{\alpha}\mathbf{g} ,w^2(0)\partial^\alpha \mathbf{g}).$$
If $|\alpha|=0$, we only have the last term.  For the second term on the left hand side of
 \eqref{3.1}, by using this, \eqref{2.48aaa} and  \eqref{2.48}, we have
\begin{align*}
\sum_{|\alpha|\leq 1}|(\partial^{\alpha}(u_{1}\mathbf{g}_{x}),w^2(0)\partial^\alpha \mathbf{g})|
\leq C\epsilon\sum_{|\alpha|\leq 1}\|\partial^{\alpha} \mathbf{g}\|^2_{\sigma,w}+C_{\epsilon}q_{3}(t)\mathcal{F}_{2,l,q}(t).
\end{align*}
For the third term on the left hand side of \eqref{3.1}. If $|\alpha|=1$,  we easily see
\begin{align*}
-(\partial^{\alpha}(v\mathcal{L}\mathbf{g}),w^2(0)\partial^\alpha \mathbf{g})
=-(v\mathcal{L}\partial^{\alpha}\mathbf{g},w^2(0)\partial^\alpha \mathbf{g})-(\partial^{\alpha}v\mathcal{L}\mathbf{g},w^2(0)\partial^\alpha \mathbf{g}).
\end{align*}
The last term in the above equality vanishes as $|\alpha|=0$. From \eqref{5.6}, \eqref{1.35} and \eqref{2.12}, it is easily seen that
\begin{align*}
-(v\mathcal{L}\partial^\alpha \mathbf{g},w^2(0)\partial^\alpha \mathbf{g})&\geq  c_{4} \|v^{\frac{1}{2}}\partial^\alpha \mathbf{g}\|^2_{\sigma,w}-C\|v^{\frac{1}{2}}\partial^\alpha \mathbf{g}\|^2_{\sigma}
\nonumber\\
 &\geq \frac{c_{4}}{2}\| \partial^\alpha \mathbf{g}\|^2_{\sigma,w}-C\| \partial^\alpha \mathbf{g}\|^2_{\sigma}.
\end{align*}
Note that $\mathcal{L}\mathbf{g}=\Gamma(\sqrt{\mu},\mathbf{g})+\Gamma(\mathbf{g},\sqrt{\mu})$.  If $|\alpha|=1$, we have from  \eqref{5.8} and  the imbedding inequality that
\begin{align*}
|(\partial^{\alpha}v\mathcal{L}\mathbf{g},w^2(0)\partial^\alpha \mathbf{g})|
\leq C\|\partial^{\alpha}v\|_{L^{\infty}}\|\mathbf{g}\|_{\sigma,w}\|\partial^\alpha \mathbf{g}\|_{\sigma,w}
\leq C(\delta+\varepsilon_{0})(\|\mathbf{g}\|^{2}_{\sigma,w}+\|\partial^\alpha \mathbf{g}\|^{2}_{\sigma,w}).
\end{align*}
We thus deduce from the above   estimates that
\begin{align*}
-(\partial^{\alpha}(v\mathcal{L}\mathbf{g}),w^2(0)\partial^\alpha \mathbf{g})\geq \frac{c_{4}}{2}\|\partial^\alpha \mathbf{g}\|^2_{\sigma,w}-C\|\partial^\alpha \mathbf{g}\|^2_{\sigma}
-C(\delta+\varepsilon_{0})(\|\mathbf{g}\|^{2}_{\sigma,w}+\|\partial^\alpha \mathbf{g}\|^{2}_{\sigma,w}).
\end{align*}
For the first term on the right hand side of \eqref{3.1}, we get from  \eqref{5.29} and \eqref{5.33} that
\begin{align*}
&|\big(\partial^{\alpha}[v\Gamma(\mathbf{g},\frac{M-\mu}{\sqrt{\mu}})]+\partial^{\alpha}[v\Gamma(\frac{M-\mu}{\sqrt{\mu}},\mathbf{g})]
+\partial^{\alpha}[v\Gamma(\frac{G}{\sqrt{\mu}},\frac{G}{\sqrt{\mu}})],w^2(0)\partial^\alpha \mathbf{g}\big)|
\nonumber\\
&\leq C(\eta_0+\delta+\varepsilon_{0})\|\partial^\alpha \mathbf{g}\|_{\sigma,w}^2+C(\eta_0+\delta+\varepsilon_{0})\mathcal{D}_{2,l,q}(t)
+C\delta(1+t)^{-\frac{4}{3}}.
\end{align*}
By using   \eqref{2.35a}, \eqref{1.35}, \eqref{2.9}, \eqref{1.29} and  the imbedding inequality,  we arrive at
\begin{align*}
&|\big(\partial^{\alpha}[\frac{1}{\sqrt{\mu}}P_{1}\xi_{1}M\big\{\frac{|\xi-u|^{2}
\widetilde{\theta}_{x}}{2R\theta^{2}}+\frac{(\xi-u)\cdot\widetilde{u}_{x}}{R\theta}\big\}],w^2(0)\partial^\alpha \mathbf{g}\big)|
\nonumber\\
&\leq C\|\langle \xi\rangle^{\frac{1}{2}}\mu^{-\frac{1}{2}}w(0)\partial^\alpha P_1\xi_1M\Big\{\frac{|\xi-u|^2 \widetilde{\theta}_x}{2R\theta^2}+\frac{(\xi-u)\cdot  \widetilde {u}_{x}}{R\theta}\Big\} \|\|\langle \xi\rangle^{-\frac{1}{2}} w(0)\partial^\alpha \mathbf{g}\|
\nonumber\\
&\leq \epsilon\|\partial^\alpha \mathbf{g}\|^2_{\sigma,w}+
C_\epsilon\|[\partial^\alpha\widetilde{u}_x,\partial^\alpha\widetilde{\theta}_x]\|^2
+C_\epsilon(\delta+\varepsilon_{0})\mathcal{D}_{2,l,q}(t).
\end{align*}
Here we used the fact that $\|w(0)\langle \xi\rangle^b\mu^{-\frac{1}{2}} M\|\leq C$ for any $b\ge 0$ by \eqref{2.9a} and
\eqref{2.12}.

Similar arguments as the above imply
\begin{align*}
|\big(\partial^{\alpha}[\frac{P_{0}(\xi_{1}\sqrt{\mu}\mathbf{g}_{x})}{\sqrt{\mu}}],w^2(0)\partial^\alpha \mathbf{g}\big)|
\leq \epsilon\|  \partial^\alpha \mathbf{g}\|_{\sigma,w}^2+C_\epsilon\| \partial^\alpha \mathbf{g}_x\|_{\sigma}^2
+C_\epsilon(\delta+\varepsilon_{0})\mathcal{D}_{2,l,q}(t).
\end{align*}
By  \eqref{5.38}, \eqref{1.35}, \eqref{1.29} and \eqref{2.9}, one has
\begin{align*}
&|\big(\partial^{\alpha}[\frac{P_{1}(\xi_{1}\overline{G}_{x})}{\sqrt{\mu}}],w^2(0)\partial^\alpha \mathbf{g}\big)|+|\big(\partial^{\alpha}[u_{1}\frac{\overline{G}_{x}}{\sqrt{\mu}}],w^2(0)\partial^\alpha \mathbf{g}\big)|
+|\big(\partial^{\alpha}[v\frac{\overline{G}_{t}}{\sqrt{\mu}}],w^2(0)\partial^\alpha \mathbf{g}\big)|
\nonumber\\
&\leq C\epsilon\|\langle \xi\rangle^{-\frac{1}{2}}\partial^\alpha \mathbf{g}\|_{w}^2+C_\epsilon\big\{\|\langle \xi\rangle^{\frac{1}{2}}\partial^{\alpha}[\frac{P_{1}(\xi_{1}\overline{G}_{x})}{\sqrt{\mu}}]\|_{w}^{2}
+\|\langle \xi\rangle^{\frac{1}{2}}\partial^{\alpha}[u_{1}\frac{\overline{G}_{x}}{\sqrt{\mu}}]\|_{w}^{2}
+\|\langle \xi\rangle^{\frac{1}{2}}\partial^{\alpha}[v\frac{\overline{G}_{t}}{\sqrt{\mu}}]\|_{w}^{2}\big\}
\nonumber\\
&\leq C\epsilon\|\partial^\alpha \mathbf{g}\|^2_{\sigma,w}+C_\epsilon(\delta+\varepsilon_{0})\mathcal{D}_{2,l,q}(t)
+C_{\epsilon}\delta(1+t)^{-\frac{4}{3}}.
\end{align*}
By choosing a small $\epsilon>0$, we thus have from \eqref{3.1}  and the above estimates   that
\begin{eqnarray}
\label{3.4}
\sum_{|\alpha|\leq 1}&&\Big\{\frac{1}{2}\frac{d}{dt}\|v^{\frac{1}{2}}\partial^\alpha \mathbf{g}\|_w^2
+ q_{2}q_{3}(t)\|v^{\frac{1}{2}}\langle \xi\rangle\partial^\alpha \mathbf{g}\|^2_w
+c\|\partial^\alpha \mathbf{g}\|^2_{\sigma,w}\Big\}
\notag\\
\leq&& C\sum_{|\alpha|\leq 1}(\|\partial^\alpha \mathbf{g}\|^2_{\sigma}+\|\partial^\alpha \mathbf{g}_x\|^2_{\sigma}
+\|[\partial^\alpha\widetilde{u}_x,\partial^\alpha\widetilde{\theta}_x]\|^2)
\notag\\
&&\hspace{3cm}+C( \eta_0+\delta+\varepsilon_{0})\mathcal{D}_{2,l,q}(t)+C\delta(1+t)^{-\frac{4}{3}}+Cq_{3}(t)\mathcal{F}_{2,l,q}(t).
\end{eqnarray}
This completes the proof of the low order weighted estimates of the microscopic component $\mathbf{g}$. Then we
will deduce the high order weighted estimates of the microscopic component $\mathbf{g}$.

We now consider the estimates for the microscopic component $\partial^\alpha \mathbf{g}$ with the weight $w=w(0)$ and $|\alpha|=2$.
Applying $\partial^{\alpha}$ to \eqref{2.50} and then taking the inner product with
$w^{2}(0)\frac{\partial^\alpha F}{\sqrt{\mu}}$ over $\mathbb{R}\times{\mathbb R}^3$, one has
\begin{align}
\label{3.5}
&\big(v(\frac{\partial^{\alpha}F}{\sqrt{\mu}})_{t},w^{2}(0)\frac{\partial^{\alpha}F}{\sqrt{\mu}}\big)
-\big(u_{1}(\frac{\partial^{\alpha}F}{\sqrt{\mu}})_{x},w^{2}(0)\frac{\partial^{\alpha}F}{\sqrt{\mu}}\big)
+\big(\xi_{1}(\frac{\partial^{\alpha}F}{\sqrt{\mu}})_{x},w^{2}(0)\frac{\partial^{\alpha}F}{\sqrt{\mu}}\big)
\nonumber\\
&\qquad+\sum_{1\leq|\alpha_{1}|\leq|\alpha|}\big(C^{\alpha_{1}}_{\alpha}\big[\partial^{\alpha_{1}}v\frac{\partial^{\alpha-\alpha_{1}}F_{t}}{\sqrt{\mu}}
-\partial^{\alpha_{1}}u_{1}\frac{\partial^{\alpha-\alpha_{1}}F_{x}}{\sqrt{\mu}}\big],w^{2}(0)\frac{\partial^{\alpha}F}{\sqrt{\mu}}\big)
\nonumber\\
&=\big(\partial^{\alpha}(v\mathcal{L}\mathbf{g}),w^{2}(0)\frac{\partial^{\alpha}F}{\sqrt{\mu}}\big)
+\big(\frac{1}{\sqrt{\mu}}\partial^{\alpha}P_1\xi_1M\big\{\frac{|\xi-u|^2\overline{\theta}_x}{2R\theta^2}
+\frac{(\xi-u)\cdot \overline{u}_{x}}{R\theta}\big\},w^{2}(0)\frac{\partial^{\alpha}F}{\sqrt{\mu}}\big)
\nonumber\\
&\qquad+\big(\partial^{\alpha}[v\Gamma(\mathbf{g},\frac{M-\mu}{\sqrt{\mu}})]+
\partial^{\alpha}[v\Gamma(\frac{M-\mu}{\sqrt{\mu}},\mathbf{g})]
+\partial^{\alpha}[v\Gamma(\frac{G}{\sqrt{\mu}},\frac{G}{\sqrt{\mu}})],w^{2}(0)\frac{\partial^{\alpha}F}{\sqrt{\mu}}\big).
\end{align}
We will estimate \eqref{3.5} term by term. By \eqref{3.3}, one has
\begin{align*}
(v(\frac{\partial^{\alpha}F}{\sqrt{\mu}})_{t},w^{2}(0)\frac{\partial^{\alpha}F}{\sqrt{\mu}})
=\frac{1}{2}\frac{d}{dt}\|v^{\frac{1}{2}}\frac{\partial^{\alpha}F}{\sqrt{\mu}}\|_{w}^{2}
-\frac{1}{2}(v_{t}\frac{\partial^{\alpha}F}{\sqrt{\mu}},w^{2}(0)\frac{\partial^{\alpha}F}{\sqrt{\mu}})
+q_{2}q_{3}(t)\|v^{\frac{1}{2}}\langle \xi\rangle\frac{\partial^{\alpha}F}{\sqrt{\mu}}\|_{w}^{2}.
\end{align*}
Note that the second term on the right hand side of the above equality together with the second term
of \eqref{3.5} can be cancelled by the fact that $v_{t}=u_{1x}$, namely
$$
-\frac{1}{2}(v_{t}\frac{\partial^{\alpha}F}{\sqrt{\mu}},w^{2}(0)\frac{\partial^{\alpha}F}{\sqrt{\mu}})
-\big(u_{1}(\frac{\partial^{\alpha}F}{\sqrt{\mu}})_{x},w^{2}(0)\frac{\partial^{\alpha}F}{\sqrt{\mu}}\big)=0.
$$
By using \eqref{2.9} and the fact  that $F=M+\overline{G}+\sqrt{\mu}\mathbf{g}$,
we have from \eqref{1.29}, \eqref{5.38}, \eqref{2.12} and the imbedding inequality that
\begin{align*}
\|v^{\frac{1}{2}}\langle \xi\rangle w(0)\frac{\partial^\alpha F}{\sqrt{\mu}}\|^2
&\geq \frac{1}{2}\|v^{\frac{1}{2}}\langle \xi\rangle w(0)\partial^\alpha \mathbf{g}\|^2-C\|v^{\frac{1}{2}}\langle \xi\rangle w(0)\frac{\partial^\alpha M}{\sqrt{\mu}}\|^2
-C\|v^{\frac{1}{2}}\langle \xi\rangle w(0)\frac{\partial^\alpha \overline{G}}{\sqrt{\mu}}\|^2
\notag\\
&\geq \frac{1}{2}\|v^{\frac{1}{2}}\langle \xi\rangle\partial^\alpha \mathbf{g}\|_{w}^2-C\|\partial^\alpha[\widetilde{v},\widetilde{u},\widetilde{\theta}]\|^2-C\delta(1+t)^{-\frac{4}{3}}
-C(\delta+\varepsilon_{0})\mathcal{D}_{2,l,q}(t).
\end{align*}
Here we also used the fact that $\|w(0)\langle \xi\rangle^b\mu^{-\frac{1}{2}} M^{1-\varepsilon}\|\leq C$
for any $b\ge 0$ and $\varepsilon$ small enough by \eqref{2.9a} and  \eqref{2.12}.
It follows from the above estimates and \eqref{2.9a} that
\begin{multline*}
(v(\frac{\partial^{\alpha}F}{\sqrt{\mu}})_{t},w^{2}(0)\frac{\partial^{\alpha}F}{\sqrt{\mu}})
-\big(u_{1}(\frac{\partial^{\alpha}F}{\sqrt{\mu}})_{x},w^{2}(0)\frac{\partial^{\alpha}F}{\sqrt{\mu}}\big)
\\
\geq \frac{1}{2}\frac{d}{dt}\|v^{\frac{1}{2}}\frac{\partial^{\alpha}F}{\sqrt{\mu}}\|_{w}^{2}
+\frac{1}{2}q_{2}q_{3}(t)\|v^{\frac{1}{2}}\langle \xi\rangle\partial^\alpha \mathbf{g}\|_{w}^2
\\
-C\|\partial^\alpha[\widetilde{v},\widetilde{u},\widetilde{\theta}]\|^2-C(\delta+\varepsilon_{0})\mathcal{D}_{2,l,q}(t)-C\delta(1+t)^{-\frac{4}{3}}.
\end{multline*}
Similar arguments as \eqref{2.54a} and \eqref{2.54} imply
\begin{align*}
&\sum_{1\leq|\alpha_{1}|\leq|\alpha|}C^{\alpha_{1}}_{\alpha}|\big(\big[\partial^{\alpha_{1}}v\frac{\partial^{\alpha-\alpha_{1}}F_{t}}{\sqrt{\mu}}
-\partial^{\alpha_{1}}u_{1}\frac{\partial^{\alpha-\alpha_{1}}F_{x}}{\sqrt{\mu}}\big],w^{2}(0)\frac{\partial^{\alpha}F}{\sqrt{\mu}}\big)|
\nonumber\\
&\leq  C\epsilon\sum_{|\alpha|=2}\|\partial^{\alpha}\mathbf{g}\|_{\sigma,w}^{2}+C(\delta+\varepsilon_{0})\mathcal{D}_{2,l,q}(t)
+C\delta(1+t)^{-\frac{4}{3}}+C_{\epsilon}q_{3}(t)\mathcal{F}_{2,l,q}(t).
\end{align*}
For the first term on the right hand side of \eqref{3.5}, one has
\begin{equation}
\label{3.6aa}
\big(\partial^{\alpha}(v\mathcal{L}\mathbf{g}),w^{2}(0)\frac{\partial^{\alpha}F}{\sqrt{\mu}}\big)=
\big(v\mathcal{L}\partial^{\alpha}\mathbf{g},w^{2}(0)\frac{\partial^{\alpha}F}{\sqrt{\mu}}\big)
+\sum_{1\leq|\alpha_{1}|\leq|\alpha|}C^{\alpha_{1}}_{\alpha}
\big(\partial^{\alpha_{1}}v\mathcal{L}\partial^{\alpha-\alpha_{1}}\mathbf{g},w^{2}(0)\frac{\partial^{\alpha}F}{\sqrt{\mu}}\big).
\end{equation}
Recalling that $F=M+\overline{G}+\sqrt{\mu}\mathbf{g}$ and $\mathcal{L}\mathbf{g}=\Gamma(\mathbf{g},\sqrt{\mu})+\Gamma(\sqrt{\mu},\mathbf{g})$,
we have from \eqref{5.8} and \eqref{2.12}  that
\begin{align*}
(v\mathcal{L} \partial^\alpha \mathbf{g},w^2(0)\frac{ \partial^\alpha M}{\sqrt{\mu}})&\leq C\|v^{\frac{1}{2}}w(0)\partial^\alpha \mathbf{g}\|_\sigma\|v^{\frac{1}{2}}w(0)\frac{ \partial^\alpha M }{\sqrt{\mu}}\|_\sigma\\
&\leq \epsilon\|\partial^\alpha \mathbf{g}\|_{\sigma,w}^2
+C_{\epsilon}\|\partial^\alpha[\widetilde{v},\widetilde{u},\widetilde{\theta}]\|^2
+C_{\epsilon}(\delta+\varepsilon_{0})\mathcal{D}_{2,l,q}(t)
+C_{\epsilon}\delta(1+t)^{-\frac{4}{3}}.
\end{align*}
By \eqref{5.8}, \eqref{5.38}, \eqref{2.9} and \eqref{1.29}, we have
\begin{align*}
(v\mathcal{L} \partial^\alpha \mathbf{g},w^2(0)\frac{\partial^\alpha \overline{G}}{\sqrt{\mu}})
&\leq  C\|v^{\frac{1}{2}}w(0)\partial^\alpha \mathbf{g}\|_{\sigma}\|v^{\frac{1}{2}}w(0)\frac{\partial^\alpha \overline{G}}{\sqrt{\mu}}\|_\sigma
\\
&\leq \epsilon\|\partial^\alpha \mathbf{g}\|_{\sigma,w}^2+C_{\epsilon}(\delta+\varepsilon_{0})\mathcal{D}_{2,l,q}(t)+C_{\epsilon}\delta(1+t)^{-\frac{4}{3}}.
\end{align*}
From \eqref{5.6}, \eqref{1.35} and \eqref{2.12}, it is easily seen that
$$
-(v\mathcal{L} \partial^\alpha \mathbf{g}, w^2(0)\partial^\alpha \mathbf{g})\geq \frac{c_{4}}{2}\|\partial^\alpha \mathbf{g}\|_{\sigma,w}^2-C\|\partial^\alpha \mathbf{g}\|_{\sigma }^2.
$$
For $|\alpha|=2$, it holds that
\begin{multline}
\label{3.6}
\|w (0)\frac{\partial^\alpha F}{\sqrt{\mu}}\|^2_\sigma
\leq C\|\partial^\alpha \mathbf{g}\|_{\sigma,w}^2+C\|\frac{\partial^\alpha \overline{G}}{\sqrt{\mu}}\|^2_{\sigma,w}
+C\|\frac{\partial^\alpha M}{\sqrt{\mu}}\|^2_{\sigma,w}\\
\leq C(\|\partial^\alpha \mathbf{g}\|_{\sigma,w}^2+\|\partial^\alpha[\widetilde{v},\widetilde{u},\widetilde{\theta}]\|^2)
+C(\delta+\varepsilon_{0})\mathcal{D}_{2,l,q}(t)+C\delta(1+t)^{-\frac{4}{3}}.
\end{multline}
By this, \eqref{5.8} and the  imbedding inequality, one has
\begin{align*}
\sum_{1\leq|\alpha_{1}|\leq|\alpha|}C^{\alpha_{1}}_{\alpha}\big(\partial^{\alpha_{1}}v\mathcal{L}\partial^{\alpha-\alpha_{1}}\mathbf{g}
,w^{2}(0)\frac{\partial^{\alpha}F}{\sqrt{\mu}}\big)
&\leq C\sum_{1\leq|\alpha_{1}|\leq|\alpha|}\int_{\mathbb{R}}|\partial^{\alpha_{1}}v||w(0)\partial^{\alpha-\alpha_{1}}\mathbf{g}|_{\sigma}
|w(0)\frac{\partial^{\alpha}F}{\sqrt{\mu}}|_{\sigma}\,dx
\\
&\leq C(\delta+\varepsilon_{0})\mathcal{D}_{2,l,q}(t)+C\delta(1+t)^{-\frac{4}{3}}.
\end{align*}
It follows from \eqref{3.6aa} and the above   estimates that
\begin{align*}
\big(\partial^{\alpha}(v\mathcal{L}\mathbf{g}),w^{2}(0)\frac{\partial^{\alpha}F}{\sqrt{\mu}}\big)\leq&
-\frac{c_{4}}{2}\|\partial^\alpha \mathbf{g}\|_{\sigma,w}^2+C\epsilon\|\partial^\alpha \mathbf{g}\|_{\sigma,w}^2
+C\|\partial^\alpha \mathbf{g}\|_{\sigma}^2+C_\epsilon\|\partial^\alpha[\widetilde{v},\widetilde{u},\widetilde{\theta}]\|^2
\notag\\
&\hspace{3cm}
+C_\epsilon(\delta+\varepsilon_{0})\mathcal{D}_{2,l,q}(t)+C_\epsilon\delta(1+t)^{-\frac{4}{3}}.
\end{align*}
By using \eqref{1.35} and \eqref{3.6}, one has
\begin{align*}
&|\big(\frac{1}{\sqrt{\mu}}\partial^{\alpha}\big[P_1\xi_1M\big\{\frac{|\xi-u|^2\overline{\theta}_x}{2R\theta^2}
+\frac{(\xi-u)\cdot \overline{u}_{x}}{R\theta}\big\}\big],w^{2}(0)\frac{\partial^{\alpha}F}{\sqrt{\mu}}\big)|
\\
&\leq \epsilon\|\langle \xi\rangle^{-\frac{1}{2}}w (0)\frac{\partial^\alpha F}{\sqrt{\mu}}\|^2
+C_\epsilon\|\langle \xi\rangle^{\frac{1}{2}}w (0)\frac{1}{\sqrt{\mu}}\partial^{\alpha}\big[P_1\xi_1M\big\{\frac{|\xi-u|^2\overline{\theta}_x}{2R\theta^2}
+\frac{(\xi-u)\cdot \overline{u}_{x}}{R\theta}\big\}\big]\|^2
\\
&\leq C\epsilon(\|\partial^\alpha \mathbf{g}\|_{\sigma,w}^2+\|\partial^\alpha[\widetilde{v},\widetilde{u},\widetilde{\theta}]\|^2)
+C_{\epsilon}(\delta+\varepsilon_{0})\mathcal{D}_{2,l,q}(t)+C_\epsilon\delta(1+t)^{-\frac{4}{3}}.
\end{align*}
Here we used the fact that $\|w(0)\langle \xi\rangle^b\mu^{-\frac{1}{2}} M^{1-\varepsilon}\|\leq C$
for any $b\ge 0$ and $\varepsilon$ small enough by \eqref{2.9a} and  \eqref{2.12}.

By using \eqref{5.29}, \eqref{5.33} and    \eqref{3.6},  we get
\begin{align*}
&|\big(\partial^{\alpha}[v\Gamma(\mathbf{g},\frac{M-\mu}{\sqrt{\mu}})]+
\partial^{\alpha}[v\Gamma(\frac{M-\mu}{\sqrt{\mu}},\mathbf{g})]
+\partial^{\alpha}[v\Gamma(\frac{G}{\sqrt{\mu}},\frac{G}{\sqrt{\mu}})],w^{2}(0)\frac{\partial^{\alpha}F}{\sqrt{\mu}}\big)|
\\
&\leq C(\eta_0+\delta+\varepsilon_{0})(\|\partial^\alpha \mathbf{g}\|_{\sigma,w}^2+\|\partial^\alpha[\widetilde{v},\widetilde{u},\widetilde{\theta}]\|^2)
+C(\eta_0+\delta+\varepsilon_{0})\mathcal{D}_{2,l,q}(t)
+C\delta(1+t)^{-\frac{4}{3}}.
\end{align*}
Substituting the above   estimates into \eqref{3.5} and taking a small $\epsilon>0$, we get
\begin{align}
\label{3.7}
&\sum_{|\alpha|=2}\big\{\frac{1}{2}\frac{d}{dt}\|v^{\frac{1}{2}}\frac{\partial^\alpha F}{\sqrt{\mu}}\|_w^2
+\frac{1}{2}q_{2}q_{3}(t)\|v^{\frac{1}{2}}\langle \xi\rangle\partial^\alpha \mathbf{g}\|_{w}^2+c\| \partial^\alpha \mathbf{g}\|_{\sigma,w}^2\big\}
\nonumber\\
&\leq C\sum_{|\alpha|=2}(\| \partial^\alpha \mathbf{g}\|_{\sigma}^2+\|\partial^\alpha[\widetilde{v},\widetilde{u},\widetilde{\theta}]\|^2)\notag\\
&\qquad+C(\eta_{0}+\delta+\varepsilon_{0})\mathcal{D}_{2,l,q}(t)+C\delta(1+t)^{-\frac{4}{3}}+Cq_{3}(t)\mathcal{F}_{2,l,q}(t).
\end{align}
Taking the summation of \eqref{3.4} and  \eqref{3.7}, we arrive at
\begin{align}
\label{3.8}
&\frac{d}{dt}\big\{\sum_{|\alpha|\leq 1}\|v^{\frac{1}{2}}\partial^\alpha \mathbf{g}\|_w^2
+\sum_{|\alpha|=2}\|v^{\frac{1}{2}}\frac{\partial^\alpha F}{\sqrt{\mu}}\|_w^2\big\}
+c\sum_{|\alpha|\leq 2}\big\{q_{2}q_{3}(t)\|v^{\frac{1}{2}}\langle \xi\rangle\partial^\alpha \mathbf{g}\|_{w}^2+ \|\partial^\alpha \mathbf{g}\|^2_{\sigma,w}\big\}
\nonumber\\
& \leq C\sum_{|\alpha|\leq 2}\|\partial^\alpha \mathbf{g}\|^2_{\sigma}
+C\sum_{1\leq|\alpha|\leq 2}\|\partial^{\alpha}[\widetilde{v},\widetilde{u},\widetilde{\theta}]\|^2\notag\\
&\qquad+C(\eta_{0}+\delta+\varepsilon_{0})\mathcal{D}_{2,l,q}(t)
+C\delta(1+t)^{-\frac{4}{3}}+Cq_{3}(t)\mathcal{F}_{2,l,q}(t).
\end{align}
This completes the proof of the weighted time-spatial energy estimates of the microscopic component $\mathbf{g}$.\qed

\subsection{Time-spatial-velocity derivative estimates}
In this subsection, we will deduce the weighted mixed derivative estimates of the microscopic component $\mathbf{g}$.
 Applying $\partial^\alpha_\beta$ to \eqref{2.4} with $|\alpha|+|\beta|\leq2$ and $|\beta|\geq1$, then we
take the inner product of the resulting equation with $w^2(\beta)\partial_{\beta}^\alpha \mathbf{g}$ over $\mathbb{R}\times{\mathbb R}^3$ to get
\begin{align}
\label{3.9}
&\big(\partial^\alpha_\beta (v\mathbf{g}_{t}),w^2(\beta)\partial_{\beta}^\alpha \mathbf{g}\big)+\big(\xi_{1}\partial^\alpha_\beta \mathbf{g}_{x},w^2(\beta)\partial_{\beta}^\alpha \mathbf{g}\big)+\big(C_\beta^{\beta-e_1}\partial^\alpha_{\beta-e_1}\mathbf{g}_{x},w^2(\beta)\partial_{\beta}^\alpha \mathbf{g}\big)\notag\\
&\qquad-\big(\partial^\alpha_\beta(u_{1}\mathbf{g}_{x}),w^2(\beta)\partial_{\beta}^\alpha \mathbf{g}\big)-\big(\partial^\alpha_\beta(v\mathcal{L}\mathbf{g}),w^2(\beta)\partial_{\beta}^\alpha \mathbf{g}\big)\notag\\
&=\big(\partial^\alpha_\beta[v\Gamma(\mathbf{g},\frac{M-\mu}{\sqrt{\mu}})]+
\partial^\alpha_\beta[v\Gamma(\frac{M-\mu}{\sqrt{\mu}},\mathbf{g})]+\partial^\alpha_\beta[v\Gamma(\frac{G}{\sqrt{\mu}},\frac{G}{\sqrt{\mu}})]
,w^2(\beta)\partial_{\beta}^\alpha \mathbf{g}\big)
\nonumber\\
&\qquad+\big(\partial^\alpha_\beta\big[\frac{P_{0}(\xi_{1}\sqrt{\mu}\mathbf{g}_{x})}{\sqrt{\mu}}\big]
-\partial^\alpha_\beta\big[\frac{1}{\sqrt{\mu}}P_{1}\xi_{1}M\big\{\frac{|\xi-u|^{2}
\widetilde{\theta}_{x}}{2R\theta^{2}}+\frac{(\xi-u)\cdot\widetilde{u}_{x}}{R\theta}\big\}\big]
,w^2(\beta)\partial_{\beta}^\alpha \mathbf{g}\big)
\nonumber\\
&\qquad+\big(-\partial^\alpha_\beta[\frac{P_{1}(\xi_{1}\overline{G}_{x})}{\sqrt{\mu}}]+\partial^\alpha_\beta[u_{1}\frac{\overline{G}_{x}}{\sqrt{\mu}}]
-\partial^\alpha_\beta[v\frac{\overline{G}_{t}}{\sqrt{\mu}}],w^2(\beta)\partial_{\beta}^\alpha \mathbf{g}\big),
\end{align}
where $e_1=(1,0,0)$. We shall estimate each term in \eqref{3.9}.
Since $|\alpha|+|\beta|\leq2$ and $|\beta|\geq1$, then $|\alpha|\leq1$.
If $|\alpha|=1$, it can be seen that
$$
\partial^\alpha_\beta (v\mathbf{g}_{t})=v\partial^\alpha_\beta \mathbf{g}_{t}+\partial^{\alpha}v\partial_\beta \mathbf{g}_{t}.
$$
If $|\alpha|=0$, the last term in the above equality vanishes.  In view of \eqref{2.48aa} and \eqref{2.48},  we deduce
\begin{align}
\label{3.14a}
|(\partial^{\alpha}v\partial_\beta \mathbf{g}_{t},w^2(\beta)\partial_{\beta}^\alpha \mathbf{g})|
\leq C\epsilon(\|\partial_{\beta}^{\alpha}\mathbf{g}\|_{\sigma,w}^{2}+\sum_{|\beta|=1}\|\partial_{\beta}\mathbf{g}_{t}\|_{\sigma,w}^{2})+C_{\epsilon}q_{3}(t)\mathcal{F}_{2,l,q}(t).
\end{align}
By the similar arguments as \eqref{3.14a}, one has
\begin{align*}
|(v_{t} \partial_{\beta}^\alpha \mathbf{g},w^2(\beta)\partial_{\beta}^\alpha \mathbf{g})|
\leq C\epsilon \|\partial_{\beta}^{\alpha}\mathbf{g}\|_{\sigma,w}^{2} +C_{\epsilon}q_{3}(t)\mathcal{F}_{2,l,q}(t).
\end{align*}
By using  this and \eqref{3.3}, we have
\begin{multline}
\label{3.15a}
 (v\partial^\alpha_\beta \mathbf{g}_{t},w^2(\beta)\partial_{\beta}^\alpha \mathbf{g})
=\frac{1}{2}\frac{d}{dt}\|v^{\frac{1}{2}}\partial_{\beta}^\alpha \mathbf{g}\|_w^2
-\frac{1}{2}(v\partial_{\beta}^\alpha \mathbf{g} ,[w^{2}(\beta)]_{t}\partial_{\beta}^\alpha \mathbf{g})
-\frac{1}{2}(v_{t} \partial_{\beta}^\alpha \mathbf{g},w^2(\beta)\partial_{\beta}^\alpha \mathbf{g})\\
\geq\frac{1}{2}\frac{d}{dt}\|v^{\frac{1}{2}}\partial_{\beta}^\alpha \mathbf{g}\|_w^2
+q_{2}q_{3}(t)\|v^{\frac{1}{2}}\langle \xi\rangle\partial_{\beta}^\alpha \mathbf{g}\|^2_w
-C\epsilon\|\partial_{\beta}^\alpha \mathbf{g}\|^2_{\sigma,w}-C_{\epsilon}q_{3}(t)\mathcal{F}_{2,l,q}(t).
\end{multline}
For the first term on the left hand side of \eqref{3.9}, by  \eqref{3.14a} and \eqref{3.15a}, we arrive at
\begin{align*}
\big(\partial^\alpha_\beta (v\mathbf{g}_{t}),w^2(\beta)\partial_{\beta}^\alpha \mathbf{g}\big)
&\geq \frac{1}{2}\frac{d}{dt}\|v^{\frac{1}{2}}\partial_{\beta}^\alpha \mathbf{g}\|_w^2
+q_{2}q_{3}(t)\|v^{\frac{1}{2}}\langle \xi\rangle\partial_{\beta}^\alpha \mathbf{g}\|^2_w
-C\epsilon\|\partial_{\beta}^\alpha \mathbf{g}\|^2_{\sigma,w}
\nonumber\\
&\qquad- C\epsilon\sum_{|\beta|=1}\|\partial_{\beta}\mathbf{g}_{t}\|_{\sigma,w}^{2}-C_{\epsilon}q_{3}(t)\mathcal{F}_{2,l,q}(t).
\end{align*}
By integration by parts, we  see
$$
\big(\xi_{1}\partial^\alpha_\beta \mathbf{g}_{x},w^2(\beta)\partial_{\beta}^\alpha \mathbf{g}\big)=0.
$$
For the third term on the left hand side of \eqref{3.9}, by  the H\"{o}lder inequality and \eqref{1.35}, one has
\begin{eqnarray*}
|\big(\partial^\alpha_{\beta-e_1}\mathbf{g}_{x},w^2(\beta)\partial^\alpha_\beta \mathbf{g}\big)|
&&\leq \epsilon\|\langle \xi\rangle^{-\frac{1}{2}}w(\beta)\partial^\alpha_{\beta}\mathbf{g}\|^2+C_\epsilon\|\langle \xi\rangle^{-\frac{1}{2}}\langle \xi\rangle w(\beta) \partial^\alpha_{\beta-e_1}\mathbf{g}_{x}\|^2
\notag\\
&&\leq\epsilon\|\langle \xi\rangle^{-\frac{1}{2}}w(\beta)\partial^\alpha_{\beta}\mathbf{g}\|^2+C_\epsilon\|\langle \xi\rangle^{-\frac{1}{2}}w(\beta-e_1) \partial^\alpha_{\beta-e_1}\mathbf{g}_{x}\|^2
\notag\\
&&\leq C\epsilon\|\partial^\alpha_{\beta}\mathbf{g}\|^2_{\sigma,w}+C_\epsilon\|\partial^\alpha_{\beta-e_1}\mathbf{g}_{x}\|^2_{\sigma,w(\beta-e_1)}.
\end{eqnarray*}
Here   we have used  the fact that $\langle \xi\rangle\langle \xi\rangle^{(l-|\beta|)}=\langle \xi\rangle^{(l-|\beta-e_1|)}$.

Since  $|\alpha|+|\beta|\leq 2$ and $|\beta|\geq 1$, then $|\alpha|\leq 1$ and $1\leq |\beta|\leq 2$. For any $|\alpha|=1$, one has
$$\partial^\alpha_\beta(u_{1}\mathbf{g}_{x})=u_{1}\partial^\alpha_\beta\mathbf{g}_{x}+ \partial^{\alpha }u_{1}\partial_\beta \mathbf{g}_{x}.$$
If $|\alpha|=0$, the last term in the above equality vanishes. In view of \eqref{2.48aaa} and \eqref{2.48},  we deduce
$$|(u_{1}\partial^\alpha_\beta\mathbf{g}_{x},w^2(\beta)\partial_{\beta}^\alpha \mathbf{g})|=\frac{1}{2}|(u_{1x}\partial^\alpha_\beta\mathbf{g},w^2(\beta)\partial_{\beta}^\alpha \mathbf{g})|\leq C\epsilon\|\partial^\alpha_\beta\mathbf{g}\|_{\sigma,w}^{2}+C_{\epsilon}q_{3}(t)\mathcal{F}_{2,l,q}(t).$$
Similarly we also have
$$|(\partial^{\alpha }u_{1}\partial_\beta \mathbf{g}_{x},w^2(\beta)\partial_{\beta}^\alpha \mathbf{g})| \leq C\epsilon(\|\partial^\alpha_\beta\mathbf{g}\|_{\sigma,w}^{2}+\sum_{|\beta|=1}\|\partial_\beta\mathbf{g}_x\|_{\sigma,w}^{2})+C_{\epsilon}q_{3}(t)\mathcal{F}_{2,l,q}(t).$$
For the fourth term on the left hand side of \eqref{3.9}, we get
\begin{align*}
|\big(\partial^\alpha_\beta(u_{1}\mathbf{g}_{x}),w^2(\beta)\partial_{\beta}^\alpha \mathbf{g}\big)|
\leq C\epsilon(\|\partial^\alpha_\beta\mathbf{g}\|_{\sigma,w}^{2}+\sum_{|\beta|=1}\|\partial_\beta\mathbf{g}_x\|_{\sigma,w}^{2})+C_{\epsilon}q_{3}(t)\mathcal{F}_{2,l,q}(t).
\end{align*}
For any $|\alpha|=1$, one has
$$\partial^\alpha_\beta(v\mathcal{L}\mathbf{g})=v\partial^\alpha_\beta\mathcal{L}\mathbf{g}+ \partial^{\alpha }v\partial_\beta \mathcal{L}\mathbf{g}.$$
If $|\alpha|=0$, the last term in the above equality vanishes. In view of \eqref{5.5} and \eqref{2.12},  we deduce
\begin{eqnarray*}
-(v\partial^\alpha_\beta\mathcal{L}\mathbf{g},w^2(\beta)\partial_{\beta}^\alpha \mathbf{g})
&&=-(\partial_\beta\mathcal{L}(v^{\frac{1}{2}}\partial^\alpha\mathbf{g}),w^2(\beta)\partial_\beta (v^{\frac{1}{2}}\partial^\alpha\mathbf{g}))\notag\\
&&
\geq \|v^{\frac{1}{2}}\partial^\alpha_\beta \mathbf{g}\|^2_{\sigma,w}-\epsilon\sum_{|\beta_1|=|\beta|}\|v^{\frac{1}{2}}\partial^\alpha_{\beta_1} \mathbf{g}\|_{\sigma,w(\beta_1)}^2
-C_\epsilon\sum_{|\beta_1|<|\beta|}\|v^{\frac{1}{2}}\partial^\alpha_{\beta_1}\mathbf{g}\|_{\sigma,w(\beta_1)}^2
\notag\\
&&\geq c\| \partial^\alpha_\beta \mathbf{g}\|^2_{\sigma,w}-C\epsilon\sum_{|\beta_1|=|\beta|}\| \partial^\alpha_{\beta_1} \mathbf{g}\|_{\sigma,w(\beta_1)}^2
-C_\epsilon\sum_{|\beta_1|<|\beta|}\| \partial^\alpha_{\beta_1}\mathbf{g}\|_{\sigma,w(\beta_1)}^2.
\end{eqnarray*}
Recalling that $\mathcal{L}\mathbf{g}=\Gamma(\sqrt{\mu},\mathbf{g})+\Gamma(\mathbf{g},\sqrt{\mu})$,  by using \eqref{5.8}, \eqref{1.35} and the imbedding inequality, one has
\begin{eqnarray*}
&&|(\partial^{\alpha }v\partial_\beta \mathcal{L}\mathbf{g},w^2(\beta)\partial_{\beta}^\alpha \mathbf{g})|
\leq \|\partial^{\alpha }v\|_{L^\infty_x}\sum_{|\beta'|\leq |\beta|} \|\partial_{\beta'} \mathbf{g}\|_{\sigma,w}\| \partial^\alpha_\beta \mathbf{g}\|_{\sigma,w}\leq C(\delta+\varepsilon_{0})\mathcal{D}_{2,l,q}(t).
\end{eqnarray*}
For the fifth term on the left hand side of \eqref{3.9}, we arrive at
\begin{align*}
-\big(\partial^\alpha_\beta(v\mathcal{L}\mathbf{g}),w^2(\beta)\partial^\alpha_\beta \mathbf{g}\big)
&\geq c\|\partial^\alpha_\beta \mathbf{g}\|^2_{\sigma,w}-C\epsilon\sum_{|\beta_1|=|\beta|}\|\partial^\alpha_{\beta_1} \mathbf{g}\|_{\sigma,w(\beta_1)}^2
\\
&\qquad-C_\epsilon\sum_{|\beta_1|<|\beta|}\|\partial^\alpha_{\beta_1}\mathbf{g}\|_{\sigma,w(\beta_1)}^2-
C(\delta+\varepsilon_{0})\mathcal{D}_{2,l,q}(t).
\end{align*}
By using  \eqref{5.29} and \eqref{5.33}, we get
\begin{align*}
&|\big(\partial^\alpha_\beta[v\Gamma(\mathbf{g},\frac{M-\mu}{\sqrt{\mu}})]+
\partial^\alpha_\beta[v\Gamma(\frac{M-\mu}{\sqrt{\mu}},\mathbf{g})]+\partial^\alpha_\beta[v\Gamma(\frac{G}{\sqrt{\mu}},\frac{G}{\sqrt{\mu}})]
,w^2(\beta)\partial_{\beta}^\alpha \mathbf{g}\big)|
\\
&\leq C(\eta_0+\delta+\varepsilon_{0})\|\partial_{\beta}^\alpha \mathbf{g}\|^2_{\sigma,w}
+C(\eta_0+\delta+\varepsilon_{0})\mathcal{D}_{2,l,q}(t)
+C\delta(1+t)^{-\frac{4}{3}}.
\end{align*}
By \eqref{1.20a}, \eqref{1.35}, \eqref{2.35a} and the H\"{o}lder inequality, we get
\begin{align*}
&|\big(\partial^\alpha_\beta[\frac{P_{0}(\xi_{1}\sqrt{\mu}\mathbf{g}_{x})}{\sqrt{\mu}}],w^2(\beta)\partial_{\beta}^\alpha \mathbf{g}\big)|
=|\sum^{4}_{j=0}\big(\langle\xi\rangle^{\frac{1}{2}}w(\beta)\partial^\alpha_\beta[\langle\xi_{1}\sqrt{\mu}\mathbf{g}_{x},
\frac{\chi_{j}}{M}\rangle\frac{\chi_{j}}{\sqrt{\mu}}],
\langle\xi\rangle^{-\frac{1}{2}}w(\beta)\partial_{\beta}^\alpha \mathbf{g}\big)|
\\
&\leq C\epsilon\|\partial^\alpha_\beta \mathbf{g}\|^2_{\sigma,w}+
C_\epsilon\|\partial^\alpha \mathbf{g}_x\|_\sigma^2
+C_\epsilon(\delta+\varepsilon_{0})\mathcal{D}_{2,l,q}(t)+C_{\epsilon}\delta(1+t)^{-\frac{4}{3}},
\end{align*}
and
\begin{align*}
&|\big(\partial^\alpha_\beta[\frac{1}{\sqrt{\mu}}P_{1}\xi_{1}M\big\{\frac{|\xi-u|^{2}
\widetilde{\theta}_{x}}{2R\theta^{2}}+\frac{(\xi-u)\cdot\widetilde{u}_{x}}{R\theta}\big\}]
,w^2(\beta)\partial_{\beta}^\alpha \mathbf{g}\big)|
\\
&=|\big(\langle\xi\rangle^{\frac{1}{2}}w(\beta)\partial^\alpha_\beta[\frac{1}{\sqrt{\mu}}P_{1}\xi_{1}M\big\{\frac{|\xi-u|^{2}
\widetilde{\theta}_{x}}{2R\theta^{2}}+\frac{(\xi-u)\cdot\widetilde{u}_{x}}{R\theta}\big\}]
,\langle\xi\rangle^{-\frac{1}{2}}w(\beta)\partial_{\beta}^\alpha \mathbf{g}\big)|
\\
&\leq C\epsilon\|\partial^\alpha_\beta \mathbf{g}\|^2_{\sigma,w}+C_\epsilon\|[\partial^\alpha\widetilde {u}_{x},\partial^\alpha\widetilde{\theta}_x]\|^2
+C_{\epsilon}(\delta+\varepsilon_{0})\mathcal{D}_{2,l,q}(t)+C_{\epsilon}\delta(1+t)^{-\frac{4}{3}}.
\end{align*}
Here we used the inequality $|\langle \xi\rangle^bw(\beta)\mu^{-\frac{1}{2}}M|_2\leq C$
for any $b\ge 0$ by \eqref{2.9a} and  \eqref{2.12}.

For the terms containing $\overline{G}$ in \eqref{3.9}, we have from   \eqref{5.38} and \eqref{1.35} that
\begin{align*}
&|\big(-\partial^\alpha_\beta[\frac{P_{1}(\xi_{1}\overline{G}_{x})}{\sqrt{\mu}}]+\partial^\alpha_\beta[u_{1}\frac{\overline{G}_{x}}{\sqrt{\mu}}]
-\partial^\alpha_\beta[v\frac{\overline{G}_{t}}{\sqrt{\mu}}],w^2(\beta)\partial_{\beta}^\alpha \mathbf{g}\big)|
\\
&\leq C\epsilon\|\langle\xi\rangle^{-\frac{1}{2}}w(\beta)\partial^\alpha_\beta \mathbf{g}\|^2
+C_{\epsilon}\|\langle\xi\rangle^{\frac{1}{2}}w(\beta)\partial^\alpha_\beta[\frac{P_{1}(\xi_{1}\overline{G}_{x})}{\sqrt{\mu}}]\|^2
\\
&\qquad+C_{\epsilon}\|\langle\xi\rangle^{\frac{1}{2}}w(\beta)\partial^\alpha_\beta[u_{1}\frac{\overline{G}_{x}}{\sqrt{\mu}}]\|^2
+C_{\epsilon}\|\langle\xi\rangle^{\frac{1}{2}}w(\beta)\partial^\alpha_\beta[v\frac{\overline{G}_{t}}{\sqrt{\mu}}]\|^2
\\
&\leq C\epsilon\| \partial^\alpha_\beta \mathbf{g}\|_{\sigma,w}^2
+C_{\epsilon}(\delta+\varepsilon_{0})\mathcal{D}_{2,l,q}(t)+C_{\epsilon}\delta(1+t)^{-\frac{4}{3}}.
\end{align*}

If $|\alpha|+|\beta|\leq 2$ with  $|\beta|\geq 1$, for  any small $\epsilon>0$, we have from \eqref{3.9} and the above estimates that
\begin{align}
\label{3.11}
&\frac{1}{2}\frac{d}{dt}\|v^{\frac{1}{2}}\partial^\alpha_\beta \mathbf{g}\|_w^2
+ q_{2}q_{3}(t)\|v^{\frac{1}{2}}\langle \xi\rangle\partial_{\beta}^\alpha \mathbf{g}\|^2_w+c\|\partial^\alpha_\beta \mathbf{g}\|^2_{\sigma,w}
\nonumber\\
&\leq  C\epsilon(\|\partial^\alpha_\beta \mathbf{g}\|^2_{\sigma,w}+\sum_{|\beta'|=1}\|\partial_{\beta'}\mathbf{g}_{t}\|_{\sigma,w}^{2}+\sum_{|\beta'|=1}\|\partial_{\beta'}\mathbf{g}_{x}\|_{\sigma,w}^{2})
+C_\epsilon(\|[\partial^\alpha\widetilde {u}_{x},\partial^\alpha\widetilde{\theta}_x]\|^2+\|\partial^\alpha \mathbf{g}_x\|_{\sigma}^2)
\nonumber\\
&\qquad+ C\epsilon\sum_{|\beta_1|=|\beta|}\|\partial^\alpha_{\beta_1}\mathbf{g}\|_{\sigma,w(\beta_{1})}^2
+C_\epsilon\|\partial^\alpha_{\beta-e_1}\mathbf{g}_{x}\|^2_{\sigma,w(\beta-e_1)}
+C_\epsilon\sum_{|\beta_1|<|\beta|}\|\partial^\alpha_{\beta_1}\mathbf{g}\|_{\sigma,w(\beta_{1})}^2
\nonumber\\
&\qquad+C_\epsilon\delta(1+t)^{-\frac{4}{3}}
+C_\epsilon( \eta_0+\delta+\varepsilon_{0})\mathcal{D}_{2,l,q}(t)+C_\epsilon q_{3}(t)\mathcal{F}_{2,l,q}(t).
\end{align} 
Notice that the coefficients in the last two terms of the third line in \eqref{3.11} are large. We will use the induction in $|\beta|$ and then choose suitably  small $\epsilon>0$ to control these terms. By the suitable linear combinations, there exist some positive constants $c_{5}$,
and $C_{ \alpha,\beta }$  such that
\begin{align}
\label{3.12}
   \sum_{|\alpha|+|\beta|\leq 2,|\beta|\geq 1}&\big\{ \frac{d}{dt}(C_{ \alpha,\beta }\|v^{\frac{1}{2}}\partial^\alpha_\beta \mathbf{g}\|_w^2)
+c_{5}q_{2}q_{3}(t)\|v^{\frac{1}{2}}\langle \xi\rangle\partial_{\beta}^\alpha \mathbf{g}\|^2_w
+c_{5}\|\partial^\alpha_\beta \mathbf{g}\|^2_{\sigma,w}\big\}
\nonumber\\
&\leq C\sum_{ |\alpha|\leq 1}(\|[\partial^\alpha\widetilde{\theta}_x,\partial^\alpha\widetilde {u}_{x}]\|^2+\|\partial^\alpha \mathbf{g}\|_{\sigma,w}^2+\|\partial^\alpha \mathbf{g}_x\|_{\sigma,w}^2)
\nonumber\\
&\qquad +C\delta(1+t)^{-\frac{4}{3}}
+C( \eta_0+\delta+\varepsilon_{0})\mathcal{D}_{2,l,q}(t)+Cq_{3}(t)\mathcal{F}_{2,l,q}(t).
\end{align}
For some large constant  $\widetilde{C}_{6}>0$, we denote $E_{2}(t)$ as
\begin{align}
\label{3.14}
E_{2}(t)=\widetilde{C}_{6}\big\{\sum_{|\alpha|\leq 1}\|v^{\frac{1}{2}}\partial^\alpha \mathbf{g}\|_w^2
+\sum_{|\alpha|=2}\|v^{\frac{1}{2}}\frac{\partial^\alpha F}{\sqrt{\mu}}\|_w^2\big\}
+\sum_{|\alpha|+|\beta|\leq 2,|\beta|\geq 1} C_{ \alpha,\beta }\|v^{\frac{1}{2}}\partial^\alpha_\beta \mathbf{g}\|_w^2.
\end{align}
By using this, \eqref{2.12} and a suitable linear combination of \eqref{3.8} and \eqref{3.12}, we arrive at
\begin{align}
\label{3.13}
&\frac{d}{dt}E_{2}(t)+\frac{c_{5}}{4}q_{2}q_{3}(t)\mathcal{F}_{2,l,q}(t)+\frac{c_{5}}{4}\sum_{|\alpha|+|\beta|\leq 2}\|\partial^\alpha \mathbf{g}\|^2_{\sigma,w}
\nonumber\\
&\leq C(\sum_{1\leq|\alpha|\leq 2}\|\partial^{\alpha}[\widetilde{v},\widetilde{u},\widetilde{\theta}]\|^2
+\sum_{|\alpha|\leq 2}\|\partial^\alpha \mathbf{g}\|^2_{\sigma})+Cq_{3}(t)\mathcal{F}_{2,l,q}(t)
\nonumber\\
&\qquad+C(\eta_{0}+\delta+\varepsilon_{0})\mathcal{D}_{2,l,q}(t)
+C\delta(1+t)^{-\frac{4}{3}}.
\end{align}
Here $\mathcal{D}_{2,l,q}(t)$, $q_{3}(t)$ and $\mathcal{F}_{2,l,q}(t)$ are defined by \eqref{2.11}, \eqref{1.33} and \eqref{2.48}, respectively.
This completes the proof of the weighted energy estimates for the microscopic component $\mathbf{g}$.\qed

\section{Global existence and large time behavior}\label{sec.4}
In this section, we will establish our main theorem by the energy estimates derived in section \ref{sec.2} and \ref{sec.3}.

\medskip
\noindent{\it Proof of Theorem \ref{thm1.1}:}
For some large positive constant  $\widetilde{C}_{7}$ with $\widetilde{C}_{7}\gg\widetilde{C}_{6}$, by \eqref{2.59} and \eqref{3.14},  we define $\overline{\mathcal{E}}_{2,l,q}(t)$ as
\begin{align}
\label{4.3}
\overline{\mathcal{E}}_{2,l,q}(t)&=\widetilde{C}_{7}E_{1}(t)+E_{2}(t)\notag\\
&=\widetilde{C}_{7}\Big\{\int_{\mathbb{R}}\Big(\widetilde{C}_1\big(\frac{2}{3}\bar{\theta}\Phi(\frac{v}{\bar{v}})+\frac{1}{2}\widetilde{u}^{2}
+\bar{\theta}\Phi(\frac{\theta}{\bar{\theta}})\big)-\widetilde{C}_1\kappa_{1}\widetilde{u}_{1}\widetilde{v}_{x}\Big)\,dx
+\|v^{\frac{1}{2}}\mathbf{g}\|^{2}\Big\}
\nonumber\\
&\quad+\widetilde{C}_{7}\widetilde{C}_{5}\widetilde{C}_{3}\Big\{\sum_{|\alpha|=1}\widetilde{C}_{2}
\int_{\mathbb{R}}\big(\frac{p_{+}}{2}|\partial^{\alpha}\widetilde{v}|^{2}+\frac{v}{2}|\partial^{\alpha}\widetilde{u}_{1}|^{2}
+\frac{1}{2}\sum^{3}_{i=2}|\partial^{\alpha}\widetilde{u}_{i}|^{2}+\frac{R}{2p_{+}}|\partial^{\alpha}\widetilde{\theta}|^{2}\big)\,dx\notag\\
&\qquad\qquad\qquad\qquad\qquad\qquad-(\widetilde{u}_{1x},\widetilde{v}_{xx})\Big\}
\nonumber\\
&\quad+\widetilde{C}_{7}\widetilde{C}_{5}\Big\{\sum_{|\alpha|=1}\|v^{\frac{1}{2}}\partial^{\alpha}\mathbf{g}\|^{2}
+\widetilde{C}_{4}\sum_{|\alpha|=2}\|v^{\frac{1}{2}}\frac{\partial^{\alpha}F}{\sqrt{\mu}}\|^{2}\Big\}
\nonumber\\
&\quad+\widetilde{C}_{6}\Big\{\sum_{|\alpha|\leq 1}\|v^{\frac{1}{2}}\partial^\alpha \mathbf{g}\|_w^2
+\sum_{|\alpha|=2}\|v^{\frac{1}{2}}\frac{\partial^\alpha F}{\sqrt{\mu}}\|_w^2\Big\}
+\sum_{|\alpha|+|\beta|\leq 2,|\beta|\geq 1} C_{ \alpha,\beta }\|v^{\frac{1}{2}}\partial^\alpha_\beta \mathbf{g}\|_w^2.
\end{align}
By using this  and a suitable linear combination of \eqref{2.58} and \eqref{3.13}, we arrive at
\begin{multline}
\label{4.1}
\frac{d}{dt}\overline{\mathcal{E}}_{2,l,q}(t)+2c_{6}q_{2}q_{3}(t)\mathcal{F}_{2,l,q}(t)+2c_{6}
\sum_{1\leq|\alpha|\leq2}\|\partial^\alpha[\widetilde{v} ,\widetilde{u} ,\widetilde{\theta}]\|^2
+2c_{6}\sum_{|\alpha|+|\beta|\leq2}\|\partial^\alpha_\beta \mathbf{g}\|^2_{\sigma,w(\beta)}\\
\leq C_{6}(\eta_{0}+\delta+\varepsilon_{0})\mathcal{D}_{2,l,q}(t)
+C_{6}\delta(1+t)^{-\frac{4}{3}}+C_{6}\delta\int_{\mathbb{R}}(\widetilde{v}^{2}+\widetilde{\theta}^{2})\omega^{2}\,dx
+C_{6}q_{3}(t)\mathcal{F}_{2,l,q}(t).
\end{multline}
Here $c_{6}$ and $C_{6}$ are some given positive constants.
Integrating \eqref{4.1} with respect to $t$ and taking $\eta_0$, $\delta$ and $\varepsilon_{0}$ small enough, we have from \eqref{2.11} and
Lemma \ref{lem5.6} that
\begin{align}
\label{4.2}
&\overline{\mathcal{E}}_{2,l,q}(t)+2c_{6}q_{2}\int^{t}_{0}q_{3}(s)\mathcal{F}_{2,l,q}(s)\,ds+c_{6}\int^{t}_{0}\mathcal{D}_{2,l,q}(s)\,ds\nonumber\\
&
\leq \overline{\mathcal{E}}_{2,l,q}(0)+3C_{6}\delta+C_{6}C_{2}\delta+C_{6}\int^{t}_{0}q_{3}(s)\mathcal{F}_{2,l,q}(s)\,ds.
\end{align}
Due to the fact that $q_{2}=\frac{1}{\tilde{C}_0\sqrt{\varepsilon_{0}}}$ for some sufficiently small $\varepsilon_{0}>0$, then
$C_{6}< c_{6}q_{2}$. It follows from \eqref{4.2} that
\begin{align}
\label{4.4}
\overline{\mathcal{E}}_{2,l,q}(t)+c_{6}q_{2}\int^{t}_{0}q_{3}(s)\mathcal{F}_{2,l,q}(s)\,ds+c_{6}\int^{t}_{0}\mathcal{D}_{2,l,q}(s)\,ds
\leq \overline{\mathcal{E}}_{2,l,q}(0)+3C_{6}\delta+C_{6}C_{2}\delta.
\end{align}
By the definition of $\overline{\mathcal{E}}_{2,l,q}(t)$ in \eqref{4.3} and $\mathcal{E}_{2,l,q}(t)$
in \eqref{2.10}, there exists a constant $C_{7}>1$ such that
\begin{eqnarray}
\label{4.5}
C_{7}^{-1}(\mathcal{E}_{2,l,q}(t)-\delta)\leq\overline{\mathcal{E}}_{2,l,q}(t)
\leq C_{7}(\mathcal{E}_{2,l,q}(t)+\delta).
\end{eqnarray}
We thus derive from \eqref{4.4} and \eqref{4.5} that
\begin{align*}
\sup_{0\leq t\leq T}\mathcal{E}_{2,l,q}(t)&\leq C_{7}\sup_{0\leq t\leq T}\overline{\mathcal{E}}_{2,l,q}(t)
+\delta\leq C_{7}(\overline{\mathcal{E}}_{2,l,q}(0)+3C_{6}\delta+C_{6}C_{2}\delta)+\delta
\nonumber\\
&<(C_{7})^{2}(\mathcal{E}_{2,l,q}(0)+\delta)+C_{7}(3C_{6}\delta+C_{6}C_{2}\delta+\delta).
\end{align*}
If we choose $C_0=6(C_{7})^{2}+6C_{7}(3C_{6}+C_{6}C_{2}+1)$ in \eqref{1.39}, we arrive at
\begin{align}
\label{4.6}
 \sup_{0\leq t\leq T}\mathcal{E}_{2,l,q}(t)<\frac{1}{3}C_0(\mathcal{E}_{2,l,q}(0)+\delta)
<\frac{1}{2}\varepsilon^{2}_{0}.
\end{align}
If we choose $C_{1}=c_{6}$ in \eqref{2.9}, for any $t\in(0,T]$,  we have from \eqref{4.4}, \eqref{4.5} and \eqref{4.6} that
\begin{align}
\label{4.6a}
 C_{1}\int^{t}_{0}\mathcal{D}_{2,l,q}(s)\,ds
\leq \overline{\mathcal{E}}_{2,l,q}(0)+3C_{6}\delta+C_{6}C_{2}\delta<\frac{1}{3}C_0(\mathcal{E}_{2,l,q}(0)+\delta)
<\frac{1}{2}\varepsilon^{2}_{0}.
\end{align}
Thus the a priori assumption \eqref{2.9} can be closed by \eqref{4.6} and \eqref{4.6a}.

The local existence of the solutions to the
Landau system \eqref{1.1} near a global Maxwellian was proved in \cite{Guo-2002}. By
the estimates of $[\bar{v},\bar{u},\bar{\theta}]$ and a straightforward modification of the arguments there, we can obtain the local existence
of the solutions to the Landau equation \eqref{1.16} and \eqref{1.17} with $F(t, x, \xi)\geq 0$
under the assumptions in Theorem \ref{thm1.1}.
Hence, by the uniform   estimates  and the local existence of the solution,
the standard continuity argument gives the existence and uniqueness of global  solution
to the Landau equation \eqref{1.16} with initial data \eqref{1.17}. For any $t>0$, we also have
\begin{align}
\label{4.9}
\mathcal{E}_{2,l,q}(t)
+C\int^{t}_{0}\mathcal{D}_{2,l,q}(s)\,ds\leq \varepsilon^{2}_{0}.
\end{align}

We are going to justify the time asymptotic stability of contact waves as
\eqref{1.41}. By the expression of $M$ in \eqref{1.5} with $\rho=1/v$ and $\overline{M}$ in \eqref{1.36}, we have from \eqref{2.12}, \eqref{5.38} and the
imbedding inequality that
$$
\|\frac{(M-\overline{M})_{x}}{\sqrt{\mu}}\|^2+\|\frac{\overline{G}_{x}}{\sqrt{\mu}}\|^{2}
\leq  C \delta(1+t)^{-\frac{4}{3}}+C\mathcal{D}_{2,l,q}(t).
$$
By \eqref{1.35} and  \eqref{2.11}, for any $l\geq 2$, one has
$$
\|\mathbf{g}_{x}\|^{2}\leq\|\langle \xi\rangle^{\frac{1}{2}}\mathbf{g}_{x}\|_{\sigma}^{2}\leq
\|\langle \xi\rangle^{l}\mathbf{g}_{x}\|_{\sigma}^{2}\leq C\mathcal{D}_{2,l,q}(t).
$$
Since $F=M+\overline{G}+\sqrt{\mu}\mathbf{g}$, we have from the above two estimates that
\begin{align*}
\|\frac{(F-\overline{M})_{x}}{\sqrt{\mu}}\|^2
\leq C\big\{\|\frac{(M-\overline{M})_{x}}{\sqrt{\mu}}\|^2+\|\frac{\overline{G}_{x}}{\sqrt{\mu}}\|^{2}
+\|\mathbf{g}_{x}\|^{2}\big\}
\leq C\mathcal{D}_{2,l,q}(t)+C \delta(1+t)^{-\frac{4}{3}}.
\end{align*}
Similarly, it holds that
\begin{align*}
\|\frac{(F-\overline{M})_{xt}}{\sqrt{\mu}}\|^2
\leq C\big\{\|\frac{(M-\overline{M})_{xt}}{\sqrt{\mu}}\|^2+\|\frac{\overline{G}_{xt}}{\sqrt{\mu}}\|^{2}
+\|\mathbf{g}_{xt}\|^{2}\big\}
\leq C\mathcal{D}_{2,l,q}(t)+C \delta(1+t)^{-\frac{4}{3}}.
\end{align*}
By  \eqref{4.9} and the above two estimates, one has
\begin{align*}
\int_{0}^{+\infty}\|\frac{(F-\overline{M})_{x}}{\sqrt{\mu}}\|^{2}\,dt
+\int_{0}^{+\infty}|\frac{d}{dt}\|\frac{(F-\overline{M})_{x}}{\sqrt{\mu}}\|^{2}|\,dt
<C(\varepsilon^{2}_0+ \delta),
\end{align*}
which implies that
\begin{align*}
\lim_{t\to +\infty}\|\frac{(F-\overline{M})_{x}}{\sqrt{\mu}}\|^{2}=0.
\end{align*}
By using the  imbedding inequality, we get
$$
\sup_{x\in\mathbb{R}}|\frac{(F-\overline{M})}{\sqrt{\mu}}|_{2}^{2}\leq 2\|\frac{F-\overline{M}}{\sqrt{\mu}}\|\|\frac{(F-\overline{M})_{x}}{\sqrt{\mu}}\|.
$$
It follows that
\begin{align*}
&\lim_{t\to +\infty}\sup_{x\in\mathbb{R}}|\frac{(F-\overline{M})}{\sqrt{\mu}}|_{2}^{2}=0.
\end{align*}
This gives \eqref{1.41} and then completes the proof of Theorem \ref{thm1.1}.\qed

\section{Appendix}\label{sec.5}

\subsection{Burnett functions}\label{sec.6.1}
In this appendix, we will give some basic estimates used in the previous energy estimates.
To overcome some difficulties due to the term involving $L^{-1}_{M}$ and $\overline{G}$, we need to consider the integrality about the velocity.
To this end, we first list some properties of the Burnett functions and then give the fast decay about the velocity $\xi$ of the Burnett functions.
Recall the Burnett functions, cf. \cite{Bardos,Chapman-1990,Guo-2006,Ukai-Yang}:
\begin{equation}
\label{5.1}
\hat{A}_{j}(\xi)=\frac{|\xi|^{2}-5}{2}\xi_{j}\quad \mbox{and} \quad \hat{B}_{ij}(\xi)=\xi_{i}\xi_{j}-\frac{1}{3}\delta_{ij}|\xi|^{2} \quad \mbox{for} \quad i,j=1,2,3.
\end{equation}
Noting that $\hat{A}_{j}M$ and $\hat{B}_{ij}M$ are orthogonal to the null space $\mathcal{N}$ of $L_{M}$, we can define
functions $A_{j}(\frac{\xi-u}{\sqrt{R\theta}})$ and $ B_{ij}(\frac{\xi-u}{\sqrt{R\theta}})$ such that $P_{0}A_{j}=0$,
$P_{0}B_{ij}=0$ and
\begin{equation}
\label{5.2}
A_{j}(\frac{\xi-u}{\sqrt{R\theta}})=L^{-1}_{M}[\hat{A}_{j}(\frac{\xi-u}{\sqrt{R\theta}})M]\quad
\mbox{and} \quad B_{ij}(\frac{\xi-u}{\sqrt{R\theta}})=L^{-1}_{M}[\hat{B}_{ij}(\frac{\xi-u}{\sqrt{R\theta}})M].
\end{equation}
We shall list some elementary but important properties of the Burnett functions summarized in the following lemma, cf. \cite{Guo-2006,Bardos,Ukai-Yang}.

\begin{lemma}
The Burnett functions have the following properties:
\begin{itemize}
\item{$-\langle \hat{A}_{i}, A_{i}\rangle$ ~~is positive and independent of i;}
\item{$\langle \hat{A}_{i}, A_{j}\rangle=0$ ~~for ~any ~$i\neq j$;\quad $\langle
     \hat{A}_{i}, B_{jk}\rangle=0$~~for ~any ~i,~j,~k;}
\item{$\langle\hat{B}_{ij},B_{kj}\rangle=\langle\hat{B}_{kl},B_{ij}\rangle=\langle\hat{B}_{ji},B_{kj}\rangle$,~~
      which is independent of ~i,~j, for fixed~~k,~l;}
\item{$-\langle \hat{B}_{ij}, B_{ij}\rangle$ ~~is positive and independent of i,~j when $i\neq j$;}
\item{$-\langle \hat{B}_{ii}, B_{jj}\rangle$ ~~is positive and independent of i,~j when $i\neq j$;}
\item{$-\langle \hat{B}_{ii}, B_{ii}\rangle$ ~~is positive and independent of i;}
\item{$\langle \hat{B}_{ij}, B_{kl}\rangle=0$ ~~unless~either~$(i,j)=(k,l)$~or~$(l,k)$,~or~i=j~and~k=l;}
\item{$\langle \hat{B}_{ii}, B_{ii}\rangle-\langle \hat{B}_{ii}, B_{jj}\rangle=2\langle \hat{B}_{ij},
      B_{ij}\rangle$ ~~holds for any~ $i\neq j$.}
\end{itemize}
\end{lemma}
In terms of Burnett functions, the viscosity coefficient $\mu(\theta)$ and heat conductivity
coefficient $\kappa(\theta)$ in \eqref{1.15} can be represented by
\begin{align}
\label{5.3}
\mu(\theta)=&- R\theta\int_{\mathbb{R}^{3}}\hat{B}_{ij}(\frac{\xi-u}{\sqrt{R\theta}})
B_{ij}(\frac{\xi-u}{\sqrt{R\theta}})\,d\xi>0,\quad i\neq j,
\nonumber\\
\kappa(\theta)=&-R^{2}\theta\int_{\mathbb{R}^{3}}\hat{A}_{j}(\frac{\xi-u}{\sqrt{R\theta}})
A_{j}(\frac{\xi-u}{\sqrt{R\theta}})\,d\xi>0.
\end{align}
Notice that these coefficients are positive smooth functions depending only on $\theta$.
\par
The following lemma is borrowed from \cite[Lemma 6.1]{Duan-Yu1}, which is about
the fast velocity decay of the Burnett functions.
\begin{lemma}\label{lem5.2}
Suppose that $U(\xi)$ is any polynomial of $\frac{\xi-\hat{u}}{\sqrt{R}\hat{\theta}}$ such that
$U(\xi)\widehat{M}\in(\ker{L_{\widehat{M}}})^{\perp}$ for any Maxwellian $\widehat{M}=M_{[1/\widehat{v},\widehat{u},\widehat{\theta}]}(\xi)$
 as \eqref{1.36} where $L_{\widehat{M}}$ is as in \eqref{1.21} .
For any $\varepsilon\in(0,1)$ and any multi-index $\beta$, there exists constant $C_{\beta}>0$ such that
$$
|\partial_{\beta}L^{-1}_{\widehat{M}}(U(\xi)\widehat{M})|\leq C_{\beta}(\widehat{v},\widehat{u},\widehat{\theta})\widehat{M}^{1-\varepsilon}.
$$
In particular, under the assumptions of \eqref{2.12}, there exists constant $C_{\beta}>0$ such that
\begin{equation}
\label{5.4}
|\partial_{\beta}A_{j}(\frac{\xi-u}{\sqrt{R\theta}})|+|\partial_{\beta}B_{ij}(\frac{\xi-u}{\sqrt{R\theta}})|
\leq C_{\beta}M^{1-\varepsilon}.
\end{equation}
\end{lemma}

\subsection{Estimates on terms of $\mathcal{L}$ and $\Gamma$}
Now, we shall turn to recall the refined estimates for the linearized operator $\mathcal{L}$
and  the nonlinear collision terms $\Gamma(g_1,g_2)$ defined in \eqref{2.5}. They can be proved by a straightforward modification of the arguments used
in \cite[Lemmas 9]{Strain-Guo-2008} and \cite[Lemmas 2.2-2.3]{Wang} and we thus omit their proofs for brevity.
\begin{lemma}\label{lem5.3}
Assume $0\leq q(t)\ll 1$ in $w=w(\beta)$ defined by \eqref{1.32}.
For any $\epsilon>0$ small enough, there exists $C_\epsilon>0$ such that
\begin{equation}
\label{5.5}
-\langle\partial^\alpha_\beta\mathcal{L}g,w^2(\beta)\partial^\alpha_\beta g\rangle\geq |w(\beta)\partial^\alpha_\beta g|_\sigma^2-\epsilon\sum_{|\beta_1|=|\beta|}|w(\beta_1)\partial^\alpha_{\beta_1} g|_\sigma^2
-C_\epsilon\sum_{|\beta_1|<|\beta|}|w(\beta_1)\partial^\alpha_{\beta_1} g|_\sigma^2.
\end{equation}
If $|\beta| = 0$, there exists $c_{4}>0$ such that
\begin{equation}
\label{5.6}
-\langle\partial^\alpha\mathcal{L}g,w^2(0)\partial^\alpha g\rangle\geq c_{4}|w(0)\partial^\alpha g|_\sigma^2-C_\epsilon|\chi_{\epsilon}(\xi)\partial^\alpha g|_2^2,
\end{equation}
where $\chi_\epsilon(\xi)$ is a general cutoff function depending on $\epsilon$.
\end{lemma}

\begin{lemma}
Under the assumptions of Lemma \ref{lem5.3}. For any $\varepsilon>0$ small enough, one has
\begin{equation}
\label{5.7}
\langle\partial^\alpha \Gamma(g_1,g_2),    g_3\rangle\leq C\sum_{|\alpha_1|\leq|\alpha|}|\mu^\varepsilon\partial^{\alpha_1}g_1|_2| \partial^{\alpha-\alpha_1}g_2|_\sigma|  g_3|_\sigma,
\end{equation}
and
\begin{equation}
\label{5.8}
\langle\partial^\alpha_\beta \Gamma(g_1,g_2), w^2(\beta)   g_3\rangle\leq
C\sum_{|\alpha_1|\leq|\alpha|}\sum_{|\bar{\beta}|\leq|\beta_1|\leq|\beta|}|\mu^\varepsilon\partial^{\alpha_1}_{\bar{\beta}}g_1|_2|w(\beta)  \partial^{\alpha-\alpha_1}_{\beta-\beta_1}g_2|_{\sigma}|w(\beta)    g_3|_{\sigma}.
\end{equation}
\end{lemma}

Next we prove some linear and nonlinear estimates, which are used in Sections \ref{sec.2} and \ref{sec.3}.
We first consider the estimates of the terms $\Gamma(\mathbf{g},\frac{M-\mu}{\sqrt{\mu}})$ and $\Gamma(\frac{M-\mu}{\sqrt{\mu}},\mathbf{g})$.
\begin{lemma}\label{lem5.7}
Let $|\alpha|+|\beta|\leq 2$ and  $0\leq q(t)\ll 1$ in $w=w(\beta)$ defined by \eqref{1.32}.
Suppose that \eqref{2.9}, \eqref{2.9a}  and \eqref{2.12} hold.
If we choose  $\eta_0>0$ in \eqref{2.12}, $\varepsilon_{0}>0$ in \eqref{2.9} and $\delta>0$ in \eqref{1.27} small enough, one has
\begin{eqnarray}
\label{5.29}
&&|(\partial^\alpha_\beta [v\Gamma(\frac{M-\mu}{\sqrt{\mu}},\mathbf{g})], w^2(\beta)  h)|
+|(\partial^\alpha_\beta [v\Gamma(\mathbf{g},\frac{M-\mu}{\sqrt{\mu}})],w^2(\beta)  h)|
\notag\\
&&\leq C(\eta_0+\delta+\varepsilon_{0})\|w(\beta)h\|_{\sigma}^2+C(\eta_0+\delta+\varepsilon_{0})\mathcal{D}_{2,l,q}(t),
\end{eqnarray}
and
\begin{eqnarray}
\label{5.30}
&&|(\partial^\alpha[v\Gamma(\frac{M-\mu}{\sqrt{\mu}},\mathbf{g})],h)|
+|(\partial^\alpha[v \Gamma(\mathbf{g},\frac{M-\mu}{\sqrt{\mu}})],h)|
\notag\\
&&\leq C(\eta_0+\delta+\varepsilon_{0})\|h\|_{\sigma}^2+C(\eta_0+\delta+\varepsilon_{0})\mathcal{D}_{2,l,q}(t).
\end{eqnarray}
\end{lemma}
\begin{proof}
We only consider the first term on the left hand side of \eqref{5.29} while the second term can be handled in the same way.
Notice that
\begin{align*}
\partial^\alpha_\beta [v\Gamma(\frac{M-\mu}{\sqrt{\mu}},\mathbf{g})]=\sum_{|\alpha_1|\leq|\alpha|}
C^{\alpha_{1}}_{\alpha}\partial^{\alpha-\alpha_{1}}v\partial^{\alpha_{1}}_{\beta}\Gamma(\frac{M-\mu}{\sqrt{\mu}},\mathbf{g}).
\end{align*}
It follows from this and \eqref{5.8} that
\begin{align}
\label{5.31}
&|(\partial^\alpha_\beta [v\Gamma(\frac{M-\mu}{\sqrt{\mu}},\mathbf{g})], w^2(\beta)  h)|
\nonumber\\
&\leq C\sum_{|\alpha_2|\leq|\alpha_1|\leq|\alpha|}\sum_{ |\bar{\beta}|\leq|\beta_1|\leq|\beta|}
\underbrace{\int_{\mathbb R}|\partial^{\alpha-\alpha_{1}}v||\mu^\varepsilon\partial^{\alpha_2}_{\bar{\beta}}(\frac{M-\mu}{\sqrt{\mu}})|_2| w(\beta) \partial^{\alpha_{1}-\alpha_2}_{\beta-\beta_1}\mathbf{g}|_{\sigma}|  w(\beta)h|_{\sigma}\,dx}_{I_{1}}.
\end{align}
For any $\beta'\geq0$ and any $b>0$, from \eqref{1.32}, \eqref{1.35}, \eqref{2.9a} and \eqref{2.12}, there exists a small $\varepsilon_{1}>0$ such that
$$
| \langle \xi\rangle^{b}\partial _{\beta'}(\frac{M-\mu}{\sqrt{\mu}})|_{\sigma,w}^2+| \langle \xi\rangle^{b}\partial _{\beta'}(\frac{M-\mu}{\sqrt{\mu}})|_{2,w}^2\leq
C_b\sum_{|\beta'|\leq|\beta''|\leq|\beta'|+1}\int_{{\mathbb R}^3}\mu^{-\varepsilon_1}
|\partial _{\beta''}(\frac{M-\mu}{\sqrt{\mu}})|^2\,d\xi.
$$
For $\eta_{0}>0$ in \eqref{2.12}, there exists some large constant  $R>0$ such that
$$
\int_{|\xi|\geq R}\mu^{-\varepsilon_1}|\partial _{\beta''}(\frac{M-\mu}{\sqrt{\mu}})|^2 \,d\xi\leq C(\eta_0+\varepsilon_{0})^2,
$$
and
$$
\int_{|\xi|\leq R}\mu^{-\varepsilon_1}|\partial _{\beta''}(\frac{M-\mu}{\sqrt{\mu}})|^2 \,d\xi\leq C(|v-1|+|u|+|\theta-\frac{3}{2}|)^2\leq C(\eta_0+\varepsilon_{0})^2.
$$
Thus, for any $\beta'\geq0$ and $b>0$, we deduce from the above  estimates that
\begin{equation}
\label{5.32}
| \langle \xi\rangle^{b}\partial _{\beta'}(\frac{M-\mu}{\sqrt{\mu}})|_{\sigma,w}^2+| \langle \xi\rangle^{b}\partial _{\beta'}(\frac{M-\mu}{\sqrt{\mu}})|_{2,w}^2
 \leq C(\eta_0+\varepsilon_{0})^2.
\end{equation}
Note that $|\alpha_2|\leq|\alpha_1|\leq|\alpha|\leq 2$ in \eqref{5.31} since we  consider $|\alpha|+|\beta|\leq 2$.
If $|\alpha_2|=0$ and $|\alpha-\alpha_{1}|\leq\frac{|\alpha|}{2}$,
we have from  \eqref{5.32} and \eqref{2.11} that
\begin{align*}
I_{1}=&\int_{\mathbb R}|\partial^{\alpha-\alpha_{1}}v||\mu^\varepsilon\partial^{\alpha_2}_{\bar{\beta}}(\frac{M-\mu}{\sqrt{\mu}})|_2| w(\beta) \partial^{\alpha_{1}-\alpha_2}_{\beta-\beta_1}\mathbf{g}|_{\sigma}|  w(\beta)h|_{\sigma}\,dx
\\
&\leq C(\eta_0+\varepsilon_{0})(\|\partial^{\alpha-\alpha_{1}}\widetilde{v}\|_{L_{x}^{\infty}}
+\|\partial^{\alpha-\alpha_{1}}\bar{v}\|_{L_{x}^{\infty}})\|w(\beta)\partial^{\alpha_{1}-\alpha_2}_{\beta-\beta_1}\mathbf{g}\|_{\sigma}\| w(\beta)h\|_{\sigma}
\\
&\leq C(\eta_0+\varepsilon_{0})(\|w(\beta) h\|_{\sigma }^2+\|\partial^{\alpha_{1}-\alpha_2}_{\beta-\beta_1}\mathbf{g}\|^{2}_{\sigma,w})
\leq C(\eta_0+\varepsilon_{0})(\|w(\beta) h\|_{\sigma }^2+\mathcal{D}_{2,l,q}(t)),
\end{align*}
where   we have used  the facts that $w(\beta)\leq w(\beta-\beta_1)$
and
$$
\|\partial^{\alpha-\alpha_{1}}\widetilde{v}\|_{L_{x}^{\infty}}
+\|\partial^{\alpha-\alpha_{1}}\bar{v}\|_{L_{x}^{\infty}}\leq C,
$$
due to the imbedding inequality, \eqref{1.29} and \eqref{2.9}.
If $|\alpha_2|=0$ and $|\alpha-\alpha_{1}|>\frac{|\alpha|}{2}$, we have
\begin{align*}
I_{1}&\leq C(\eta_0+\varepsilon_{0})\int_{\mathbb R}|\partial^{\alpha-\alpha_{1}}v|| w(\beta) \partial^{\alpha_{1}-\alpha_2}_{\beta-\beta_1}\mathbf{g}|_{\sigma}|  w(\beta)h|_{\sigma}\,dx
\\
&\leq C(\eta_0+\varepsilon_{0})\|\partial^{\alpha-\alpha_{1}}v\|\Big\||w(\beta)\partial^{\alpha_{1}-\alpha_2}_{\beta-\beta_1}\mathbf{g}|_{\sigma}\Big\|_{L_{x}^{\infty}}\| w(\beta)h\|_{\sigma}
\\
&\leq C(\eta_0+\varepsilon_{0})\|w(\beta)\partial^{\alpha_{1}-\alpha_2}_{\beta-\beta_1}\mathbf{g}\|^{\frac{1}{2}}_{\sigma}
\|w(\beta)\partial^{\alpha_{1}-\alpha_2}_{\beta-\beta_1}\mathbf{g}_{x}\|^{\frac{1}{2}}_{\sigma}\| w(\beta)h\|_{\sigma}
\\
&\leq (\eta_0+\varepsilon_{0})\|w(\beta) h\|_{\sigma }^2+ C(\eta_0+\varepsilon_{0})\mathcal{D}_{2,l,q}(t).
\end{align*}
If $|\alpha_2|=1$, then $|\alpha-\alpha_{1}|\leq1$, we have from the imbedding inequality, \eqref{1.29}, \eqref{2.10} and \eqref{2.9}that
\begin{align*}
I_{1}&\leq C\|\partial^{\alpha-\alpha_{1}}v\|_{L_{x}^{\infty}}\|\partial^{\alpha_2}[v,u,\theta]\|_{L_{x}^\infty}
\| w(\beta)\partial^{\alpha_{1}-\alpha_2}_{\beta-\beta_1}\mathbf{g}\|_{\sigma}\|  w(\beta)  h\|_{\sigma}
\\
&\leq C\big(\delta+\sqrt{\mathcal{E}_{2,l,q}(t)}\big)\| w(\beta)\partial^{\alpha_{1}-\alpha_2}_{\beta-\beta_1}\mathbf{g}\|_{\sigma}\|w(\beta)h\|_{\sigma}
\\
&\leq C(\delta+\varepsilon_{0})\|  w(\beta)  h\|^{2}_{\sigma}+C(\delta+\varepsilon_{0})\mathcal{D}_{2,l,q}(t).
\end{align*}
If $|\alpha_2|=2$, then $|\alpha_1|=|\alpha|=2$, we can obtain
\begin{align*}
I_{1}&\leq C(\|\partial^{\alpha_2}[v,u,\theta]\|+\sum_{|\alpha'|=1}\|\partial^{\alpha'}[v,u,\theta]\|^{2})
\Big\||w(\beta)\partial^{\alpha-\alpha_1}_{\beta_{1}-\beta_2}\mathbf{g}|_{\sigma}\Big\|_{L_{x}^{\infty}}\|  w(\beta)  h\|_{\sigma}
\\
&\leq C(\delta+\varepsilon_{0})\|  w(\beta)  h\|^{2}_{\sigma}+C(\delta+\varepsilon_{0})\mathcal{D}_{2,l,q}(t).
\end{align*}
Hence, for $\eta_0>0$, $\delta>0$ and $\varepsilon_{0}>0$ small enough, we deduce from the above   estimates that
$$
|(\partial^\alpha_\beta[v \Gamma(\frac{M-\mu}{\sqrt{\mu}},\mathbf{g})],w^2(\beta)   h)|
\leq C(\eta_0+\delta+\varepsilon_{0})\big(\|w(\beta)h\|_{\sigma}^2+\mathcal{D}_{2,l,q}(t)\big).
$$
Similar arguments as the above give
\begin{equation*}
|(\partial^\alpha_\beta [v\Gamma(\mathbf{g},\frac{M-\mu}{\sqrt{\mu}})], w^2(\beta)  h)|
\leq C(\eta_0+\delta+\varepsilon_{0})\big(\|w(\beta)h\|_{\sigma}^2+\mathcal{D}_{2,l,q}(t)\big).
\end{equation*}
Estimate \eqref{5.29} thus follows from the above two estimates.  By \eqref{5.7} and  the similar calculations
as \eqref{5.29},  we can prove that \eqref{5.30} holds and we omit the details for brevity.
This completes the proof of Lemma \ref{lem5.7}.
\end{proof}
The following estimates are concerned with the nonlinear term $\Gamma(\frac{G}{\sqrt{\mu}},\frac{G}{\sqrt{\mu}})$.

\begin{lemma}\label{lem5.8}
Let $|\alpha|+|\beta|\leq 2$ and $0\leq q(t)\ll 1$ in $w=w(\beta)$ defined by \eqref{1.32}.
Suppose that \eqref{2.9}, \eqref{2.9a}  and \eqref{2.12} hold.
If we choose $\varepsilon_{0}>0$ in \eqref{2.9} and $\delta>0$ in \eqref{1.27} small enough, one has
\begin{equation}
\label{5.33}
|(\partial^\alpha_\beta[v\Gamma(\frac{G}{\sqrt{\mu}},\frac{G}{\sqrt{\mu}})], w^2(\beta) h)|
\leq C(\delta+\varepsilon_{0})\big(\|w(\beta)h\|_{\sigma}^2+\mathcal{D}_{2,l,q}(t)\big)+C\delta(1+t)^{-\frac{4}{3}},
\end{equation}
and
\begin{equation}
\label{5.34}
|(\partial^\alpha [v\Gamma(\frac{G}{\sqrt{\mu}},\frac{G}{\sqrt{\mu}})],h)|
\leq C(\delta+\varepsilon_{0})\big(\|h\|_{\sigma}^2+\mathcal{D}_{2,l,q}(t)\big)+C\delta(1+t)^{-\frac{4}{3}}.
\end{equation}
\end{lemma}
\begin{proof}
Recalling that  $G=\overline{G}+\sqrt{\mu}\mathbf{g}$, we   see
\begin{equation}
\label{5.35}
\Gamma(\frac{G}{\sqrt{\mu}},\frac{G}{\sqrt{\mu}})=\Gamma(\frac{\overline{G}}{\sqrt{\mu}},\frac{\overline{G}}{\sqrt{\mu}})
 +\Gamma(\frac{\overline{G}}{\sqrt{\mu}},\mathbf{g})+\Gamma(\mathbf{g},\frac{\overline{G}}{\sqrt{\mu}})+\Gamma(\mathbf{g},\mathbf{g}).
\end{equation}
For the first term in \eqref{5.35}, we have from the similar arguments as \eqref{5.31} that
\begin{multline}
\label{5.36}
|(\partial^\alpha_\beta [v\Gamma(\frac{\overline{G}}{\sqrt{\mu}},\frac{\overline{G}}{\sqrt{\mu}})], w^2(\beta)h)|\\
\leq C\sum_{|\alpha_2|\leq|\alpha_1|\leq|\alpha|}\sum_{ |\bar{\beta}|\leq|\beta_1|\leq|\beta|}
\underbrace{\int_{\mathbb R}|\partial^{\alpha-\alpha_{1}}v||\mu^\varepsilon\partial^{\alpha_2}_{\bar{\beta}}(\frac{\overline{G}}{\sqrt{\mu}})|_2| w(\beta) \partial^{\alpha_{1}-\alpha_2}_{\beta-\beta_1}(\frac{\overline{G}}{\sqrt{\mu}})|_{\sigma}|  w(\beta)h|_{\sigma}\,dx}_{I_{2}}.
\end{multline}
By \eqref{5.1} and \eqref{5.2}, we can rewrite $\overline{G}$ in \eqref{2.3} as
\begin{equation}
\label{5.19a}
\overline{G}(t,x,\xi)=\frac{1}{v}\frac{\sqrt{R}\;\overline{\theta}_x}{\sqrt{\theta}}A_1(\frac{\xi-u}{\sqrt{R\theta}})
+\frac{1}{v}\overline{u}_{1x}B_{11}(\frac{\xi-u}{\sqrt{R\theta}}),
\end{equation}
which implies that for  $\beta_1=(1,0,0)$,
\begin{equation*}
\partial_{\beta_1}\overline{G}=\frac{1}{v}\frac{\sqrt{R}\;\overline{\theta}_x}{\sqrt{\theta}}\partial_{\xi_1}A_1(\frac{\xi-u}{\sqrt{R\theta}})(\frac{1}{\sqrt{R\theta}})
+\frac{1}{v}\overline{u}_{1x}\partial_{\xi_1}B_{11}(\frac{\xi-u}{\sqrt{R\theta}})(\frac{1}{\sqrt{R\theta}}).
\end{equation*}
Similarly, we also have
\begin{align}
\label{5.20}
\overline{G}_x=&-\frac{v_{x}\,\sqrt{R}\;\overline{\theta}_x }{v^{2}\sqrt{\theta}}A_1(\frac{\xi-u}{\sqrt{R\theta}})
-\frac{v_{x}}{v^{2}}\overline{u}_{1x}B_{11}(\frac{\xi-u}{\sqrt{R\theta}})
\notag\\
&+\frac{1}{v}\frac{\sqrt{R}\;\overline{\theta}_{xx}}{\sqrt{\theta}}A_1(\frac{\xi-u}{\sqrt{R\theta}})
-\frac{1}{v}\frac{\sqrt{R}\;\overline{\theta}_{x}{\theta}_{x}}{2\sqrt{\theta^3}}A_1(\frac{\xi-u}{\sqrt{R\theta}})
\notag\\
&-\frac{1}{v}\frac{\sqrt{R}\;\overline{\theta}_{x}}{\sqrt{\theta}}\nabla_\xi A_1(\frac{\xi-u}{\sqrt{R\theta}})\cdot
\frac{u_x}{\sqrt{R\theta}} -\frac{1}{v}\frac{\sqrt{R}\;
\overline{\theta}_{x}{\theta}_{x}}{\sqrt{\theta}}\nabla_\xi A_1(\frac{\xi-u}{\sqrt{R\theta}})\cdot\frac{\xi-u}{2\sqrt{R\theta^3}}
\notag\\
&+\frac{1}{v}\overline{u}_{1xx}B_{11}(\frac{\xi-u}{\sqrt{R\theta}})
-\frac{1}{v}\frac{\overline{u}_{1x}u_x}{\sqrt{R\theta}}\cdot \nabla_\xi B_{11}(\frac{\xi-u}{\sqrt{R\theta}})
-\frac{1}{v}\frac{\overline{u}_{1x}\theta_x(\xi-u)}{2\sqrt{R\theta^3}}\cdot \nabla_\xi B_{11}(\frac{\xi-u}{\sqrt{R\theta}}).
\end{align}
And $\overline{G}_t$ has the similar expression as \eqref{5.20}.
For any $|\bar{\alpha}|\geq 1$ and $|\bar{\beta}|\geq 0$, we use the similar expansion as the above to get
\begin{equation}
\label{5.37}
| \langle \xi\rangle^{b}\partial _{\bar{\beta}}(\frac{\overline{G}}{\sqrt{\mu}})|_{2,w}+|\langle \xi\rangle^{b} \partial _{\bar{\beta}}(\frac{\overline{G}}{\sqrt{\mu}})|_{\sigma,w}
\leq C|[\overline{u}_x,\overline{\theta}_x]|,
\end{equation}
and
\begin{equation}
\label{5.38}
|\langle \xi\rangle^{b} \partial^{\bar{\alpha}}_{\bar{\beta}}(\frac{\overline{G}}{\sqrt{\mu}})|_{2,w}+| \langle \xi\rangle^{b} \partial^{\bar{\alpha}}_{\bar{\beta}}(\frac{\overline{G}}{\sqrt{\mu}})|_{\sigma,w}
\leq C(|\partial^{\bar{\alpha}}[\overline{u}_x,\overline{\theta}_x]|+...
+|[\overline{u}_x,\overline{\theta}_x]||\partial^{\bar{\alpha}}[v,u,\theta]|).
\end{equation}
Here we have used Lemma \ref{lem5.2} and the fact that $|\langle \xi\rangle^b w(\bar{\beta})\mu^{-\frac{1}{2}}M^{1-\varepsilon}|_2\leq C$
for any $b\ge0$ and any small $\varepsilon>0$ by \eqref{2.9a} and \eqref{2.12}.  

Note that $|\alpha_2|\leq|\alpha_1|\leq|\alpha|\leq 2$ in \eqref{5.36} due to the fact that $|\alpha|+|\beta|\leq 2$.
If $|\alpha-\alpha_{1}|\leq 1$, by using  \eqref{5.37}, \eqref{5.38}, \eqref{1.29}, \eqref{2.9}  and
the imbedding inequality, one has from \eqref{5.36} that
\begin{align*}
I_{2}&\leq C\|\partial^{\alpha-\alpha_{1}}v\|_{L_{x}^{\infty}}\int_{\mathbb R}
\big\{|\partial^{\alpha_{2}}[\overline{u}_x,\overline{\theta}_x]|+...
+|[\overline{u}_x,\overline{\theta}_x]||\partial^{\alpha_{2}}[v,u,\theta]|\big\}
\\
&\qquad\times\big\{|\partial^{\alpha_{1}-\alpha_{2}}[\overline{u}_x,\overline{\theta}_x]|+...
+|[\overline{u}_x,\overline{\theta}_x]||\partial^{\alpha_{1}-\alpha_{2}}[v,u,\theta]|\big\}
|w(\beta) h|_{\sigma}\,dx
\\
&\leq C(\delta+\varepsilon_{0})\big(\|w(\beta)h\|_{\sigma}^2+\mathcal{D}_{2,l,q}(t)\big)+C\delta(1+t)^{-\frac{4}{3}}.
\end{align*}
If $|\alpha-\alpha_{1}|=2$, then $|\alpha|=2$ and $|\alpha_{1}|=|\alpha_{2}|=0$, we have
\begin{align*}
I_{2}&\leq C\|[\overline{u}_x,\overline{\theta}_x]\|^{2}_{L_{x}^{\infty}}\int_{\mathbb R}|\partial^{\alpha-\alpha_{1}}v||w(\beta) h|_{\sigma}\,dx
\\
&\leq C\delta\|w(\beta) h\|_{\sigma}^2
+C\delta\mathcal{D}_{2,l,q}(t)+C\delta(1+t)^{-\frac{4}{3}}.
\end{align*}
It follows from the above two estimates and \eqref{5.36} that
\begin{equation}
\label{5.39}
|(\partial^\alpha_\beta [v\Gamma(\frac{\overline{G}}{\sqrt{\mu}},\frac{\overline{G}}{\sqrt{\mu}})], w^2(\beta)h)|
\leq C(\delta+\varepsilon_{0})\big(\|w(\beta)h\|_{\sigma}^2+\mathcal{D}_{2,l,q}(t)\big)+C\delta(1+t)^{-\frac{4}{3}}.
\end{equation}
For the second term in \eqref{5.35}, by \eqref{5.8}, we   can obtain
\begin{align}
\label{5.40c}
&|(\partial^\alpha_\beta [v\Gamma(\frac{\overline{G}}{\sqrt{\mu}},\mathbf{g})], w^2(\beta)h)|
\nonumber\\
&\leq C\sum_{|\alpha_2|\leq|\alpha_1|\leq|\alpha|}\sum_{ |\bar{\beta}|\leq|\beta_1|\leq|\beta|}
\underbrace{\int_{\mathbb R}|\partial^{\alpha-\alpha_{1}}v||\mu^\varepsilon\partial^{\alpha_2}_{\bar{\beta}}(\frac{\overline{G}}{\sqrt{\mu}})|_2| w(\beta) \partial^{\alpha_{1}-\alpha_2}_{\beta-\beta_1}\mathbf{g}|_{\sigma}|  w(\beta)h|_{\sigma}\,dx}_{I_{3}}.
\end{align}
Notice that $|\alpha_2|\leq|\alpha_1|\leq|\alpha|\leq 2$ in \eqref{5.40c}.
If $|\alpha-\alpha_{1}|\leq1$   and  $|\alpha_{2}|\leq1$,
we can deduce from  \eqref{5.37}, \eqref{5.38}, \eqref{1.29}, \eqref{2.9} and the imbedding inequality that
\begin{align*}
I_{3}&\leq C\Big\||\partial^{\alpha-\alpha_{1}}v|\big\{|\partial^{\alpha_{2}}[\overline{u}_x,\overline{\theta}_x]|
+|[\overline{u}_x,\overline{\theta}_x]||\partial^{\alpha_{2}}[v,u,\theta]|\big\}\Big\|_{L^{\infty}_{x}}
\int_{\mathbb R}|w(\beta)\partial^{\alpha_{1}-\alpha_2}_{\beta-\beta_1}\mathbf{g}|_{\sigma} |w(\beta)h|_{\sigma}\,dx
\\
&\leq C(\delta+\varepsilon_{0})\|w(\beta)h\|_{\sigma}^2+C(\delta+\varepsilon_{0})\mathcal{D}_{2,l,q}(t),
\end{align*}
where    we used the fact that $w(\beta)\leq w(\beta-\beta_1)$ due to \eqref{1.32}.

If $|\alpha-\alpha_{1}|\leq1$   and  $|\alpha_{2}|=2$, then $|\alpha|=|\alpha_{1}|=|\alpha_{2}|=2$ and we have
\begin{align*}
I_{3}&\leq C\Big\||w(\beta)\partial^{\alpha_{1}-\alpha_2}_{\beta-\beta_1}\mathbf{g}|_{\sigma}\Big\|_{L^{\infty}_{x}}
\int_{\mathbb R}\big\{|\partial^{\alpha_{2}}[\overline{u}_x,\overline{\theta}_x]|+\cdot\cdot\cdot
+|[\overline{u}_x,\overline{\theta}_x]||\partial^{\alpha_{2}}[v,u,\theta]|\big\} |w(\beta)h|_{\sigma}\,dx
\\
&\leq C(\delta+\varepsilon_{0})\|w(\beta)h\|_{\sigma}^2+C(\delta+\varepsilon_{0})\mathcal{D}_{2,l,q}(t).
\end{align*}
\\
If $|\alpha-\alpha_{1}|=2$ , then $|\alpha|=2$ and $|\alpha_{1}|=|\alpha_{2}|=0$, it follows that
\begin{align*}
I_{3}&\leq C\Big\||[\overline{u}_x,\overline{\theta}_x]||w(\beta) \partial^{\alpha_{1}-\alpha_2}_{\beta-\beta_1}\mathbf{g}|_{\sigma}\Big\|_{L^{\infty}_{x}}
\int_{\mathbb R}|\partial^{\alpha-\alpha_{1}}v||w(\beta) h|_{\sigma}\,dx
\\
&\leq C(\delta+\varepsilon_{0})\|w(\beta)h\|_{\sigma}^2+ C(\delta+\varepsilon_{0})\mathcal{D}_{2,l,q}(t).
\end{align*}
Owing to these, we can derive that
\begin{equation}
\label{5.41c}
|(\partial^\alpha_\beta [v\Gamma(\frac{\overline{G}}{\sqrt{\mu}},\mathbf{g})], w^2(\beta)h)|
\leq C(\delta+\varepsilon_{0})\big(\|w(\beta)h\|_{\sigma}^2+\mathcal{D}_{2,l,q}(t)\big).
\end{equation}
Similar arguments as \eqref{5.41c} imply
\begin{equation*}
|(\partial^\alpha_\beta [v\Gamma(\mathbf{g},\frac{\overline{G}}{\sqrt{\mu}})], w^2(\beta)h)|
\leq C(\delta+\varepsilon_{0})\big(\|w(\beta)h\|_{\sigma}^2+\mathcal{D}_{2,l,q}(t)\big).
\end{equation*}
By \eqref{5.8} and the similar calculations as \eqref{5.41c}, we can arrive at
\begin{align}
\label{5.42}
&|(\partial^\alpha_\beta [v\Gamma(\mathbf{g},\mathbf{g})],w^2(\beta)h)|
\nonumber\\
&\leq C\sum_{|\alpha_2|\leq|\alpha_1|\leq|\alpha|}\sum_{ |\bar{\beta}|\leq|\beta_1|\leq|\beta|}
\int_{\mathbb R}|\partial^{\alpha-\alpha_{1}}v||\mu^\varepsilon\partial^{\alpha_2}_{\bar{\beta}}\mathbf{g}|_2| w(\beta) \partial^{\alpha_{1}-\alpha_2}_{\beta-\beta_1}\mathbf{g}|_{\sigma}|  w(\beta)h|_{\sigma}\,dx
\nonumber\\
&\leq C(\delta+\varepsilon_{0})\|w(\beta)h\|_{\sigma}^2+C(\delta+\varepsilon_{0})\mathcal{D}_{2,l,q}(t).
\end{align}
By the estimates from \eqref{5.39} to \eqref{5.42}, one gets \eqref{5.33}.
We can follow the similar calculations as \eqref{5.39}-\eqref{5.42} to get \eqref{5.34}.
Therefore, the proof of Lemma \ref{lem5.8} is completed.
\end{proof}

\subsection{A technical lemma for weighted macro estimates}
Finally, we  will deduce  a crucial estimate to control the last term in \eqref{2.23} by using the system \eqref{2.7} and the properties of the viscous  contact wave profiles.
We first give the following  lemma, which can be found in \cite[Lemma 1]{Huang-Li-Matsumura}.
\begin{lemma}\label{lem5.5}
For $0<T\leq+\infty$, suppose that $h(t,x)$ satisfies
$$
h_{x}\in L^{2}(0,T;L^{2}(\mathbb{R})), \quad h_{t}\in L^{2}(0,T;H^{-1}(\mathbb{R})).
$$
Then the following estimate holds
\begin{align*}
\int^{T}_{0}\int_{\mathbb{R}} h^{2}\omega^{2}\,dxdt\leq 4\pi\|h(0)\|^{2}+4\pi\lambda^{-1}\int^{T}_{0}\|h_{x}(t)\|^{2}\,dt
+8\lambda\int^{T}_{0}( h_{t},h\mathfrak{g}^{2})\,dt,
\end{align*}
for some $\lambda>0$, where
\begin{align}
\label{5.10}
\omega(t,x)=(1+t)^{-\frac{1}{2}}\exp\big(-\frac{\lambda x^{2}}{1+t}\big), \quad \mathfrak{g}(t,x)=\int^{x}_{-\infty}\omega(t,y)\,dy.
\end{align}
\end{lemma}
 The following lemma  is used to deal with the last term in \eqref{2.23}.
\begin{lemma}\label{lem5.6}
For  $\lambda\in(0,c_{1}/4]$ with $c_{1}$ in \eqref{1.27} and $\omega$ defined in \eqref{5.10}, if \eqref{2.9}  holds,  there exists $C_{2}>0$ 
 such that the following estimate holds
\begin{align}
\label{5.11}
\int^{t}_{0}\int_{\mathbb{R}}&(\widetilde{v}^{2}+\widetilde{u}^{2}+\widetilde{\theta}^{2})\omega^{2}\,dxds
\nonumber\\
&\leq C_{2}+C_{2}\varepsilon_{0} \int^{t}_{0} \|\mathbf{g}\|^{2}_{\sigma}\, ds
+C_{2}\sum_{|\alpha|=1}\int^{t}_{0}\big(\|\partial^{\alpha}[\widetilde{v},\widetilde{u},\widetilde{\theta}]\|^{2}
+\|\partial^{\alpha}\mathbf{g}\|^{2}_{\sigma} \big)\,ds.
\end{align}
\end{lemma}
\begin{proof}
As in \cite{Huang-Li-Matsumura}, we define
\begin{align}
\label{5.12}
\mathfrak{f}(t,x)=\int_{-\infty}^{x}\omega^{2}(t,y)\,dy.
\end{align}
It is easy to check that
\begin{align}
\label{5.12a}
\|\mathfrak{f}(t,x)\|_{L_{x}^{\infty}}\leq 2\lambda^{-\frac{1}{2}}(1+t)^{-\frac{1}{2}}, \quad
\|\mathfrak{f}_{t}(t,x)\|_{L_{x}^{\infty}}\leq 4\lambda^{-\frac{1}{2}}(1+t)^{-\frac{3}{2}}.
\end{align}
Taking the inner product of \eqref{2.7}$_{2}$ with $(\frac{2}{3}\widetilde{\theta}-p_{+}\widetilde{v})v\mathfrak{f}$ with respect to $x$ over $\mathbb{R}$ and
using the fact that $p-p_{+}=\frac{2\widetilde{\theta}-3p_{+}\widetilde{v}}{3v}$, the integration by parts and \eqref{5.12}, we have
\begin{align}
\label{5.13}
\frac{1}{2}\big((\frac{2}{3}\widetilde{\theta}-p_{+}\widetilde{v})^{2},\omega^{2}\big)
&=\big(\widetilde{u}_{1t},(\frac{2}{3}\widetilde{\theta}-p_{+}\widetilde{v})v\mathfrak{f}\big)-
\big(\frac{1}{v}(\frac{2}{3}\widetilde{\theta}-p_{+}\widetilde{v})^{2},v_{x}\mathfrak{f}\big)
\nonumber\\
&\qquad+\big(\frac{4}{3}\frac{\mu(\theta)}{v}u_{1x},
[(\frac{2}{3}\widetilde{\theta}-p_{+}\widetilde{v})v\mathfrak{f}]_{x}\big)
+\big(\bar{u}_{1t},(\frac{2}{3}\widetilde{\theta}-p_{+}\widetilde{v})v\mathfrak{f}\big)
\nonumber\\
&\qquad-\big(\int_{\mathbb{R}^{3}} \xi^{2}_{1}L^{-1}_{M}\Theta_{1}\, d\xi,[(\frac{2}{3}\widetilde{\theta}-p_{+}\widetilde{v})v\mathfrak{f}]_{x}\big)
:=\sum_{i=4}^{8}I_{i}.
\end{align}
By \eqref{5.13}, the proof of \eqref{5.11} is similar to \cite[Lemma 5]{Huang-Li-Matsumura} for the stability of viscous contact wave for the compressible Navier-Stokes system.  Here the difference   is that we need to estimate the terms involving  $L^{-1}_{M}$ additionally.  For completeness, we will estimate each term in \eqref{5.13}. For the term $I_4$ in \eqref{5.13}, we see
\begin{align}\label{5.13a}
I_{4}&=\big(\widetilde{u}_{1},(\frac{2}{3}\widetilde{\theta}-p_{+}\widetilde{v})v\mathfrak{f}\big)_{t}
-\big(\widetilde{u}_{1},(\frac{2}{3}\widetilde{\theta}-p_{+}\widetilde{v})_{t}v\mathfrak{f}\big)\notag\\
&\qquad-\big(\widetilde{u}_{1},(\frac{2}{3}\widetilde{\theta}-p_{+}\widetilde{v})v_{t}\mathfrak{f}\big)
-\big(\widetilde{u}_{1},(\frac{2}{3}\widetilde{\theta}-p_{+}\widetilde{v})v\mathfrak{f}_{t}\big).
\end{align}
By using \eqref{2.7}$_{1}$ and  \eqref{2.7}$_{4}$, one has
\begin{align}
\label{5.14}
(\widetilde{\theta}+p_{+}\widetilde{v})_{t}=&-(\frac{\frac{2}{3}\widetilde{\theta}-p_{+}\widetilde{v}}{v})u_{1x}
+\big(\frac{\kappa(\theta)}{v}\theta_{x}-\frac{\kappa(\bar{\theta})}{\bar{v}}\bar{\theta}_{x}\big)_{x}
+Q_{1}
\nonumber\\
&\hspace{2cm}+u\cdot\int_{\mathbb{R}^{3}} \xi\xi_{1}(L^{-1}_{M}\Theta_{1})_{x}\,d\xi
-\frac{1}{2}\int_{\mathbb{R}^{3}}\xi_{1}|\xi|^{2}(L^{-1}_{M}\Theta_{1})_{x} \,d\xi.
\end{align}
For the second term on the right hand side of \eqref{5.13a}, by this,  \eqref{2.7}$_{1}$ and the integration by parts,  one has
\begin{align}
\label{5.32A}
&-\big(\widetilde{u}_{1},(\frac{2}{3}\widetilde{\theta}-p_{+}\widetilde{v})_{t}v\mathfrak{f}\big)=
\frac{5}{3}p_{+}\big(\widetilde{u}_{1},\widetilde{v}_{t}v\mathfrak{f}\big)
-\frac{2}{3}\big(\widetilde{u}_{1},(\widetilde{\theta}+p_{+}\widetilde{v})_{t}v\mathfrak{f}\big)
\nonumber\\
&=\frac{5}{3}p_{+}\big(\widetilde{u}_{1},\widetilde{v}_{t}v\mathfrak{f}\big)+\frac{2}{3}\big(\widetilde{u}_{1},(\frac{2}{3}\widetilde{\theta}-p_{+}\widetilde{v})u_{1x}\mathfrak{f}\big)
+\frac{2}{3}\big(\frac{\kappa(\theta)}{v}\theta_{x}-\frac{\kappa(\bar{\theta})}{\bar{v}}\bar{\theta}_{x},(\widetilde{u}_{1}v\mathfrak{f})_{x}\big)
-\frac{2}{3}\big(\widetilde{u}_{1}v\mathfrak{f},Q_{1}\big)\nonumber\\
&\quad+\frac{2}{3}\big(\widetilde{u}_{1}v\mathfrak{f},u_{x}\cdot\int_{\mathbb{R}^{3}} \xi\xi_{1}L^{-1}_{M}\Theta_{1}\, d\xi\big)
-\frac{2}{3}\big((\widetilde{u}_{1}v\mathfrak{f})_{x},\int_{\mathbb{R}^{3}}(\frac{1}{2}\xi_{1}|\xi|^{2}-u\cdot \xi\xi_{1})L^{-1}_{M}\Theta_{1} \,d\xi\big).
\end{align}

For the first term on the right hand side of \eqref{5.32A}, we have from \eqref{2.7}$_{1}$, \eqref{1.27}, \eqref{5.12}, \eqref{2.17} and the integration by parts that
\begin{align*}
\frac{5}{3}p_{+}\big(\widetilde{u}_{1},\widetilde{v}_{t}v\mathfrak{f}\big)&=\frac{5}{3}p_{+}\big(\widetilde{u}_{1},\widetilde{u}_{1x}v\mathfrak{f}\big)
=\frac{5}{6}p_{+}\big((\widetilde{u}_{1}^2)_x, v\mathfrak{f}\big)\\
&= -\frac{5}{6}p_{+}\int_{\mathbb{R}}v\widetilde{u}_{1}^{2}\omega^{2}\,dx
-\frac{5}{6}p_{+}\int_{\mathbb{R}} \widetilde{u}_{1}^{2}(\bar{v}_{x}+\widetilde{v}_{x})\mathfrak{f}\,dx
\\
&\leq  -\frac{5}{6}p_{+}\int_{\mathbb{R}}v\widetilde{u}_{1}^{2}\omega^{2}\,dx
+C\delta\int_{\mathbb{R}}\widetilde{u}_{1}^{2}\omega^{2}\,dx
+C(\|\widetilde{v}_{x}\|^{2}+\|\widetilde{u}_{x}\|^{2})+C(1+t)^{-\frac{3}{2}}.
\end{align*}
By using \eqref{2.9}, \eqref{5.12a} and \eqref{1.29}, one has
\begin{align*}
&\frac{2}{3}\big(\widetilde{u}_{1},(\frac{2}{3}\widetilde{\theta}-p_{+}\widetilde{v})u_{1x}\mathfrak{f}\big)
+\frac{2}{3}\big(\frac{\kappa(\theta)}{v}\theta_{x}-\frac{\kappa(\bar{\theta})}{\bar{v}}\bar{\theta}_{x},(\widetilde{u}_{1}v\mathfrak{f})_{x}\big)
\\
&\leq C(1+t)^{-\frac{1}{2}}\|\widetilde{u}_{1}\|_{L_{x}^{\infty}} (\|\widetilde{v}\|+\|\widetilde{\theta}\| )\times
 (\|\widetilde{u}_{1x}\|+\|\bar{u}_{1x}\| )
\\
&\qquad+C \{\|\widetilde{\theta}_{x}\|+(1+t)^{-\frac{1}{2}}(\|\widetilde{\theta}\|+\|\widetilde{v}\|) \}\times
 \{(1+t)^{-\frac{1}{2}}(\|\widetilde{u}_{1x}\|+\|\widetilde{u}_{1}v_{x}\|)+(1+t)^{-1}\|\widetilde{u}_{1}\| \}
\\
&\leq  C(\|\widetilde{v}_{x}\|^{2}+\|\widetilde{u}_{x}\|^{2}+\|\widetilde{\theta}_{x}\|^{2})+C(1+t)^{-\frac{3}{2}}.
\end{align*}
By the expression of $Q_{1}$ in \eqref{2.8}, one gets from \eqref{5.12a} and  \eqref{1.29} that
\begin{equation*}
-\frac{2}{3}\big(\widetilde{u}_{1}v\mathfrak{f},Q_{1}\big)
\leq C(1+t)^{-\frac{1}{2}}\|\widetilde{u}_{1}\|_{L_{x}^{\infty}}\big(\|\widetilde{u}_{x}\|^{2}+\|\bar{u}_{x}\|^{2}\big)
\leq C\|\widetilde{u}_{x}\|^{2}+C(1+t)^{-\frac{3}{2}}.
\end{equation*}
For the last term in \eqref{5.32A}, we have from \eqref{5.15} that
\begin{multline}
\label{5.17}
-\frac{2}{3}\big((\widetilde{u}_{1}v\mathfrak{f})_{x},\int_{\mathbb{R}^{3}}(\frac{1}{2}\xi_{1}|\xi|^{2}-u\cdot \xi\xi_{1})L^{-1}_{M}\Theta_{1}\,d\xi\big)\\
=-\frac{2}{3}\int_{\mathbb{R}}\Big\{(\widetilde{u}_{1}v\mathfrak{f})_{x}
(R\theta)^{\frac{3}{2}}\int_{\mathbb{R}^{3}}A_{1}(\frac{\xi-u}{\sqrt{R\theta}})\frac{\Theta_{1}}{M}\, d\xi\Big\}\,dx.
\end{multline}
Recall $\Theta_{1}$ in \eqref{1.22} that
\begin{equation*}
\Theta_{1}=G_{t}-\frac{u_{1}}{v}G_{x}+\frac{1}{v}P_{1}(\xi_{1}G_{x})-Q(G,G).
\end{equation*}
Recalling that $G=\overline{G}+\sqrt{\mu}\mathbf{g}$, by using \eqref{5.20}, \eqref{5.18}, \eqref{1.29},
\eqref{5.12} and the imbedding inequality, we have
\begin{align}
\label{5.21}
&-\frac{2}{3}\int_{\mathbb{R}}\Big\{(\widetilde{u}_{1}v\mathfrak{f})_{x}
(R\theta)^{\frac{3}{2}}\int_{\mathbb{R}^{3}}A_{1}(\frac{\xi-u}{\sqrt{R\theta}})\frac{\overline{G}_{t}}{M}\, d\xi\Big\}\,dx
\nonumber\\
&\leq C\{(1+t)^{-\frac{1}{2}}(\|\widetilde{u}_{1x}\|+\|v_{x}\|)+(1+t)^{-1}\|\widetilde{u}_{1}\|\}
\times\big(\int_{\mathbb{R}}\int_{\mathbb{R}^{3}}|\frac{\overline{G}_{t}}{\sqrt{\mu}}|^{2}\,d\xi dx\big)^{\frac{1}{2}}
\nonumber\\
&\leq C\{(1+t)^{-\frac{1}{2}}(\|\widetilde{u}_{x}\|+\|\widetilde{v}_{x}\|
+\|\bar{v}_{x}\|)+(1+t)^{-1}\|\widetilde{u}\|\}
\nonumber\\
&\hspace{2cm}\times \{\|[\bar{u}_{1xt},\bar{\theta}_{xt}]\|+\|[\bar{u}_{1x},\bar{\theta}_{x}]\cdot[v_{t},u_{t},\theta_{t}]\|\}
\nonumber\\
&\leq C\|[\widetilde{v}_{x},\widetilde{u}_{x},\widetilde{\theta}_{x}]\|^{2}
+C\|[\widetilde{v}_{t},\widetilde{u}_{t},\widetilde{\theta}_{t}]\|^{2}+C(1+t)^{-\frac{4}{3}}.
\end{align}
Similarly, it holds that
\begin{align}
\label{5.22}
&-\frac{2}{3}\int_{\mathbb{R}}\Big\{(\widetilde{u}_{1}v\mathfrak{f})_{x}
(R\theta)^{\frac{3}{2}}\int_{\mathbb{R}^{3}}A_{1}(\frac{\xi-u}{\sqrt{R\theta}})\frac{\sqrt{\mu}\mathbf{g}_{t}}{M} d\xi\Big\}\,dx
\nonumber\\
&\leq C\{(1+t)^{-\frac{1}{2}}(\|\widetilde{u}_{1x}\|+\|v_{x}\|)+(1+t)^{-1}\|\widetilde{u}_{1}\|\}
\times\|\langle \xi\rangle^{-\frac{1}{2}}\mathbf{g}_{t}\|
\nonumber\\
&\leq C\|[\widetilde{v}_{x},\widetilde{u}_{x},\widetilde{\theta}_{x}]\|^{2}+C\|\mathbf{g}_{t}\|^{2}_{\sigma}
+C(1+t)^{-\frac{4}{3}}.
\end{align}
It follows from \eqref{5.21} and \eqref{5.22} that
\begin{align}
\label{5.23}
&-\frac{2}{3}\int_{\mathbb{R}}\Big\{(\widetilde{u}_{1}v \mathfrak{{f}})_{x}
(R\theta)^{\frac{3}{2}}\int_{\mathbb{R}^{3}}A_{1}(\frac{\xi-u}{\sqrt{R\theta}})\frac{G_{t}}{M} \,d\xi\Big\}\,dx
\nonumber\\
&\leq C\|[\widetilde{v}_{x},\widetilde{u}_{x},\widetilde{\theta}_{x}]\|^{2}
+C\|[\widetilde{v}_{t},\widetilde{u}_{t},\widetilde{\theta}_{t}]\|^{2}+C\|\mathbf{g}_{t}\|^{2}_{\sigma}+C(1+t)^{-\frac{4}{3}}.
\end{align}
Using the similar calculations as \eqref{5.21}, \eqref{5.22} and \eqref{5.23}, we can obtain
\begin{align*}
&-\frac{2}{3}\int_{\mathbb{R}}\Big\{(\widetilde{u}_{1}v \mathfrak{{f}})_{x}
(R\theta)^{\frac{3}{2}}\int_{\mathbb{R}^{3}}A_{1}(\frac{\xi-u}{\sqrt{R\theta}})
\{-\frac{u_{1}}{v}G_{x}+\frac{1}{v}P_{1}(\xi_{1}G_{x})\}\frac{1}{M}\,d\xi\Big\}\,dx
\nonumber\\
&\leq C\|[\widetilde{v}_{x},\widetilde{u}_{x},\widetilde{\theta}_{x}]\|^{2}+C\|\mathbf{g}_{x}\|^{2}_{\sigma}+C(1+t)^{-\frac{4}{3}}.
\end{align*}
By \eqref{5.7}, \eqref{5.18} and the similar calculations as \eqref{2.22}, we have
\begin{align*}
&-\frac{2}{3}\int_{\mathbb{R}}\Big\{(\widetilde{u}_{1}v \mathfrak{{f}})_{x}
(R\theta)^{\frac{3}{2}}\int_{\mathbb{R}^{3}}A_{1}(\frac{\xi-u}{\sqrt{R\theta}})\frac{Q(G,G)}{M}\, d\xi\Big\}\,dx
\nonumber\\
&=-\frac{2}{3}\int_{\mathbb{R}}\Big\{(\widetilde{u}_{1}v \mathfrak{{f}})_{x}
(R\theta)^{\frac{3}{2}}\int_{\mathbb{R}^{3}}\frac{\sqrt{\mu}A_{1}(\frac{\xi-u}{\sqrt{R\theta}})}{M}\Gamma(\frac{G}{\sqrt{\mu}},\frac{G}{\sqrt{\mu}}) \,d\xi\Big\}\,dx
\nonumber\\
&\leq C\|[\widetilde{v}_{x},\widetilde{u}_{x},\widetilde{\theta}_{x}]\|^{2}+C\|\mathbf{g}_{x}\|^{2}_{\sigma}
+C\varepsilon_{0}\|\mathbf{g}\|^{2}_{\sigma}+C(1+t)^{-\frac{4}{3}}.
\end{align*}
By using \eqref{5.17} and the above estimates, we arrive at 
\begin{align}
\label{5.25}
&-\frac{2}{3}\big((\widetilde{u}_{1}v \mathfrak{{f}})_{x},\int_{\mathbb{R}^{3}}(\frac{1}{2}\xi_{1}|\xi|^{2}-u\cdot \xi\xi_{1})L^{-1}_{M}\Theta_{1}\, d\xi\big)
\nonumber\\
&\leq C\sum_{|\alpha|=1}(\|\partial^{\alpha}[\widetilde{v},\widetilde{u},\widetilde{\theta}]\|^{2}
+\|\partial^{\alpha}\mathbf{g}\|^{2}_{\sigma})+C\varepsilon_{0}\|\mathbf{g}\|^{2}_{\sigma}+C(1+t)^{-\frac{4}{3}}.
\end{align}
Similar arguments as \eqref{5.25} imply
\begin{align*}
&\frac{2}{3}\big(\widetilde{u}_{1}v\mathfrak{{f}},u_{x}\cdot\int_{\mathbb{R}^{3}} \xi\xi_{1}L^{-1}_{M}\Theta_{1} \,d\xi\big)
=\sum^{3}_{i=1}\frac{2}{3}\big(\widetilde{u}_{1}v\mathfrak{{f}},u_{ix} R\theta\int_{\mathbb{R}^{3}}B_{1i}(\frac{\xi-u}{\sqrt{R\theta}})\frac{\Theta_{1}}{M} \,d\xi\big)
\nonumber\\
&\leq C\sum_{|\alpha|=1}(\|\partial^{\alpha}[\widetilde{v},\widetilde{u},\widetilde{\theta}]\|^{2}
+\|\partial^{\alpha}\mathbf{g}\|^{2}_{\sigma})+C\varepsilon_{0}\|\mathbf{g}\|^{2}_{\sigma}+C(1+t)^{-\frac{4}{3}}.
\end{align*}
For the second term on the right hand side of \eqref{5.13a},  by choosing a small $\delta>0$,  we deduce from \eqref{5.32A} and the above estimates that
\begin{align}
\label{5.40}
-\big(\widetilde{u}_{1},(\frac{2}{3}\widetilde{\theta} -p_{+}\widetilde{v})_{t}v\mathfrak{f}\big)
\leq& -\frac{5}{12}p_{+}\int_{\mathbb{R}}v\widetilde{u}_{1}^{2}\omega^{2}\,dx+C\varepsilon_{0}\|\mathbf{g}\|^{2}_{\sigma}+C(1+t)^{-\frac{4}{3}}
\nonumber\\
&\hspace{2cm}+ C\sum_{|\alpha|=1}(\|\partial^{\alpha}[\widetilde{v},\widetilde{u},\widetilde{\theta}]\|^{2}
+\|\partial^{\alpha}\mathbf{g}\|^{2}_{\sigma}).
\end{align}
For the last two terms on the right hand side of \eqref{5.13a}, by using \eqref{5.12a}, \eqref{1.19}$_{1}$, \eqref{1.29},
\eqref{2.9} and the  imbedding inequality, one has
\begin{align}
\label{5.41}
&-\big(\widetilde{u}_{1},(\frac{2}{3}\widetilde{\theta}-p_{+}\widetilde{v})v_{t}\mathfrak{f}\big)
-\big(\widetilde{u}_{1},(\frac{2}{3}\widetilde{\theta}-p_{+}\widetilde{v})v\mathfrak{f}_{t}\big)
\nonumber\\
&\leq C\|\mathfrak{f}\|_{L_{x}^{\infty}}\|\widetilde{u}_{1}\|_{L_{x}^{\infty}}(\|\widetilde{v}\|+\|\widetilde{\theta}\|)\|v_{t}\|
+C\|\mathfrak{f}_{t}\|_{L_{x}^{\infty}}
\|\widetilde{u}_{1}\|(\|\widetilde{v}\|+\|\widetilde{\theta}\|)
\nonumber\\
&\leq C(1+t)^{-\frac{1}{2}}\|\widetilde{u}_{1}\|^{\frac{1}{2}}\|\widetilde{u}_{1x}\|^{\frac{1}{2}}(\|\widetilde{v}\|
+\|\widetilde{\theta}\|)\|u_{1x}\|+C(1+t)^{-\frac{3}{2}}
\nonumber\\
&\leq C\|\widetilde{u}_{1x}\|^{2}+C(1+t)^{-\frac{3}{2}}.
\end{align}
It follows from \eqref{5.13a}, \eqref{5.40}  and \eqref{5.41} that
\begin{align}\label{5.43}
I_{4}&\leq \big(\widetilde{u}_{1},(\frac{2}{3}\widetilde{\theta}-p_{+}\widetilde{v})v\mathfrak{f}\big)_{t}
-\frac{5}{12}p_{+}\int_{\mathbb{R}}v\widetilde{u}_{1}^{2}\omega^{2}\,dx+C\varepsilon_{0}\|\mathbf{g}\|^{2}_{\sigma}
\nonumber\\
&\qquad+C\sum_{|\alpha|=1}(\|\partial^{\alpha}[\widetilde{v},\widetilde{u},\widetilde{\theta}]\|^{2}
+\|\partial^{\alpha}\mathbf{g}\|^{2}_{\sigma})+C(1+t)^{-\frac{4}{3}}.
\end{align}
By \eqref{5.13}, \eqref{5.10}, \eqref{5.12a}, \eqref{1.27}, \eqref{2.9} and the  imbedding inequality, we get
\begin{align}\label{5.44}
|I_{5}|&\leq   |\big(\frac{1}{v}(\frac{2}{3}\widetilde{\theta}-p_{+}\widetilde{v})^{2},\bar{v}_{x}\mathfrak{f}\big)|
+|\big(\frac{1}{v}(\frac{2}{3}\widetilde{\theta}-p_{+}\widetilde{v})^{2},\widetilde{v}_{x}\mathfrak{f}\big)|
\nonumber\\
&\leq C\delta\int_{\mathbb{R}}(\widetilde{v}^{2}+\widetilde{\theta}^{2})\omega^{2}\,dx+C(\|\widetilde{v}_{x}\|^{2}
+\|\widetilde{\theta}_{x}\|^{2})+C(1+t)^{-\frac{4}{3}}.
\end{align}
By \eqref{5.13}, \eqref{5.12a} and \eqref{1.29}, one has
\begin{align}\label{5.45a}
|I_{6}|+|I_{7}|\leq C(\|\widetilde{v}_{x}\|^{2}+\|\widetilde{u}_{x}\|^{2}+\|\widetilde{\theta}_{x}\|^{2})+C(1+t)^{-\frac{4}{3}}.
\end{align}
By using \eqref{5.16}, \eqref{5.18} and the similar arguments as \eqref{5.25},  we can obtain
\begin{align}\label{5.45}
|I_{8}|&= |\int_{\mathbb{R}}\Big\{[(\frac{2}{3}\widetilde{\theta}-p_{+}\widetilde{v})v\mathfrak{f}]_{x}
R\theta\int_{\mathbb{R}^{3}}B_{11}(\frac{\xi-u}{\sqrt{R\theta}})\frac{\Theta_{1}}{M}\, d\xi\Big\}\,dx|
\nonumber\\
&\leq C\sum_{|\alpha|=1}(\|\partial^{\alpha}[\widetilde{v},\widetilde{u},\widetilde{\theta}]\|^{2}
+\|\partial^{\alpha}\mathbf{g}\|^{2}_{\sigma})+C\varepsilon_{0}\|\mathbf{g}\|^{2}_{\sigma}+C(1+t)^{-\frac{4}{3}}.
\end{align}
By using \eqref{5.13},   \eqref{5.43}, \eqref{5.44}, \eqref{5.45a} and \eqref{5.45}, we arrive at
\begin{align*}
\int_{\mathbb{R}}\big\{(\frac{2}{3}\widetilde{\theta}-p_{+}\widetilde{v})^{2}+\widetilde{u}_{1}^{2}\big\}\omega^{2}\,dx
&\leq C\big(\widetilde{u}_{1},(\frac{2}{3}\widetilde{\theta}-p_{+}\widetilde{v})v\mathfrak{f}\big)_{t}
+C\delta\int_{\mathbb{R}}(\widetilde{v}^{2}+\widetilde{\theta}^{2})\omega^{2}\,dx
+C\varepsilon_{0}\|\mathbf{g}\|^{2}_{\sigma}
\nonumber\\
&\qquad+C\sum_{|\alpha|=1}\big\{\|\partial^{\alpha}[\widetilde{v},\widetilde{u},\widetilde{\theta}]\|^{2}
+\|\partial^{\alpha}\mathbf{g}\|^{2}_{\sigma}\big\}+C(1+t)^{-\frac{4}{3}}.
\end{align*}
Integrating it over $(0,t)$,  we have from this and \eqref{2.9} that
\begin{align}
\label{5.26}
\int^{t}_{0}\int_{\mathbb{R}}\big\{(\frac{2}{3}\widetilde{\theta}-p_{+}\widetilde{v})^{2}+\widetilde{u}_{1}^{2}\big\}\omega^{2}\,dxds
&\leq C+C\delta\int^{t}_{0}\int_{\mathbb{R}}(\widetilde{v}^{2}+\widetilde{\theta}^{2})\omega^{2}\,dxds+C\varepsilon_{0}\int^{t}_{0}\|\mathbf{g}\|^{2}_{\sigma}\,ds
\nonumber\\
&\quad+C\sum_{|\alpha|=1}\int^{t}_{0}\big\{\|\partial^{\alpha}[\widetilde{v},\widetilde{u},\widetilde{\theta}]\|^{2}
+\|\partial^{\alpha}\mathbf{g}\|^{2}_{\sigma}\big\}\,ds.
\end{align}

On the other hand, we choose $h=\frac{2}{3}\widetilde{\theta}+\frac{2}{3}p_{+}\widetilde{v}$ in Lemma \ref{lem5.5} and use \eqref{5.14} to deduce
\begin{align}
\label{5.41a}
( h_{t},h \mathfrak{g}^{2})
&=-\frac{2}{3}\big(\frac{\frac{2}{3}\widetilde{\theta}-p_{+}\widetilde{v}}{v}(\widetilde{u}_{1x}+\bar{u}_{1x}),h\mathfrak{g}^{2}\big)
+\frac{2}{3}\big([\frac{\kappa(\theta)}{v}\theta_{x}-\frac{\kappa(\bar{\theta})}{\bar{v}}\bar{\theta}_{x}]_{x},h\mathfrak{g}^{2}\big)
+\frac{2}{3}(Q_{1},h\mathfrak{g}^{2})
\nonumber\\
&+\frac{2}{3}(\int_{\mathbb{R}^{3}}(\frac{1}{2}\xi_{1}|\xi|^{2}-u\cdot\xi\xi_{1})L^{-1}_{M}\Theta_{1} d\xi,(h\mathfrak{g}^{2})_{x})
-\frac{2}{3}(u_{x}\int_{\mathbb{R}^{3}} \xi\xi_{1}L^{-1}_{M}\Theta_{1} d\xi,h\mathfrak{g}^{2}).
\end{align}
By using the facts that $\frac{2}{3}\widetilde{\theta}-p_{+}\widetilde{v}=h-\frac{5}{3}p_{+}\widetilde{v}$ and $\widetilde{u}_{1x}=\widetilde{v}_{t}$, we have
\begin{align}
\label{5.42aa}
-&\frac{2}{3}\big(\frac{\frac{2}{3}\widetilde{\theta}-p_{+}\widetilde{v}}{v}\widetilde{u}_{1x},h\mathfrak{g}^{2}\big)
\nonumber\\
=&-\frac{2}{3}\int_{\mathbb{R}}v^{-1}(h^{2}-\frac{5}{3}p_{+}\widetilde{v}h)\widetilde{v}_{t}\mathfrak{g}^{2}\,dx
=-\frac{1}{3}\int_{\mathbb{R}}\Big(2v^{-1}h^{2}\mathfrak{g}^{2}\widetilde{v}_{t}
-\frac{5}{3}p_{+}v^{-1}h\mathfrak{g}^{2}(\widetilde{v}^{2})_{t}\Big)\,dx
\nonumber\\
=&-\frac{1}{3}\big(\int_{\mathbb{R}}v^{-1}h\mathfrak{g}^{2}\widetilde{v}(2h-\frac{5}{3}p_{+}\widetilde{v})dx\big)_{t}
+\frac{2}{3}\int_{\mathbb{R}}v^{-1}h\mathfrak{g}\widetilde{v}(2h-\frac{5}{3}p_{+}\widetilde{v})\mathfrak{g}_{t}\,dx
\nonumber\\
&-\frac{1}{3}\int_{\mathbb{R}}v^{-2}v_{t}h\mathfrak{g}^{2}\widetilde{v}(2h-\frac{5}{3}p_{+}\widetilde{v})\,dx
+\frac{1}{3}\int_{\mathbb{R}}v^{-1}\mathfrak{g}^{2}\widetilde{v}(4h-\frac{5}{3}p_{+}\widetilde{v})h_{t}\,dx
:=\sum^{12}_{i=9}I_{i}.
\end{align}
We only estimate the last three terms in \eqref{5.42aa}. In view of \eqref{5.10}, it is easy to check that
\begin{align}
\label{5.43aa}
4\lambda\mathfrak{g}_{t}=\omega_{x}, \quad \|\mathfrak{g}(t,x)\|_{L_{x}^{\infty}}=\sqrt{\pi}\lambda^{-\frac{1}{2}}.
\end{align}
By using this and the facts that $h=\frac{2}{3}(\widetilde{\theta}+p_{+}\widetilde{v})$, $|\omega_{x}|\leq C(1+t)^{-1}$ and $v_{t}=u_{1x}$, one has
\begin{align*}
|I_{10}|+|I_{11}|&\leq C(1+t)^{-1}\int_{\mathbb{R}}(|\widetilde{v}|^{3}+|\widetilde{\theta}|^{3})\,dx
+C\int_{\mathbb{R}}|v_{t}|(|\widetilde{v}|^{3}+|\widetilde{\theta}|^{3})\,dx
\\
&\leq C\|[\widetilde{v}_{x},\widetilde{u}_{x},\widetilde{\theta}_{x}]\|^{2}+C(1+t)^{-\frac{4}{3}}.
\end{align*}
By \eqref{5.14} and  $h=\frac{2}{3}(\widetilde{\theta}+p_{+}\widetilde{v})$, we have
\begin{align*}
I_{12}&=\frac{2}{9}\Big\{\int_{\mathbb{R}}v^{-1}\mathfrak{g}^{2}\widetilde{v}(4h-\frac{5}{3}p_{+}\widetilde{v})
\Big(-(\frac{\frac{2}{3}\widetilde{\theta}-p_{+}\widetilde{v}}{v})u_{1x}
+\big(\frac{\kappa(\theta)}{v}\theta_{x}-\frac{\kappa(\bar{\theta})}{\bar{v}}\bar{\theta}_{x}\big)_{x}
+Q_{1} \Big)\,dx\Big\}
\nonumber\\
&\quad+\frac{2}{9}\Big\{\int_{\mathbb{R}}v^{-1}\mathfrak{g}^{2}\widetilde{v}(4h-\frac{5}{3}p_{+}\widetilde{v})
\Big(u\cdot\int_{\mathbb{R}^{3}} \xi\xi_{1}(L^{-1}_{M}\Theta_{1})_{x}\,d\xi
-\frac{1}{2}\int_{\mathbb{R}^{3}}\xi_{1}|\xi|^{2}(L^{-1}_{M}\Theta_{1})_{x} d\xi \Big)dx\Big\}
\nonumber\\
&:=I^{1}_{12}+I^{2}_{12}.
\end{align*}
By using \eqref{5.43aa}, \eqref{1.29}, \eqref{2.9}  and the expression of $Q_{1}$ in \eqref{2.8}, one has
$$
|I^{1}_{12}|\leq C\|[\widetilde{v}_{x},\widetilde{u}_{x},\widetilde{\theta}_{x}]\|^{2}
+C(1+t)^{-\frac{4}{3}}.
$$
Similar arguments as \eqref{5.25} imply
\begin{align*}
|I^{2}_{12}|\leq C\sum_{|\alpha|=1}(\|\partial^{\alpha}[\widetilde{v},\widetilde{u},\widetilde{\theta}]\|^{2}
+\|\partial^{\alpha}\mathbf{g}\|^{2}_{\sigma})+C\varepsilon_{0}\|\mathbf{g}\|^{2}_{\sigma}
+C(1+t)^{-\frac{4}{3}}.
\end{align*}
It follows from the above two estimates that
\begin{align*}
|I_{12}|\leq C\sum_{|\alpha|=1}(\|\partial^{\alpha}[\widetilde{v},\widetilde{u},\widetilde{\theta}]\|^{2}
+\|\partial^{\alpha}\mathbf{g}\|^{2}_{\sigma})+C\varepsilon_{0}\|\mathbf{g}\|^{2}_{\sigma}
+C(1+t)^{-\frac{4}{3}}.
\end{align*}
By using \eqref{5.42aa} and the above estimates, we can obtain
\begin{align}\label{5.43b}
&-\frac{2}{3}\big(\frac{\frac{2}{3}\widetilde{\theta}-p_{+}\widetilde{v}}{v}\widetilde{u}_{1x},h\mathfrak{g}^{2}\big)
+\frac{1}{3}\big(\int_{\mathbb{R}}v^{-1}h \mathfrak{g}^{2}\widetilde{v}(2h-\frac{5}{3}p_{+}\widetilde{v})dx\big)_{t}
\nonumber\\
&\leq C\sum_{|\alpha|=1}(\|\partial^{\alpha}[\widetilde{v},\widetilde{u},\widetilde{\theta}]\|^{2}
+\|\partial^{\alpha}\mathbf{g}\|^{2}_{\sigma})+C\varepsilon_{0}\|\mathbf{g}\|^{2}_{\sigma}
+C(1+t)^{-\frac{4}{3}}.
\end{align}
By \eqref{5.41a}, \eqref{1.27}, \eqref{5.43aa} and \eqref{2.17}, we have
\begin{align}\label{5.43c}
&-\frac{2}{3}\big(\frac{\frac{2}{3}\widetilde{\theta}-p_{+}\widetilde{v}}{v}\bar{u}_{1x},h\mathfrak{g}^{2}\big)
+\frac{2}{3}\big((\frac{\kappa(\theta)}{v}\theta_{x}-\frac{\kappa(\bar{\theta})}{\bar{v}}\bar{\theta}_{x})_{x},h\mathfrak{g}^{2}\big)
+\frac{2}{3}(Q_{1},h\mathfrak{g}^{2})
\nonumber\\
&\leq C(\epsilon+\delta)\int_{\mathbb{R}}(\widetilde{v}^{2}+\widetilde{\theta}^{2})\omega^{2}\,dx
+C_{\epsilon}\|[\widetilde{v}_{x},\widetilde{u}_{x},\widetilde{\theta}_{x}]\|^{2}
+C_{\epsilon}(1+t)^{-\frac{4}{3}}.
\end{align}
For the last two terms of \eqref{5.41a}, by using \eqref{5.15}, \eqref{5.16} and  the similar arguments as \eqref{5.25}, one has
\begin{align}\label{5.43d}
&\big(\int_{\mathbb{R}^{3}}(\frac{1}{2}\xi_{1}|\xi|^{2}-u\cdot\xi\xi_{1})L^{-1}_{M}\Theta_{1} d\xi,(h\mathfrak{g}^{2})_{x}\big)
-\big(u_{x}\cdot\int_{\mathbb{R}^{3}} \xi\xi_{1}L^{-1}_{M}\Theta_{1} d\xi,h\mathfrak{g}^{2}\big)
\nonumber\\
&=\int_{\mathbb{R}}(R\theta)^{\frac{3}{2}}\int_{\mathbb{R}^{3}}A_{1}(\frac{\xi-u}{\sqrt{R\theta}})\frac{\Theta_{1}}{M} d\xi(h\mathfrak{g}^{2})_{x}\,dx
-\sum^{3}_{i=1}\int_{\mathbb{R}}R\theta\int_{\mathbb{R}^{3}}B_{1i}(\frac{\xi-u}{\sqrt{R\theta}})\frac{\Theta_{1}}{M} \,d\xi
u_{ix}h\mathfrak{g}^{2}\,dx
\nonumber\\
&\leq C\epsilon\int_{\mathbb{R}}(\widetilde{v}^{2}+\widetilde{\theta}^{2})\omega^{2}\,dx+C_{\epsilon}\varepsilon_{0}\|\mathbf{g}\|^{2}_{\sigma}
+C_{\epsilon}\sum_{|\alpha|=1}(\|\partial^{\alpha}[\widetilde{v},\widetilde{u},\widetilde{\theta}]\|^{2}
+\|\partial^{\alpha}\mathbf{g}\|^{2}_{\sigma})+C_{\epsilon}(1+t)^{-\frac{4}{3}}.
\end{align}
 By using \eqref{5.41a}, \eqref{5.43b}, \eqref{5.43c} and \eqref{5.43d}, we arrive at
\begin{align*}
( h_{t},h \mathfrak{g}^{2})&\leq -\frac{1}{3}\big(\int_{\mathbb{R}}v^{-1}h\mathfrak{g}^{2}\widetilde{v}(2h-\frac{5}{3}p_{+}\widetilde{v})\,dx\big)_{t}
+C(\epsilon+\delta)\int_{\mathbb{R}}(\widetilde{v}^{2}+\widetilde{\theta}^{2})\omega^{2}\,dx
\nonumber\\
&\hspace{2cm}+C_{\epsilon}\varepsilon_{0}\|\mathbf{g}\|^{2}_{\sigma}
+C_{\epsilon}\sum_{|\alpha|=1}(\|\partial^{\alpha}[\widetilde{v},\widetilde{u},\widetilde{\theta}]\|^{2}
+\|\partial^{\alpha}\mathbf{g}\|^{2}_{\sigma})+C_{\epsilon}(1+t)^{-\frac{4}{3}}.
\end{align*}
Recalling that $h=\frac{2}{3}\widetilde{\theta}+\frac{2}{3}p_{+}\widetilde{v}$, by using this, \eqref{2.9} and Lemma \ref{lem5.5}, we have
\begin{align}
\label{5.27}
\int^{t}_{0}\int_{\mathbb{R}} (\frac{2}{3}\widetilde{\theta}+\frac{2}{3}p_{+}\widetilde{v})^{2}\omega^{2}\,dxds&\leq C_\epsilon
+C(\epsilon+\delta)\int^{t}_{0}\int_{\mathbb{R}}(\widetilde{v}^{2}+\widetilde{\theta}^{2})\omega^{2}\,dxds
+C_{\epsilon}\varepsilon_{0}\int^{t}_{0}\|\mathbf{g}\|^{2}_{\sigma}\,ds
\nonumber\\
&\quad+C_{\epsilon}\sum_{|\alpha|=1}\int^{t}_{0}(\|\partial^{\alpha}[\widetilde{v},\widetilde{u},\widetilde{\theta}]\|^{2}
+\|\partial^{\alpha}\mathbf{g}\|^{2}_{\sigma})\,ds.
\end{align}
Similarly, we take $h=\widetilde{u}_{i}$ $(i=2,3)$ in Lemma \ref{lem5.5} and use \eqref{2.7}$_{3}$ to deduce
\begin{multline*}
( h_{t},h\mathfrak{g}^{2})
=-\big(\frac{\mu(\theta)}{v}\widetilde{u}_{ix},(h\mathfrak{g}^{2})_{x}\big)
+\big(\int_{\mathbb{R}^{3}}\xi_{i}\xi_{1}L^{-1}_{M}\Theta_{1} d\xi,(h\mathfrak{g}^{2})_{x}\big)
\\
\leq C\epsilon\int_{\mathbb{R}}(\widetilde{v}^{2}+\widetilde{\theta}^{2})\omega^{2}\,dx
+C_{\epsilon}\varepsilon_{0}\|\mathbf{g}\|^{2}_{\sigma}
+C_{\epsilon}\sum_{|\alpha|=1}(\|\partial^{\alpha}[\widetilde{v},\widetilde{u},\widetilde{\theta}]\|^{2}
+\|\partial^{\alpha}\mathbf{g}\|^{2}_{\sigma})+C_{\epsilon}(1+t)^{-\frac{4}{3}}.
\end{multline*}
It follows from this, \eqref{2.9} and Lemma \ref{lem5.5} that
\begin{align}
\label{5.28}
\sum^{3}_{i=2}\int^{t}_{0}\int_{\mathbb{R}}\widetilde{u}_{i}^{2}\omega^{2}\,dxds&\leq C_\epsilon
+C\epsilon\int^{t}_{0}\int_{\mathbb{R}}(\widetilde{v}^{2}+\widetilde{\theta}^{2})\omega^{2}\,dxds
+C_{\epsilon}\varepsilon_{0}\int^{t}_{0}\|\mathbf{g}\|^{2}_{\sigma}\,ds
\nonumber\\
&\quad+C_{\epsilon}\sum_{|\alpha|=1}\int^{t}_{0}(\|\partial^{\alpha}[\widetilde{v},\widetilde{u},\widetilde{\theta}]\|^{2}
+\|\partial^{\alpha}\mathbf{g}\|^{2}_{\sigma})\,ds.
\end{align}
Therefore, the estimate \eqref{5.11} follows from \eqref{5.26}, \eqref{5.27} and \eqref{5.28}
by   choosing both $\epsilon$ and   $\delta$ small enough.
This completes the proof of Lemma \ref{lem5.6}.
\end{proof}

\noindent {\bf Acknowledgment:}\,
The research of Renjun Duan was partially supported by the General Research Fund (Project No.~14301719) from RGC of Hong Kong and a Direct Grant from CUHK. The research of Hongjun Yu was supported by the GDUPS 2017 and the NNSFC Grant 11371151. Dongcheng Yang would like to thank Department of Mathematics, CUHK  for hosting his visit in the year 2020-2021.

\medskip

\noindent{\bf Conflict of Interest:} The authors declare that they have no conflict of interest.


\normalsize
\end{document}